\documentclass[11pt]{article}
\usepackage{amsmath,amssymb,amsthm,mathtools,bm}
\usepackage{microtype,booktabs,tabularx,array,enumitem,graphicx,xcolor}
\usepackage[round,authoryear]{natbib}
\usepackage{paper-identity}
\usepackage[nameinlink,capitalize,noabbrev]{cleveref}
\usepackage{needspace}
\usepackage{placeins}
\usepackage{caption,longtable}
\usepackage[most]{tcolorbox}
\definecolor{ReviewOld}{HTML}{B42318}
\definecolor{ReviewNew}{HTML}{1558B0}
\definecolor{TheoremShade}{HTML}{F8FAFC}
\definecolor{TableInk}{HTML}{233E58}
\definecolor{CitationColor}{HTML}{9B3A3A}

\tcbset{resultbox/.style={enhanced,breakable,colback=TheoremShade,
 colframe=blue!25!black,boxrule=0pt,borderline west={1.5pt}{0pt}{blue!40!black},
 sharp corners,left=8pt,right=8pt,top=5pt,bottom=5pt,before skip=9pt,after skip=9pt}}
\tcolorboxenvironment{theorem}{resultbox}
\tcolorboxenvironment{lemma}{resultbox}
\tcolorboxenvironment{proposition}{resultbox}
\tcolorboxenvironment{corollary}{resultbox}
\hypersetup{citecolor=CitationColor,hypertexnames=false,pdftitle={A Sharper Theory of Ball-Proximal Optimization: Convergence and Radius Selection},pdfauthor={Peter Richtarik and Hanmin Li},pdfkeywords={convex optimization, ball-proximal point method, convergence bounds, radius selection}}
\newtheorem{theorem}{Theorem}[section]
\newtheorem{lemma}[theorem]{Lemma}
\newtheorem{proposition}[theorem]{Proposition}
\newtheorem{corollary}[theorem]{Corollary}
\theoremstyle{definition}
\newtheorem{assumption}{Assumption}
\crefname{assumption}{assumption}{assumptions}
\Crefname{assumption}{Assumption}{Assumptions}
\newcommand{\assref}[1]{\Cref{#1} (\nameref*{#1})}
\newtheorem{definition}[theorem]{Definition}
\newtheorem{example}[theorem]{Example}
\theoremstyle{remark}
\newtheorem{remark}[theorem]{Remark}
\newcommand{\R}{\mathbb R}

\newcommand{\norm}[1]{\left\|#1\right\|}
\newcommand{\ip}[2]{\left\langle #1,#2\right\rangle}
\newcommand{\B}{\mathbb{B}}
\newcommand{\X}{X^\star}
\newcommand{\fs}{f_\star}
\newcommand{\BPM}{\textnormal{\textsc{BPM}}}
\newcommand{\brox}{\operatorname{brox}}
\newcommand{\prox}{\operatorname{prox}}
\newcommand{\dist}{\operatorname{dist}}
\newcommand{\dom}{\operatorname{dom}}
\newcommand{\ri}{\operatorname{ri}}
\newcommand{\argmin}{\operatorname*{arg\,min}}
\newcommand{\Proj}{\operatorname{Proj}}

\newcommand{\PaperDate}{28 September 2026}

\hypersetup{pdfsubject={Convex convergence theory and radius selection for the ball-proximal point method}}

\setlist{topsep=4pt,itemsep=3pt}
\title{A Sharper Theory of Ball-Proximal Optimization: Convergence and Radius Selection}
\author{Peter Richt\'arik\qquad Hanmin Li\\[4pt]
\normalsize King Abdullah University of Science and Technology\\
\normalsize Thuwal 23955-6900, Saudi Arabia}
\date{\PaperDate}

\begin{document}
\maketitle

\begin{abstract}\phantomsection\label{concept:abstract}
We study the exact Euclidean ball-proximal point method for proper, closed, convex functions, where each iteration minimizes the objective over a ball centered at the current point. Retaining the objective gap in the decrease of squared distance yields sharper bounds on objective values, stationarity, and the symmetric Bregman distance to a minimizer. For constant radius $t>0$ and initial distance $D_0>0$ to the solution set, the objective gap after $K$ iterations is at most its initial value multiplied by $\exp(-2K^2t^2/D_0^2)$.

We characterize convergence for arbitrary positive radius sequences. If their sum diverges, the method reaches a minimizer in finitely many iterations whenever one exists; otherwise, the objective values converge to the infimum and the iterates escape every bounded set. If the radii are summable, the iterates converge to a possibly nonoptimal point. Self-contraction gives finite trajectory length for bounded iterates and a constant-radius termination bound $O_d(1+D_0/t)$, whose implicit constant depends only on the dimension. A polyhedral family shows that the dimension-independent quadratic bound remains asymptotically sharp.

We also identify a minimum successful geometric decay factor and analyze adaptive radius rules based on subgradients or objective gaps, epigraph reformulation, and relaxed updates. Together, these results strengthen the foundations and convergence guarantees of the method without assuming smoothness or strong convexity.
\end{abstract}

\vspace{3pt}
\noindent\textbf{Keywords:} ball-proximal point method; convex optimization; finite termination; self-contracted sequences; radius selection; oracle complexity.

\noindent\textbf{Mathematics Subject Classification (2020):} Primary 90C25; Secondary 90C60, 65K05.

\section{Introduction}
\label{sec:intro}
We study the convex optimization problem
\begin{equation}\label{eq:problem-intro}
\min_{x\in\R^d} f(x),
\end{equation}
where $f:\R^d\to\R\cup\{+\infty\}$ is proper, closed, and convex.
Unless stated otherwise, we assume neither that the infimum is attained nor that $f$ is smooth.
We focus on the exact ball-oracle iteration considered by \citet[Appendix A]{CarmonEtAlBallOracle} and subsequently developed as the ball-proximal (broximal) point framework by \citet{BPM}.
Starting from $x^0\in\dom f$, the ball-proximal point method (\BPM{}) minimizes $f$ over a Euclidean ball centered at the current point:
\begin{equation}\label{eq:bpm-intro}
x^{k+1}\in\brox_f^{t_k}(x^k)
:=\argmin_{u\in\B(x^k,t_k)}f(u),\qquad t_k>0, \tag{\BPM{}}
\end{equation}
where $\B(x,t)=\{u:\norm{u-x}\le t\}$.
\phantomsection\label{concept:oracle}
A \emph{call} to the exact ball oracle takes a center $x\in\dom f$ and a radius $t>0$ and returns a minimizer $u\in\brox_f^t(x)$ of $f$ over $\B(x,t)$; each BPM step consists of one such call.
\phantomsection\label{concept:terminal}
A step is \emph{terminal} if its output is a global minimizer of $f$, and \emph{nonterminal} otherwise. Equivalently, a step is terminal exactly when its ball contains a global minimizer. If the infimum is not attained, every step is nonterminal.
Throughout the paper, we measure complexity by the number of oracle calls; each BPM iteration uses one such call.

Our analysis builds on two geometric properties of an exact step.
First, retaining the objective gap in the decrease of squared distance yields sharper multi-step bounds on objective suboptimality.
These bounds also give guarantees on the minimum subgradient norm and a symmetric Bregman distance to a minimizer, see \Cref{sec:values,sec:stationarity,sec:primal-dual}.
Second, the geometry of an exact step relative to future iterates implies that bounded trajectories are self-contracted and have finite length in fixed dimension.
This leads to the convergence characterization in \Cref{sec:termination}: nonsummable radii give finite termination when a minimizer exists and convergence of the objective values to the infimum otherwise.
The same trajectory geometry yields a linear constant-radius step bound in fixed dimension, while a polyhedral construction shows that the dimension-independent quadratic bound remains asymptotically sharp.

We give a self-contained treatment intended to serve as a comprehensive reference for exact Euclidean convex \BPM{}, rebuilding the inherited foundations and distinguishing them from the stronger guarantees established here. We then study how reformulation and radius choice affect the method, including epigraphical lifting, adaptive radius rules, and relaxed updates.
We begin with the assumptions and notation, followed by seven equivalent formulations of an exact step and a comparison with prior work.

\paragraph{Standing assumptions.}

The assumptions below are invoked only when needed. In particular, \assref{ass:attained} is not required for results that allow the infimum to be unattained.

\begin{assumption}[Convex objective]\label{ass:convex}
$f:\R^d\to\R\cup\{+\infty\}$ is proper, closed, and convex, with $d<\infty$.
\end{assumption}

\begin{assumption}[Attainment]\label{ass:attained}
The solution set $\X=\argmin f$ is nonempty, and we write $\fs=\min f$.
\end{assumption}

\begin{assumption}[Exact BPM iteration]\label{ass:iteration}
The starting point satisfies $x^0\in\dom f$, every radius is positive, and each update is an exact minimizer in \ref*{eq:bpm-intro}. Once a minimizer is reached, the sequence is held fixed. If $x^0\in\X$, no oracle call is made.
\end{assumption}

\phantomsection\label{concept:finite}
We say that BPM converges finitely if $x^N\in\X$ for some finite $N$; under the stopping convention above, this is finite termination. This property does not require that a point-only oracle provide a test for detecting optimality.

\phantomsection\label{concept:errors}
When \assref{ass:attained} holds, we use $D_k=\dist(x^k,\X)$ for the distance to the solution set and $\Delta_k=f(x^k)-\fs$ for the objective gap. For a constant-radius run, $N_t$ counts iterations up to and including the first terminal step, with $N_t=0$ at an optimal start and $N_t=\infty$ if no minimizer is ever reached.
\Cref{tab:overview} summarizes the main guarantees and points to their precise statements.

\begin{table}[!htbp]
\centering
\begin{minipage}{\textwidth}
\caption{Convergence guarantees and radius policies at a glance. Each row separates the setting, the mathematical guarantee, and its location.}
\label{tab:overview}
\small\setlength{\tabcolsep}{5pt}\renewcommand{\arraystretch}{1.38}
\begin{tabular}{@{}>{\raggedright\arraybackslash}p{.18\textwidth}>{\raggedright\arraybackslash}p{.23\textwidth}>{\raggedright\arraybackslash}p{.445\textwidth}>{\raggedright\arraybackslash}p{.075\textwidth}@{}}
\toprule
\textbf{Quantity} & \textbf{Setting} & \textbf{Bound or conclusion} & \textbf{See}\\
\midrule
\multicolumn{4}{@{}l}{\color{TableInk}\bfseries A. Accuracy after $K$ iterations}\tabularnewline[3pt]
Objective gap & Fixed radius; sequence of nonterminal steps & $\displaystyle\Delta_K\le\Delta_0e^{-2K^2t^2/D_0^2}$ & \S\ref{sec:values}\\
Stationarity & Same setting; $K\ge1$ & $\displaystyle m_K\le\frac{\Delta_0}{t}e^{-2(K-1)^2t^2/D_0^2}$ & \S\ref{sec:stationarity}\\
Symmetric Bregman distance & Same setting; $K\ge1$ & $\displaystyle\mathcal G_K\le\frac{D_0\Delta_0}{t}e^{-2(K-1)^2t^2/D_0^2}$ & \S\ref{sec:primal-dual}\\
\midrule
\multicolumn{4}{@{}l}{\color{TableInk}\bfseries B. Convergence and termination}\tabularnewline[3pt]
Nonsummable radii & $\sum\limits_{k=0}^{\infty} t_k=\infty$ & $f(x^k)\downarrow\inf f$; finite termination if $X^\star\ne\varnothing$, and $\norm{x^k}\to\infty$ otherwise & \S\ref{sec:termination}\\
Summable radii & $\sum\limits_{k=0}^{\infty} t_k<\infty$ & $x^k\to\bar x$, $f(x^k)\downarrow f(\bar x)$; finite length, possibly nonoptimal limit & \S\ref{sec:termination}\\
Iteration count & Fixed radius and dimension & $\displaystyle N_t=O_d\!\left(1+\frac{D_0}{t}\right)$; no dimension-free linear bound & \S\ref{sec:termination}\\
Total length & Fixed radius; nearest-minimizer selection at termination & $\displaystyle S_t\le\Psi_t(D_0)\le\frac{D_0^2}{t}+\frac t4$ & \S\ref{sec:termination}\\
Geometric radii & $t_k=t\rho^k$, $0<t<D_0$ & Minimum successful $\rho_\star$; $D_k=\Theta(\rho_\star^k)$ there, without finite termination & App.~\ref{app:geometric-factor}\\
\midrule
\multicolumn{4}{@{}l}{\color{TableInk}\bfseries C. Radius rules and reformulations}\tabularnewline[3pt]
Subgradient rule & $t_k=\tau\norm{h^k}$, $h^k\in\partial f(x^k)$ & $\displaystyle\Delta_K\le\frac{D_0^2}{4\tau K}$ & \S\ref{sec:calibration}\\
Objective-gap rule & $t_k=\tau\Delta_k^\alpha$, known $f_\star$ & $\displaystyle\Delta_K\le\left(\Delta_0^{-\alpha}+\frac{\alpha\tau K}{D_0}\right)^{-1/\alpha}$ & \S\ref{sec:calibration}\\
Epigraph lift & Start on the graph of $f$ & Projected steps are BPM with induced radii; error and length conversions are explicit & \S\ref{sec:reformulations}\\
Radius ordering & Same objective and start & $s<t$ can give $N_s=2<N_t=3$ & App.~\ref{app:radius-monotonicity}\\
Relaxation & Fixed $t>0$, $0<\lambda<2$; centers in $\dom f$ & $x^k\to x^\star\in X^\star$ & \S\ref{sec:relaxation}\\
\bottomrule
\end{tabular}
\par\vspace{5pt}\footnotesize\begin{itemize}[leftmargin=0pt,label={},itemsep=3pt]
\item $D_k=\dist(x^k,X^\star)$, $\Delta_k=f(x^k)-f_\star$, and $m_K=\dist(0,\partial f(x^K))$. The symmetric Bregman distance $\mathcal G_K$ is defined in \S\ref{sec:primal-dual}; $N_t$ includes the terminal step, and $S_t$ includes its displacement.
\item Rows using $X^\star$, $D_k$, or $\Delta_k$ assume attainment, except the two radius-summability rows, which explicitly distinguish the cases. Stronger implicit bounds and bounds over iterations $j,\ldots,K-1$ appear in \S\ref{sec:values}. Exact hypotheses are stated at the indicated locations.
\end{itemize}
\end{minipage}
\end{table}

\subsection{One optimality system, seven formulations}
\label{sec:four-views}
Before turning to convergence, we give seven equivalent descriptions of a single exact \BPM{} step.
The first is the ball-minimization formulation \textup{(\ref*{eq:bpm-intro})}. 
The remaining six express the same step as a normalized implicit subgradient update, an explicit update for the ball envelope, a projection onto a sublevel set, a Fenchel-dual problem, a scalar dual problem, and an epigraph-constrained problem. 
Their equivalence is governed by a common first-order optimality condition.

To state these formulations for a generic exact step, suppose \assref{ass:convex} holds, fix $x\in\dom f$ and $t>0$, and let $u\in\brox_f^t(x)$. 

We use the ball envelope introduced by \citet[Definition 5.1]{BPM} and the convex conjugate:

\begin{equation}
\label{eq:envelope-definition}
F_t(y)=\inf_{\norm{v-y}\le t}f(v),\qquad
f^*(g)=\sup_z\{\ip gz-f(z)\}.
\end{equation}

The envelope $F_t$ is itself proper, closed, and convex, but need not be smooth. 
For example, for $f(z)=|z|$ on $\R$,

\[
F_t(x)\overset{\eqref{eq:envelope-definition}}{=}\max\{|x|-t,0\},
\]

which is nonsmooth at $x=\pm t$.

At iteration $k$ of \ref*{eq:bpm-intro}, the generic notation above corresponds to $x=x^k$, $t=t_k$, and $u=x^{k+1}$. 
The proposition below first gives an optimality system that applies to both terminal and nonterminal iterations. 
The seven formulations are then stated for a nonterminal step, for which the normalized subgradient expressions are well defined.
The terminal case is discussed immediately after the proposition.

\begin{proposition}[Seven equivalent formulations of \BPM{}]\label{prop:four-views}
Fix $x\in\dom f$ and $t>0$. The relation $u\in\brox_f^t(x)$, where the operator is defined in \ref*{eq:bpm-intro}, is equivalent to the existence of $g$ satisfying
\begin{equation}\label{eq:view-system}
\boxed{g\in\partial f(u),\qquad x-u\in t\partial\norm{\cdot}(g).}
\end{equation}
The optimality system \eqref{eq:view-system} applies to both terminal and nonterminal iterations. On a nonterminal step, the following seven formulations are equivalent, the output $u$ is unique, and $\norm{x-u}=t$.
\begin{enumerate}[label=(\roman*),leftmargin=*]

\item \emph{Exact ball minimization (the original formulation).} The relation $u\in\brox_f^t(x)$ means
\begin{equation}\label{eq:view-ball}
  u\in\argmin_{v\in\B(x,t)}f(v).
\end{equation}
This is the defining \BPM{} step in \ref*{eq:bpm-intro}.

\item \emph{Normalized implicit subgradient step on $f$.} There exists a nonzero subgradient $g\in\partial f(u)$ satisfying
\begin{equation}\label{eq:view-implicit}
u=x-t\frac{g}{\norm g}
=\prox_{\lambda f}(x),\qquad \lambda=\frac{t}{\norm g}.
\end{equation}
We call $g$ a radial subgradient because it belongs to $\partial f(u)$ and is a positive multiple of $x-u$.
Here the proximal map \citep{Rockafellar1976} $\prox_{\lambda f}(x)$ minimizes $f(v)+\norm{v-x}^2/(2\lambda)$. 
The subgradient is evaluated at the new point. 
An arbitrary subgradient of $f$ there need not be radial.

\item \emph{Normalized explicit subgradient step on the ball envelope.} The same subgradient belongs to $\partial F_t(x)$, and
\begin{equation}\label{eq:view-explicit}
 u=x-t\frac{g}{\norm g},\qquad g\in\partial F_t(x).
\end{equation}
At such a center every envelope subgradient is nonzero and gives this same normalized direction. More precisely, for every $u\in\brox_f^t(x)$, including an optimal point,
\begin{equation}\label{eq:view-common-subgrad}
 \partial F_t(x)=\{g\in\partial f(u):\ip g{x-u}=t\norm g\}.
\end{equation}

\item \emph{Projection onto the returned sublevel set.} With $C_u=\{z:f(z)\le f(u)\}$,
\begin{equation}\label{eq:view-sublevel}
 u=\Proj_{C_u}(x),\qquad \norm{x-u}=t.
\end{equation}
The full-radius condition is essential: projection onto a higher sublevel set may produce a point strictly inside the ball. 
The level $f(u)$ is determined by the ball-minimization problem itself and is not an independent parameter of the projection formulation.

\item \emph{Fenchel-dual formulation.} The subgradient solves
\begin{equation}\label{eq:view-dual}
g\in\argmin_h\{f^*(h)+t\norm h-\ip xh\},
\qquad u\in\partial f^*(g),\quad x-u\in t\partial\norm{\cdot}(g).
\end{equation}
Equivalently, the primal and dual optimal values satisfy
\begin{equation}\label{eq:view-dual-value}
F_t(x)=\max_g\{\ip xg-f^*(g)-t\norm g\}.
\end{equation}
This dual problem has a solution even if $\partial f(x)$ is empty. When a previous subgradient $p\in\partial f(x)$ is available, it can also be written as
\begin{equation}\label{eq:view-dual-bregman}
 g\in\argmin_h\{t\norm h+D_{f^*}^{x}(h,p)\},
 \quad D_{f^*}^{x}(h,p)=f^*(h)-f^*(p)-\ip x{h-p}.
\end{equation}
The term $D_{f^*}^{x}(h,p)$ is the generalized Bregman distance associated with $f^*$ and the chosen subgradient $x\in\partial f^*(p)$ \citep{Bregman1967,Burger2015Bregman}. Thus the dual problem minimizes the norm penalty plus this generalized Bregman distance; neither smoothness nor strict convexity of $f^*$ is required.

\item \emph{Lagrangian dual formulation.} Dualizing the ball constraint gives
\begin{equation}\label{eq:view-scalar-dual}
 q_t(\gamma)=\inf_z\left\{f(z)+\frac\gamma2(\norm{z-x}^2-t^2)\right\},
 \qquad F_t(x)=\max_{\gamma\ge0}q_t(\gamma).
\end{equation}
Here $q_t(0)=\inf f$. 
On a nonterminal step, an optimal $\gamma_\star>0$ exists and
\begin{equation}\label{eq:view-scalar-path}
 u=u_{\gamma_\star},\qquad
 u_\gamma=\prox_{f/\gamma}(x),\qquad \norm{u_{\gamma_\star}-x}=t.
\end{equation}
The function $q_t$ is concave, is differentiable for $\gamma>0$, and satisfies
\begin{equation}\label{eq:view-scalar-derivative}
 q_t'(\gamma)=\tfrac12(\norm{u_\gamma-x}^2-t^2).
\end{equation}
The displacement $\norm{u_\gamma-x}$ is continuous and nonincreasing in $\gamma>0$. Thus an optimal multiplier $\gamma_\star$ selects a point on the proximal path whose distance from $x$ is exactly the prescribed radius $t$. 
The optimal multiplier need not be unique.

\item \emph{Epigraph formulation with a cylindrical constraint.} The ball-minimization problem can be lifted to the epigraph of $f$ as
\begin{equation}\label{eq:view-epigraph-cylinder}
(u,f(u))\in\argmin_{v\in\R^d,\ s\in\R}
\bigl\{s:\ f(v)\le s,\ \norm{v-x}\le t\bigr\}.
\end{equation}
The feasible set is
$\operatorname{epi}f\cap(\B(x,t)\times\R)$, where
$\operatorname{epi}f=\{(v,s):f(v)\le s\}$.
For any feasible $v$, minimizing over $s$ gives $s=f(v)$, so the lifted problem is equivalent to minimizing $f(v)$ over $\B(x,t)$. Thus every minimizer $v$ of the ball problem corresponds to the epigraph minimizer $(v,f(v))$, and conversely. This equivalence also holds on terminal iterations.
\end{enumerate}
\end{proposition}

Proofs of these equivalences are given in \Cref{app:four-views}.

\paragraph{How the views are connected.}

Views (i) and (iv) describe the same point by exchanging the roles of objective value and distance. On a nonterminal step, let $\alpha=f(u)=F_t(x)$. Then

\begin{equation}
\label{eq:view-exchange}\min_{\norm{v-x}\le t}f(v)\overset{\eqref{eq:view-ball}}{=}\alpha,\qquad\min_{f(v)\le\alpha}\norm{v-x}\overset{\eqref{eq:view-sublevel}}{=}t.
\end{equation}

Thus view (i) fixes the radius $t$ and determines the lowest attainable objective value $\alpha$, whereas view (iv) fixes this attained level and recovers the same point as the nearest point in the corresponding sublevel set. In particular, $\alpha$ is determined by the ball-minimization problem rather than specified in advance.

Views (ii) and (iii) use the same subgradient at two different points. In the implicit formulation, $g\in\partial f(u)$ is evaluated at the new point $u$, while in the explicit formulation, the same subgradient satisfies $g\in\partial F_t(x)$ at the current center $x$. 
Both formulations therefore produce the same normalized direction from $x$ to $u$.

View (v) is the Fenchel-dual formulation of the ball problem in view (i). The connection is provided by the ball envelope:

\begin{equation}
\label{eq:view-conjugate}
F_t\overset{\eqref{eq:envelope-definition}}{=}f\mathbin\square\delta_{\B(0,t)},\qquad
F_t^*=f^*+t\norm{\cdot}.
\end{equation}

Here $\square$ denotes infimal convolution, and $\delta_C$ is the indicator function of $C$, equal to zero on $C$ and $+\infty$ outside $C$. The dual optimizers are precisely the elements of $\partial F_t(x)$. On a nonterminal step, these subgradients may have different magnitudes but determine the same normalized direction $x-u$. The Bregman representation in \eqref{eq:view-dual-bregman} additionally requires a subgradient $p\in\partial f(x)$ at the current center. Along a \BPM{} trajectory, such a subgradient may be inherited from the preceding nonterminal step when it is retained by the oracle.

View (vi) gives a scalar alternative to the vector dual formulation. The multiplier $\gamma_\star$ selects a point on the proximal path in view (ii) whose displacement from $x$ equals the prescribed radius $t$. The scalar and vector subgradients are related by

\[
g\overset{\eqref{eq:radial}}{=}\gamma_\star(x-u),
\qquad
\gamma_\star=\frac{\norm g}{t}.
\]

Since $\norm{u_\gamma-x}$ is continuous and nonincreasing in $\gamma$, an optimal multiplier can be found by solving the one-dimensional equation
$\norm{u_\gamma-x}=t$, for example by a bracketing method when proximal evaluations are available.
This is the Lagrange-multiplier form of the proximal and trust-region connections described in \citet[Sections 3 and 7]{BPM}.

Finally, view (vii) lifts the ball problem to the epigraph of $f$. The spatial constraint $\norm{v-x}\le t$ is unchanged, while the objective value is represented by the additional variable $s$. Since minimizing $s$ subject to $f(v)\le s$ forces $s=f(v)$ at optimality, the lifted problem has exactly the same minimizing $v$ as view (i). This is the standard epigraph reformulation of an optimization problem \citep[Section 4.1.3]{BoydVandenberghe2004}.

\paragraph{Terminal iterations.}

On a terminal step, the ball $\B(x,t)$ contains a global minimizer, so every exact \BPM{} output is itself a global minimizer. The optimality system \eqref{eq:view-system} remains valid with $g=0$: since $\partial\norm{\cdot}(0)$ is the unit ball, the inclusion $x-u\in t\partial\norm{\cdot}(0)$ is equivalent to $\norm{x-u}\le t$. The normalized subgradient formulations do not apply with $g=0$, since $g/\norm g$ is undefined.

In the scalar dual formulation, $\gamma=0$ is optimal because $q_t(0)=\inf f=F_t(x)$, although this does not determine which minimizer in the ball is returned. Likewise, in the projection formulation, $C_u=\X$, so $\Proj_{C_u}(x)$ selects a nearest minimizer.

The normalized implicit and explicit subgradient descriptions, together with their radial geometry, build on \citet[Table 1 and Sections 3 and 5]{BPM}. The connection with sublevel projection also appears in proximal self-contraction theory \citep{DaniilidisEtAl2015}.

\paragraph{A common example for the seven formulations.}

The following example illustrates how the seven formulations describe the same exact \BPM{} step. Consider

\begin{equation}
\label{eq:common-example}
  f(z)=\frac{1}{2}(z_1^2+4z_2^2),\qquad x=\left(2,\frac{5}{4}\right),\qquad t=\sqrt2.
\end{equation}

The ball minimization problem returns $u=(1,1/4)$, for which

\[
  g=\nabla f(u)\overset{\eqref{eq:common-example}}{=}(1,1)=x-u,\qquad f(u)=\frac58,\qquad \lambda\overset{\eqref{eq:prox-equivalence}}{=}\frac{t}{\norm g}=1.
\]

Thus the common optimality system \eqref{eq:view-system} holds, and the same step can be read through views (i)--(vi) in \Cref{fig:views-primal,fig:views-gradients,fig:views-dual,fig:views-scalar}. For view (vii), the corresponding epigraph optimizer is
$(u,f(u))=((1,1/4),5/8)$.
Although this example is smooth for ease of visualization, the equivalences in \Cref{prop:four-views} do not require smoothness.

\begin{figure}[p]
\centering
\refstepcounter{figure}\label{fig:views-primal}
{\small\bfseries Figure \thefigure. Ball minimization and sublevel projection\par}
\includegraphics[width=\textwidth]{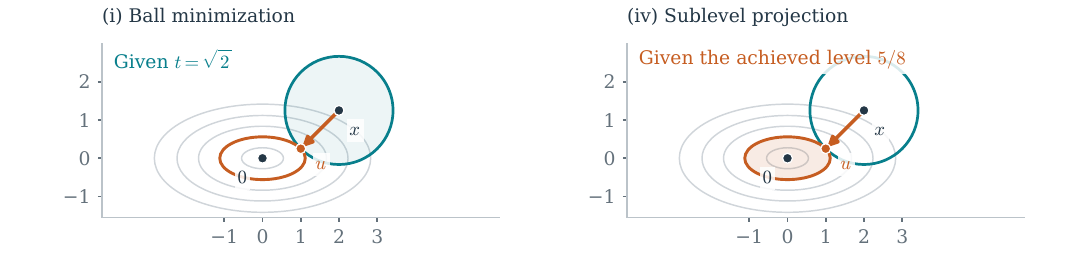}
\vspace{2pt}
\refstepcounter{figure}\label{fig:views-gradients}
{\small\bfseries Figure \thefigure. Implicit and envelope updates\par}
\includegraphics[width=\textwidth]{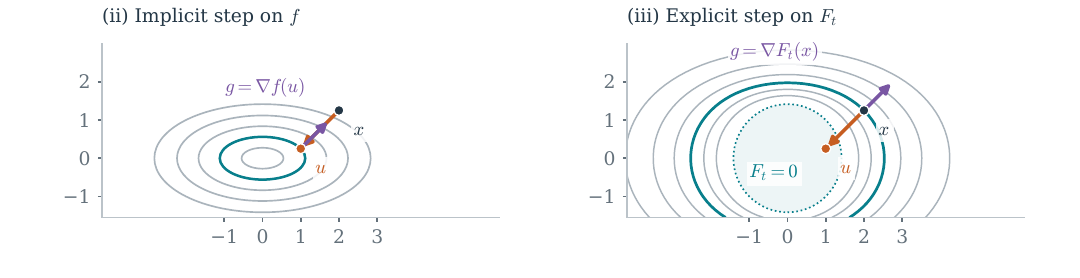}
\vspace{2pt}
\refstepcounter{figure}\label{fig:views-dual}
{\small\bfseries Figure \thefigure. Vector duality\par}
\includegraphics[width=\textwidth]{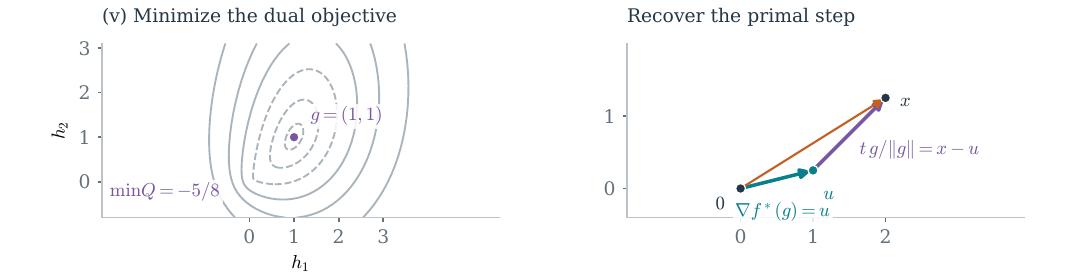}
\vspace{2pt}
\refstepcounter{figure}\label{fig:views-scalar}
{\small\bfseries Figure \thefigure. Scalar duality\par}
\includegraphics[width=\textwidth]{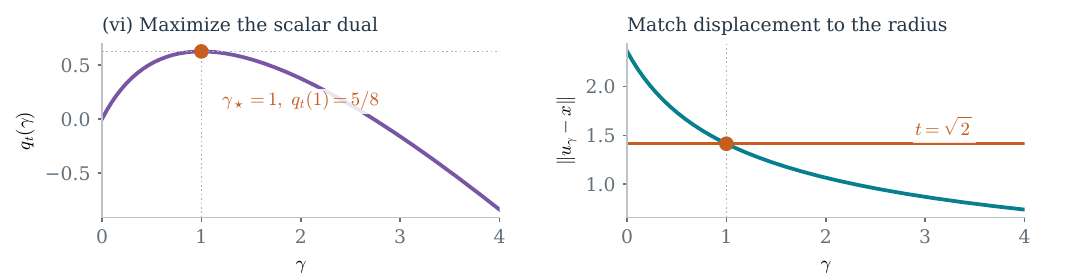}
\vspace{2pt}
\caption*{Illustrations of formulations (i)--(vi) in \Cref{prop:four-views} for a single BPM step with $f(z)=\tfrac12(z_1^2+4z_2^2)$, $x=(2,5/4)$, and $t=\sqrt2$, yielding $u=(1,1/4)$ and $g=(1,1)$. Figure~\ref{fig:views-primal} illustrates parts (i) and (iv): ball minimization \eqref{eq:view-ball} and projection onto the achieved sublevel $f(u)=5/8$ in \eqref{eq:view-sublevel}. Figure~\ref{fig:views-gradients} illustrates parts (ii) and (iii): the implicit step on $f$ in \eqref{eq:view-implicit} and the explicit step on $F_t$ in \eqref{eq:view-explicit}, using the same subgradient $g$. Figure~\ref{fig:views-dual} illustrates part (v), \eqref{eq:view-dual}: the vector-dual minimum $Q(g)=-5/8$ and the reconstruction $x=\nabla f^*(g)+tg/\norm g$. Figure~\ref{fig:views-scalar} illustrates part (vi): the scalar-dual maximum $q_t(1)=5/8$ in \eqref{eq:view-scalar-dual} and the matching displacement $\norm{u_1-x}=t$ in \eqref{eq:view-scalar-path}.}
\end{figure}

\subsection{Sharpening and extending the BPM theory}

The original analysis of \citet{BPM} establishes finite termination within $\lceil D_0^2/t^2\rceil$ constant-radius iterations and the objective suboptimality bound $\Delta_K\le\Delta_0(1+t/D_0)^{-K}$.
The results developed here sharpen and extend this theory in several directions.
In \Cref{sec:values}, retaining the objective decrease in the bound on squared-distance decrease yields stronger multi-step bounds on objective suboptimality, which in turn give guarantees on the minimum subgradient norm and a symmetric Bregman distance to a minimizer.
The original analysis already allows arbitrary positive radii and guarantees finite termination when their squared sum reaches the initial squared distance to the solution set \citep[Theorem~8.1 and Corollary~8.2]{BPM}. In \Cref{sec:termination}, we strengthen this theory by characterizing convergence through the sum of the radii themselves and by treating problems whose infimum is not attained.
When a minimizer exists, the same trajectory analysis improves the constant-radius step bound to linear dependence on $D_0/t$ in fixed dimension, while a polyhedral family shows why the dimension-independent quadratic bound remains asymptotically sharp.

The later sections develop the effects of reformulation, adaptive radius selection, and relaxation beyond the convergence guarantees in the original analysis.
\Cref{sec:reformulations} determines which reformulations preserve exact ball calls and shows how epigraphical lifting changes the induced radius sequence and associated guarantees.
\Cref{sec:calibration} studies the effect of radius choice, including geometric schedules, nonmonotonicity of the fixed-radius termination count, and adaptive rules based on subgradient or objective-gap information.
Finally, \Cref{sec:relaxation} extends the analysis to relaxed updates of the ball-proximal step.

\subsection{Related work}
\label{sec:related-work}

The exact recurrence \ref*{eq:bpm-intro} appears in \citet[Appendix A, Algorithm 6]{CarmonEtAlBallOracle}, together with the segment contraction used in \Cref{thm:segment}.
The later work of \citet{BPM} develops the broximal framework, its geometric and proximal connections, and the finite-termination analysis on which we build.

Ball optimization oracles also serve as primitives for accelerated convex optimization.
\citet{CarmonEtAlBallOracle} show that additional state and different query rules can reduce the number of oracle calls compared with the BPM recurrence.
\citet{CarmonEtAl2021InsideBall} further refine ball-oracle acceleration and develop an implementation for maximal-loss minimization through a smoothed objective. 
These works therefore address acceleration and the cost of implementing ball oracles.

Several subsequent works develop related broximal constructions in other settings. 
Non-Euclidean \BPM{} replaces the Euclidean ball by more general norm geometry \citep{NonEuclideanBPM}; stabilized proximal methods combine proximal updates with trust-region-type constraints \citep{StabilizedPPM}; and broximal alignment extends related geometric ideas beyond the convex setting \citep{BroximalAlignment}. 
Local linear minimization methods also admit a forward--backward interpretation in which a proximal backward step is replaced by a broximal one \citep{LocalLMO}. 

The trajectory-length argument uses the classical theory of self-contracted curves and sequences \citep{DaniilidisEtAl2015,BohmDaniilidis2020}.
The analysis of relaxed centers uses the cutter framework of \citet{CegielskiCensor2012} to obtain a Fej\'er-type decrease.
The subgradient-based radius rule uses classical proximal-point analysis \citep{Rockafellar1976} and the sharp value bound of \citet{TaylorHendrickxGlineur2017}; the objective-gap-based rules follow directly from the \BPM{} segment contraction.

\section{One-step geometry and progress}
\label{sec:foundation}
Throughout this section, we invoke \assref{ass:convex} and \assref{ass:iteration} and impose \assref{ass:attained} only when a minimizer is required. Under \assref{ass:attained}, let
$p^k=\Proj_{\X}(x^k)$.
The terminal and nonterminal cases introduced in \Cref{sec:four-views} are then characterized by
$D_k\le t_k$ and $D_k>t_k$, respectively.

\subsection{Geometry of an exact step}

The next lemma collects the geometric properties of an exact \BPM{} step used below. A nonterminal step has a unique output on the boundary of the ball, together with a radial subgradient and an equivalent proximal representation. Under attainment, the lemma also characterizes the possible terminal outputs. These facts are established in \citet[Appendices D--E]{BPM}, we include the proof in \Cref{proof:foundation:1}.

\begin{lemma}[Characterization of an exact \BPM{} step]\label{lem:certificate}
Under \assref{ass:convex}, let $x\in\dom f$ and $t>0$. The set $\brox_f^t(x)$ is nonempty and compact.
If the step is nonterminal, the output $u$ is unique and satisfies
\begin{equation}\label{eq:radial}
\norm{x-u}=t,\qquad g=c(x-u)\in\partial f(u),\qquad c>0.
\end{equation}
Under \assref{ass:attained}, if $D(x):=\dist(x,\X)\le t$, then $\brox_f^t(x)=\X\cap\B(x,t)$. At $D(x)=t$ this set is the singleton $\{\Proj_{\X}(x)\}$. On a nonterminal step,
\begin{equation}\label{eq:prox-equivalence}
u=\prox_{\lambda f}(x),\qquad
\lambda=\frac{1}{c}=\frac{t}{\norm g},
\end{equation}
where $\prox_{\lambda f}(x)=\argmin_z\{f(z)+\norm{z-x}^2/(2\lambda)\}$.
\end{lemma}

The subgradient $g$ in \eqref{eq:radial} need not be returned by the ball oracle. 
Its existence is sufficient for the analysis. 
When $D(x)<t$, a terminal step may have several possible outputs. 
For trajectory arguments that require a definite terminal point, we choose the minimizer nearest to the current center.

\begin{definition}[Nearest-minimizer selection at termination]\label{def:projected}
Under \assref{ass:convex} and \assref{ass:attained}, for $x\in\dom f$ define
\begin{equation*}
T_t(x)=
\begin{cases}
u, & D(x)>t,\quad \text{where } \brox_f^t(x)=\{u\},\\
\Proj_{\X}(x), & D(x)\le t.
\end{cases}
\end{equation*}
Then
\begin{equation*}
\norm{x-T_t(x)}=\min\{t,D(x)\}.
\end{equation*}
\end{definition}
Before the first terminal step, this selection coincides with \BPM{}, so it preserves the termination count. Selecting the nearest terminal minimizer requires additional information: a point-only exact ball oracle need not return that particular minimizer.

\subsection{Progress from an exact step}

We use the generalized Bregman distance associated with a specified subgradient \citep{Bregman1967,Burger2015Bregman}. For $g\in\partial f(u)$ and $z\in\dom f$, this distance from $z$ to $u$ is defined by

\begin{equation}
\label{eq:bregman-definition}
D_f^g(z,u)
=f(z)-f(u)-\ip{g}{z-u}.
\end{equation}

By the subgradient inequality, $D_f^g(z,u)\ge0$. This quantity depends on the chosen subgradient and is not generally a metric. 
The following theorem gives the basic one step progress relations.

\begin{theorem}[Exact progress identity]\label{thm:identity}
Under \assref{ass:convex} and \assref{ass:iteration}, consider a nonterminal step with output $u=x^{k+1}$ and a subgradient $g^k=c_k(x^k-u)$ from \Cref{lem:certificate}. For every $z\in\dom f$,
\begin{equation}\label{eq:identity}
 \norm{x^k-z}^2-\norm{u-z}^2
 =t_k^2+\frac{2}{c_k}\bigl[f(u)-f(z)+D_f^{g^k}(z,u)\bigr].
\end{equation}
The objective decrease satisfies
\begin{equation}\label{eq:drop}
 f(x^k)-f(u)=t_k\norm{g^k}+D_f^{g^k}(x^k,u).
\end{equation}
In particular, whenever $f(z)\le f(u)$,
\begin{equation}\label{eq:lower-fejer}
 \norm{u-z}^2\le\norm{x^k-z}^2-t_k^2.
\end{equation}
\end{theorem}
The proof is given in \Cref{proof:foundation:2}.
 
Equation \eqref{eq:identity} retains the objective-gap and Bregman terms in the distance decrease, while \eqref{eq:drop} relates objective decrease to the radial subgradient. The lower-level estimate \eqref{eq:lower-fejer} follows by discarding these additional nonnegative terms.

\begin{example}[Slack in the squared-radius estimate]\label{ex:one-step}
Take $f(x)=|x|$, $x^0=3$, and $t=1$. The first ball is $[2,4]$, so $u=2$, $g=1$, and $c=1$. For the minimizer $z=0$, the Bregman distance vanishes, and \eqref{eq:identity} gives

\[
\underbrace{3^2-2^2}_{5}
\overset{\eqref{eq:identity}}{=}
\underbrace{1^2}_{\text{radius squared}}
+
\underbrace{2\cdot2}_{\text{objective-gap contribution}}.
\]

The squared-radius estimate \eqref{eq:lower-fejer} retains only the first term, whereas the exact identity \eqref{eq:identity} accounts for the remaining four units through the objective gap. By \eqref{eq:drop}, the objective decrease is
$3-2=t|g|=1$.
With the same fixed radius, the iterates are $3,2,1,0$, so the method reaches the minimizer in three iterations. In contrast, telescoping only the squared-radius decrease in \eqref{eq:lower-fejer} gives the generic bound $D_0^2/t^2=9$. 
This example shows that the squared-radius bound can be loose on a particular instance.
\end{example}

A nonterminal step also determines a halfspace containing the lower sublevel set. 
With $n_k={(x^k-u)}/{t_k}$, the subgradient inequality gives

\begin{equation}
\label{eq:cut}
H_k:=\{z:\ip{n_k}{z-u}\le0\}
\overset{\eqref{eq:radial}}{\supseteq}\{z:f(z)\le f(u)\},
\qquad
u=\Proj_{H_k}(x^k).
\end{equation}

Thus the lower sublevel set lies in a halfspace through $u$, and $u$ is the projection of $x^k$ onto that halfspace.

Dropping the additional nonnegative terms in \eqref{eq:identity} recovers the classical squared-radius decrease.

\begin{corollary}[Squared-radius bound]\label{cor:squared}
Under \assref{ass:convex}, \assref{ass:attained}, and \assref{ass:iteration}, every nonterminal step satisfies $D_{k+1}^2\le D_k^2-t_k^2$.
Consequently, one of the first $K$ iterations is terminal whenever
\begin{equation*}
\sum\limits_{k=0}^{K-1}t_k^2\ge D_0^2.
\end{equation*}
In particular, for a constant radius $t$,
\begin{equation}\label{eq:squared-count}
N_t\le\left\lceil\frac{D_0^2}{t^2}\right\rceil.
\end{equation}
For the stopped sequence,
\begin{equation*}
D_K^2
\le
\max\left\{D_0^2-\sum\limits_{k=0}^{K-1}t_k^2,0\right\},
\end{equation*}
where the radius sequence may be extended arbitrarily after stopping.
\end{corollary}

The proof is given in \Cref{proof:foundation:3}.

\subsection{Subgradients associated with exact BPM steps}

For a common radius, the corresponding comparison of radial subgradient norms appears in \citet[Corollary E.13 and Remark E.14]{BPM}, including in the nonsmooth setting.
The following argument also allows variable radii and arbitrary radial subgradients.

\begin{proposition}[Nonincreasing norms of radial subgradients]\label{prop:certificate-monotone}
Under \assref{ass:convex} and \assref{ass:iteration}, suppose two consecutive iterations are nonterminal. Then any corresponding radial subgradients satisfy
\begin{equation}\label{eq:acute}
\ip{g^k}{\frac{g^{k+1}}{\norm{g^{k+1}}}}\ge\norm{g^{k+1}}.
\end{equation}
Consequently, $\norm{g^{k+1}}\le\norm{g^k}$, and consecutive radial subgradient directions have positive inner product. Under \assref{ass:attained}, after $K\ge1$ iterations, taking the subgradient to be zero at and after a terminal output gives
\begin{equation}\label{eq:stationarity}
\dist(0,\partial f(x^K))\le\norm{g^{K-1}}
\le\frac{f(x^0)-f(x^K)}{\sum\limits_{k=0}^{K-1}t_k}
\le\frac{\Delta_0}{\sum\limits_{k=0}^{K-1}t_k}.
\end{equation}
\end{proposition}

The proof is given in \Cref{proof:foundation:4}. 
The subgradient bound will be used below to derive stationarity estimates and adaptive radius rules.

\begin{corollary}[Monotonicity of the minimum subgradient norm]\label{prop:minimum-subgradient}
Under \assref{ass:convex}, for $x\in\dom\partial f$ and any exact output $u$, $m(u)\le m(x)$, where $m(x)=\dist(0,\partial f(x))$.
\end{corollary}

The proof is given in \Cref{proof:foundation:5}.

\section{Sharper objective-gap bounds for BPM}
\label{sec:values}

Throughout this section \assref{ass:convex}, \assref{ass:attained}, and \assref{ass:iteration} hold.
We first study the objective gap
\[
\Delta_k=f(x^k)-f_\star.
\]
The analysis proceeds in two stages. 
We begin with a one-step contraction obtained by considering a feasible point on the segment from $x^k$ to its projection $p^k=\Proj_{\X}(x^k)$ onto the solution set.
We then retain the actual objective decrease in each step and relate it to the decrease in squared distance, which leads to stronger multi-step bounds. 
In the second part of the section, we use these objective-gap estimates to control stationarity and a symmetric Bregman distance to a minimizer.

\subsection{Objective-gap contraction and squared-distance decrease}

\noindent
The following result uses the segment argument of \citet[Appendix A, Theorem 26]{CarmonEtAlBallOracle}, stated here under our assumptions and for variable radii.

\begin{theorem}[Segment contraction]\label{thm:segment}
If $D_k>0$, then
\begin{equation}\label{eq:segment}
 \Delta_{k+1}\le\left(1-\min\left\{\frac{t_k}{D_k},1\right\}\right)\Delta_k.
\end{equation}
If $D_k=0$, both gaps are zero. For $0<t<D_0$ and constant radius, the stopped sequence satisfies
\begin{equation}\label{eq:frozen}
 \Delta_K\le\left(1-\frac{t}{D_0}\right)^K\Delta_0.
\end{equation}
\end{theorem}

The proof is given in \Cref{proof:value_rates:1}.
For comparison, \citet[Corollary 8.2]{BPM} gives the fixed-radius factor $(1+t/D_0)^{-K}$.
Since $1-a\le(1+a)^{-1}$ for $0<a<1$, \eqref{eq:frozen} improves this original geometric factor.
We therefore use $(1-t/D_0)^K$ as the reference geometric bound below.

Multiplying \eqref{eq:segment} along a sequence of nonterminal steps gives

\begin{equation}
\label{eq:segment-product}
 \Delta_K\overset{\eqref{eq:segment}}{\le}\Delta_0\prod_{k<K}\left(1-\frac{t_k}{D_k}\right).
\end{equation}

For constant radius, \Cref{cor:squared} further gives

\[
 \Delta_K\overset{\eqref{eq:segment-product},\eqref{eq:lower-fejer}}{\le}\Delta_0\prod_{k<K}
 \left(1-\frac{t}{\sqrt{D_0^2-kt^2}}\right).
\]

If the first $K$ iterations are nonterminal, then $D_k>t$ for every $k<K$, so all factors are positive. If a terminal step occurs earlier, then $\Delta_K=0$.

The same one step geometry also yields a direct relation between objective suboptimality and distance to the solution set.
The proof is given in \Cref{proof:value_rates:4}.

\begin{corollary}[Objective gap--distance comparison]\label{prop:gap-distance}
If $D_0>0$, then for arbitrary positive radii,
\begin{equation}\label{eq:gap-distance}
 \Delta_K\le\frac{D_K}{D_0}\Delta_0.
\end{equation}
\end{corollary}

The segment bound in \eqref{eq:segment} can be exact: for $f(x)=|x|$ with $x>t>0$, it is attained with equality. It also applies to variable radii and does not require access to radial subgradients. However, it tracks objective decrease and distance decrease only through the current distance $D_k$. The next result retains the stronger coupling between these two quantities.

\begin{theorem}[Refined distance decrease from objective progress]\label{thm:refined}
At any nonterminal step,
\begin{align}
D_{k+1}^2
&\le D_k^2-t_k^2-\frac{2\Delta_{k+1}}{c_k},
\label{eq:refined-c}\\
D_{k+1}^2
&\le D_k^2
-t_k^2\frac{\Delta_k+\Delta_{k+1}}
{\Delta_k-\Delta_{k+1}}.
\label{eq:refined-ratio}
\end{align}
In particular, the denominator is positive and the multiplier of $t_k^2$ in \eqref{eq:refined-ratio} exceeds one.
\end{theorem}

The proof is given in \Cref{proof:value_rates:2}.
Equation \eqref{eq:refined-c} is the distance refinement in \citet[equation (23)]{BPM}. 
Eliminating the multiplier with their objective-decrease estimate gives \eqref{eq:refined-ratio}.
To make this relation explicit, define

\begin{equation}
\label{eq:ratio-definitions}
q_k\coloneq\frac{\Delta_{k+1}}{\Delta_k}, \qquad \beta_k\coloneq\frac{D_k^2-D_{k+1}^2}{t_k^2}.
\end{equation}

Here, $q_k$ is the fraction of the objective gap remaining after the $k$-th step, while $\beta_k$ measures the corresponding decrease in squared distance to the solution set, normalized by $t_k^2$.
The refined inequality \eqref{eq:refined-ratio} then gives

\begin{equation}
\label{eq:ratio-lower}
\beta_k\overset{\eqref{eq:refined-ratio},\eqref{eq:ratio-definitions}}{\ge}\frac{1+q_k}{1-q_k}.
\end{equation}

The above inequality shows that a small relative decrease in the objective gap must be accompanied by a large decrease in squared distance to the solution set. Indeed, if $q_k$ is close to $1$, then $(1+q_k)/(1-q_k)$ is large, hence $\beta_k$ must also be large. Since the total decrease in squared distance is limited, such behavior cannot persist over many iterations with the same radius. We make this observation quantitative in the next subsection.

On the other hand, $\beta_k$ cannot be arbitrarily large. At a nonterminal step, $\norm{x^{k+1}-x^k}=t_k$, and the triangle inequality gives

\begin{equation}
\label{eq:distance-triangle}
D_{k+1}
=\dist(x^{k+1},\X)
\ge D_k-\norm{x^{k+1}-x^k}
\overset{\eqref{eq:radial}}{=}D_k-t_k.
\end{equation}

Therefore,

\begin{equation}
\label{eq:ratio-upper}
\beta_k
\overset{\eqref{eq:ratio-definitions}}{=}\frac{D_k^2-D_{k+1}^2}{t_k^2}
\overset{\eqref{eq:distance-triangle}}{\le}
\frac{D_k^2-(D_k-t_k)^2}{t_k^2}
=
\frac{2D_k}{t_k}-1.
\end{equation}

Combining the lower and upper bounds on $\beta_k$ yields

\begin{equation}
\label{eq:ratio-sandwich}
\frac{1+q_k}{1-q_k}
\overset{\eqref{eq:ratio-lower},\eqref{eq:ratio-upper}}{\le}
\frac{2D_k}{t_k}-1,
\end{equation}

which is equivalent to

\[
q_k\overset{\eqref{eq:ratio-sandwich}}{\le} 1-\frac{t_k}{D_k}.
\]

Thus the refined distance inequality recovers the segment contraction, while retaining additional information about how objective-gap decrease is coupled to decrease in squared distance.

\subsection{Refined multi-step objective-gap bounds}

We now aggregate the relation between objective-gap decrease and squared-distance decrease across multiple iterations.
Suppose the radius is constant, $t_k=t$, and define $
S\coloneq {D_0^2}/{t^2}$.
For the first $K$ iterations, all assumed nonterminal,

\begin{equation}
\label{eq:beta-budget}
\sum\limits_{k=0}^{K-1}\beta_k
\overset{\eqref{eq:ratio-definitions}}{=}
\frac{1}{t^2}\sum\limits_{k=0}^{K-1}\left(D_k^2-D_{k+1}^2\right)
=
\frac{D_0^2-D_K^2}{t^2}
<S.
\end{equation}

Thus the quantities $\beta_k$ share a finite total budget.
Moreover, the inequality
$
\beta_k\ge{(1+q_k)}/{(1-q_k)}
$
is equivalent to
$
q_k\le{(\beta_k-1)}/{(\beta_k+1)}.
$
Since

\begin{equation}
\label{eq:gap-product}
\frac{\Delta_K}{\Delta_0}
\overset{\eqref{eq:ratio-definitions}}{=}
\prod_{k<K}q_k,
\end{equation}

we obtain

\[
\frac{\Delta_K}{\Delta_0}
\overset{\eqref{eq:gap-product},\eqref{eq:ratio-lower}}{\le}
\prod_{k<K}\frac{\beta_k-1}{\beta_k+1}.
\]

\phantomsection\label{concept:jensen}
The remaining question is therefore how large this product can be when the $\beta_k$ satisfy the total-budget constraint above. To apply Jensen's inequality, take the negative logarithm of each factor and define $\phi(b)=-\log((b-1)/(b+1))$ for $b>1$.
This function is strictly convex, since $\phi''(b)=4b/(b^2-1)^2>0$.
Writing $\bar\beta=K^{-1}\sum_{k=0}^{K-1}\beta_k$, Jensen's inequality with equal weights $1/K$ gives
\[
\frac1K\sum\limits_{k=0}^{K-1}\phi(\beta_k)
\overset{\text{Jensen}}{\ge}\phi(\bar\beta).
\]
Multiplying by $-K$ reverses the inequality; exponentiating then gives
\[
\prod_{k=0}^{K-1}\frac{\beta_k-1}{\beta_k+1}
=\exp\!\left(-\sum\limits_{k=0}^{K-1}\phi(\beta_k)\right)
\overset{\text{Jensen}}{\le}\exp\!\left(-K\phi(\bar\beta)\right)
=\left(\frac{\bar\beta-1}{\bar\beta+1}\right)^K.
\]
Equality holds when all $\beta_k$ are equal, so equal allocation gives the largest product for a fixed total.
Finally, the fraction $(b-1)/(b+1)$ increases with $b>1$, and the budget in \eqref{eq:beta-budget} gives $\bar\beta<S/K$.
Replacing $\bar\beta$ by $S/K$ therefore yields the upper bound $((S-K)/(S+K))^K$. This leads to the next theorem. 

\begin{theorem}[Jensen objective-gap bound]\label{thm:jensen}
Let $t_k=t>0$, $D_0>0$, and $S=D_0^2/t^2$. If the first $K\ge1$ iterations are nonterminal, then $K<S$ and
\begin{equation}\label{eq:jensen}
 \Delta_K\le\Delta_0\left(\frac{S-K}{S+K}\right)^K
 \le\Delta_0\exp\!\left(-\frac{2K^2}{S}\right).
\end{equation}
For $0<\varepsilon<\Delta_0$, the first index with $\Delta_K\le\varepsilon$ obeys
\begin{equation}\label{eq:value-complexity}
 N_\varepsilon\le\min\left\{
 \left\lceil S\right\rceil,
 \left\lceil\frac{D_0}{t}\sqrt{\frac12\log\frac{\Delta_0}{\varepsilon}}\right\rceil
 \right\}.
\end{equation}
\end{theorem}

The proof is given in \Cref{proof:value_rates:3}.
For the first step in \Cref{ex:one-step}, the segment bound is exact, giving $\Delta_1\le 2$, whereas the Jensen bound gives $\Delta_1\le 2.4$. 
Thus the Jensen bound is not necessarily sharper at short horizons, its advantage comes from aggregating information across multiple iterations.

The above Jensen bound and the geometric bound based on the initial distance capture different parts of the preceding analysis, and neither is uniformly stronger. 
For $a=t/D_0\in(0,1)$, a direct comparison gives
\[
\left(\frac{S-K}{S+K}\right)^K\le (1-a)^K
\quad\Longleftrightarrow\quad
K\ge \frac{1}{a(2-a)}.
\]
Thus the geometric bound based on the initial distance can be sharper over short horizons, as in \Cref{ex:one-step}, whereas the Jensen bound becomes sharper once sufficiently many iterations are aggregated. 
The factor $((S-K)/(S+K))^K$ cannot be improved using only the constraints $\beta_k>1$, $\sum_{k=0}^{K-1}\beta_k<S$, and $q_k\le(\beta_k-1)/(\beta_k+1)$. To see this, take all $\beta_k$ equal to a number $b$ with $1<b<S/K$, and set $q_k=(b-1)/(b+1)$. As $b$ increases toward $S/K$, the product of the $q_k$ approaches the stated factor arbitrarily closely. Equality is excluded by the strict budget constraint, since it would require $\sum_{k=0}^{K-1}\beta_k=S$.
An actual trajectory satisfies additional constraints, however, and these give a strict improvement for fixed $S$ and $K$.

\begin{remark}[Using the final distance]\label{rem:jensen-final-distance}
The Jensen bound in \Cref{thm:jensen} uses only $\sum_{k<K}\beta_k<S$ and ignores the remaining distance $D_K$.
Let $r=\Delta_K/\Delta_0\in(0,1)$.
By \eqref{eq:gap-distance}, $D_K\ge rD_0$, so

\begin{equation}
\label{eq:final-budget}
\sum\limits_{k=0}^{K-1}\beta_k\overset{\eqref{eq:beta-budget}}{=}\frac{D_0^2-D_K^2}{t^2}\overset{\eqref{eq:gap-distance}}{\le} S(1-r^2).
\end{equation}

Applying the same Jensen argument with this smaller budget gives

\begin{equation}
\label{eq:jensen-final-distance}
r\overset{\eqref{eq:final-budget}}{\le}\left(\frac{S(1-r^2)-K}{S(1-r^2)+K}\right)^K,
\qquad
r^2+\frac KS\frac{1+r^{1/K}}{1-r^{1/K}}\le1.
\end{equation}

Since each $\beta_k>1$, we have $S(1-r^2)>K$.
The left-hand side of the second inequality is strictly increasing on $[0,1)$, equals $K/S<1$ at $r=0$, and tends to infinity as $r\to1$.
Hence there is a unique $r_\star\in(0,1)$ at which it equals $1$, and \eqref{eq:jensen-final-distance} implies $r\le r_\star$.
Moreover, $r_\star<B=((S-K)/(S+K))^K$, since substituting $B$ gives $1+B^2>1$.
Thus \eqref{eq:jensen} is not the exact worst-case factor for \BPM{} at fixed $S$ and $K$, although its explicit form remains useful. For $K=1$, the root equation simplifies to $(1-r_\star)^2=1/S$, so $r_\star=1-t/D_0$ recovers the segment bound. Sharpness of the implicit bound for general $K$ remains open.
\end{remark}

\subsection{From objective decrease to stationarity}
\label{sec:stationarity}

Each nonterminal step provides a stationarity certificate at its returned point. 
Indeed, the implicit formulation \eqref{eq:view-implicit} gives $g^k\in\partial f(x^{k+1})$,
and hence $m(x^{k+1})\le\norm{g^k}$.
Moreover, the objective decrease identity \eqref{eq:drop} in \Cref{thm:identity} gives

\[
\Delta_k-\Delta_{k+1} \overset{\eqref{eq:drop}}{=} t_k\norm{g^k} + D_f^{g^k}(x^k,x^{k+1}) \overset{\eqref{eq:bregman-definition}}{\ge} t_k\norm{g^k}.
\]

Thus a bound on the preceding objective gap can be transferred directly to a bound on stationarity at the returned point.

\begin{corollary}[Stationarity from objective decrease]\label{cor:geometric-stationarity}
Under \assref{ass:convex}, \assref{ass:attained}, and \assref{ass:iteration}, for every $K\ge1$ and arbitrary positive radii,
\begin{equation}\label{eq:local-stationarity}
 m(x^K)
 \le
 \frac{\Delta_{K-1}-\Delta_K}{t_{K-1}}
 \le
 \frac{\Delta_{K-1}}{t_{K-1}}.
\end{equation}
For a constant radius $0<t<D_0$, combining this with \eqref{eq:frozen} gives
\begin{equation}\label{eq:geometric-stationarity}
 m(x^K)
 \le
 \frac{\Delta_0}{t}
 \left(1-\frac{t}{D_0}\right)^{K-1},
 \qquad K\ge1.
\end{equation}
At and after the first optimal output, $m(x^K)=0$. 
For differentiable $f$, $m(x^K)=\norm{\nabla f(x^K)}$, so \eqref{eq:geometric-stationarity} also gives a gradient-norm bound without assuming a Lipschitz gradient.
\end{corollary}

The proof is given in \Cref{app:stationarity-transfer}.
 
This does not imply that $m(x^k)$ contracts by a uniform factor strictly below one at every iteration; the monotonicity in \Cref{prop:minimum-subgradient} still applies.

\paragraph{Refining the stationarity bound over several iterations.}

\Cref{cor:geometric-stationarity} uses only the last step. 
A stronger estimate follows by summing objective decrease over iterations $j,\ldots,K-1$. 
By the monotonicity of the radial subgradient norms in \Cref{prop:certificate-monotone}, for $j\le i<K$,

\[
\norm{g^i}\overset{\eqref{eq:acute}}{\ge}\norm{g^{K-1}}\ge m(x^K).
\]

Hence the decrease accumulated from iteration $j$ to $K$ controls the final stationarity measure through the total radius used over those iterations. 
Combining this bound over iterations $j,\ldots,K-1$ with the objective-gap bounds above yields the refined stationarity bounds below.

\begin{theorem}[Refined stationarity bounds]\label{thm:refined-stationarity}

Under \assref{ass:convex}, \assref{ass:attained}, and \assref{ass:iteration}, for arbitrary positive radii and every $K\ge1$,
\begin{equation}\label{eq:suffix-stationarity}
 m(x^K) \le \min_{0\le j<K} \frac{\Delta_j-\Delta_K}{\sum\limits_{i=j}^{K-1}t_i} \le \min_{0\le j<K} \frac{\Delta_j}{\sum\limits_{i=j}^{K-1}t_i}.
\end{equation}
For a constant radius $0<t<D_0$, set $S=D_0^2/t^2$ and suppose the first $K$ iterations are nonterminal. 
Then $K<S$ and
\begin{equation}\label{eq:jensen-stationarity}
 m(x^K) \le \frac{\Delta_0}{t} \left(\frac{S-K+1}{S+K-1}\right)^{K-1} \le \frac{\Delta_0}{t} \exp\left(-\frac{2(K-1)^2}{S}\right).
\end{equation}

For integers $0\le j<S$, define\begin{equation}\label{eq:stationarity-value-envelope}
 E_j \coloneq
 \min\left\{
  \left(1-\frac{t}{D_0}\right)^j,
  \left(\frac{S-j}{S+j}\right)^j
 \right\},
\end{equation}

Then
\begin{equation}\label{eq:optimized-stationarity}
 m(x^K)
 \le
 \frac{\Delta_0}{t}
 \min_{0\le j<K}
 \frac{E_j}{K-j}.
\end{equation}
When $f$ is differentiable, the same bounds apply to $\norm{\nabla f(x^K)}$.
\end{theorem}

The proof is given in \Cref{app:refined-stationarity}. 
The bound over iterations $j,\ldots,K-1$ \eqref{eq:suffix-stationarity} uses the actual gaps $\Delta_j$ and $\Delta_K$, whereas \eqref{eq:optimized-stationarity} replaces them by the objective-gap bounds derived above. 
In \eqref{eq:optimized-stationarity}, the index $j$ determines the first iteration included in the bound and is used only in the analysis, it does not affect the \BPM{} iteration. 
Choosing $j=0$ uses the full radius sum and gives $\Delta_0/(Kt)$, while choosing $j=K-1$ reduces to the single step estimate with the better of the geometric and Jensen bounds for $\Delta_{K-1}$. 
Intermediate values of $j$ trade a smaller objective gap later in the trajectory against a shorter remaining radius sum.

\paragraph{Shrinking radii.}

For variable radii, convergence of the objective gap alone need not imply convergence to stationarity. Indeed, in \eqref{eq:local-stationarity}, the denominator may shrink at the same rate as the objective decrease. 
The second schedule in \Cref{ex:radius-outcomes}, with
$f(x)=|x|$, $x^0=1$, and $t_k=2^{-k-1}$, satisfies
\[
 D_k=\Delta_k=2^{-k},
 \qquad
 m(x^k)=1
 \quad\text{for every finite }k.
\]
Thus both the distance to the solution set and the objective gap converge geometrically to zero, while the minimum subgradient norm remains constant.

\subsection{Symmetric Bregman distance to a minimizer}
\label{sec:primal-dual}

The exact progress identity \eqref{eq:identity}, evaluated at a minimizer $x^\star\in\X$, contains the term
\[
 f(u)-\fs+D_f^g(x^\star,u).
\]
We interpret this term as the symmetric Bregman distance to a minimizer and bound it using the radial subgradient estimates developed above. 
In the dual formulation of \Cref{prop:four-views}, the radial subgradient $g$ is also a dual solution associated with the exact ball subproblem. 
Since $0\in\partial f(x^\star)$, this is the symmetric Bregman distance associated with the subgradients $g$ at $u$ and $0$ at $x^\star$ \citep[Definition~2.2]{Burger2015Bregman}. It also admits the following representation in terms of the convex conjugate.

\begin{proposition}[Symmetric Bregman distance and its dual representation]\label{prop:primal-dual-gap}
Under \assref{ass:convex} and \assref{ass:attained}, fix $x^\star\in\X$, let $u\in\dom\partial f$, and choose $g\in\partial f(u)$. 
The two Bregman distances for the conjugate satisfy
\begin{align}
 D_{f^*}^{u}(0,g)
 &:=f^*(0)-f^*(g)+\ip ug=f(u)-\fs,\label{eq:pd-reverse}\\
 D_{f^*}^{x^\star}(g,0)
 &:=f^*(g)-f^*(0)-\ip{x^\star}g
 =D_f^g(x^\star,u)\ge0.\label{eq:pd-forward}
\end{align}
Define their sum, the symmetric Bregman distance to $x^\star$, by
\begin{equation}\label{eq:pd-gap}
 \mathcal G_{x^\star}(u,g)
 :=D_{f^*}^{u}(0,g)+D_{f^*}^{x^\star}(g,0)
 =\ip g{u-x^\star}.
\end{equation}
Equivalently, $\mathcal G_{x^\star}(u,g)=D_f^g(x^\star,u)+D_f^0(u,x^\star)$.
Then
\begin{equation}\label{eq:pd-sandwich}
 0\le f(u)-\fs\le\mathcal G_{x^\star}(u,g)
 \le\norm g\,\norm{u-x^\star}.
\end{equation}
For the exact stopped iteration of \assref{ass:iteration}, set
$x^\star=\Proj_{\X}(x^0)$ and $\mathcal G_K = \mathcal G_{x^\star}(x^K,g^{K-1})$, where $g^{K-1}\in\partial f(x^K)$ is the radial subgradient associated with the step from $x^{K-1}$ to $x^K$. 
At and after termination, we use the convention $g^{K-1}=0$.
For $K\ge1$ and arbitrary positive radii,
\begin{equation}\label{eq:pd-suffix}
 \mathcal G_K\le D_0\min_{0\le j<K}
 \frac{\Delta_j-\Delta_K}{\sum\limits_{i=j}^{K-1}t_i}.
\end{equation}
For a constant radius $0<t<D_0$ and the first $K$ iterations, all assumed nonterminal,
\begin{equation}\label{eq:pd-optimized}
 \mathcal G_K\le\frac{D_0\Delta_0}{t}
 \min_{0\le j<K}\frac{E_j}{K-j},
\end{equation}
where $E_j$ is defined in \eqref{eq:stationarity-value-envelope}.
At and after termination, $\mathcal G_K=0$. 
Since the two Bregman distances for the conjugate are nonnegative, each is also bounded by the right-hand sides of \eqref{eq:pd-suffix} and \eqref{eq:pd-optimized}.
\end{proposition}

The proof is given in \Cref{app:primal-dual-gap}. 
The quantity $\mathcal G_{x^\star}(u,g)$ is exactly
$f(u)-\fs+D_f^g(x^\star,u)$ from the exact progress identity \eqref{eq:identity}.
For the iterate $x^K$, \eqref{eq:pd-sandwich} gives
$\mathcal G_K\le D_0\norm{g^{K-1}}$.
The proof of \Cref{thm:refined-stationarity} bounds $\norm{g^{K-1}}$ by the same quantity for iterations $j,\ldots,K-1$ appearing in \eqref{eq:suffix-stationarity}, which gives \eqref{eq:pd-suffix}. The fixed-radius bound \eqref{eq:pd-optimized} follows in the same way from \eqref{eq:optimized-stationarity}.

\paragraph{Interpretation of the symmetric Bregman distance.}
The quantity $\mathcal G_{x^\star}(u,g)$ depends on both a reference minimizer $x^\star$ and a subgradient $g$, so it is primarily an analytical measure rather than a stopping criterion based only on the returned point. 
\phantomsection\label{concept:fenchel-young}
It is also different from the Fenchel--Young gap $f(u)+f^*(g)-\ip ug$, which vanishes whenever $g\in\partial f(u)$ and therefore does not measure progress toward optimality here.

A vanishing $\mathcal G_{x^\star}(u,g)$ also does not imply that the subgradient $g$ converges to zero. 
In the shrinking-radius example of \Cref{ex:radius-outcomes}, with $f(x)=|x|$ and $x^k=2^{-k}$, the radial subgradients satisfy $g^{k-1}=1$, while $\mathcal G_0(x^k,1)=2^{-k}\to0$. 
Thus convergence of the symmetric Bregman distance should be distinguished from convergence to stationarity.

\section{Trajectory convergence and finite termination under arbitrary radii}
\label{sec:termination}

The behavior of an exact \BPM{} trajectory depends strongly on whether the total requested radius is finite or infinite. 
The following theorem gives the complete alternative for positive radii without assuming that a minimizer exists.
The rest of the section establishes the geometric argument behind finite termination and then examines what can happen when the radii are summable.

\begin{theorem}[General convergence under arbitrary radii]\label{thm:general-convergence}
Let $f:\R^d\to\R\cup\{+\infty\}$ satisfy \assref{ass:convex}, with $d\ge1$, and use the exact stopped iteration of \assref{ass:iteration} from $x^0\in\dom f$.
For any positive radius sequence $(t_k)$, exactly one of the following cases applies.

\begin{enumerate}[label=(\roman*)]
\item If $\sum_k t_k<\infty$, then the trajectory has finite total length and
\[
 x^k\to\bar x\in\dom f,\qquad f(x^k)\downarrow f(\bar x).
\]
The limit $\bar x$ need not be a minimizer.

\item If $\sum_k t_k=\infty$ and a minimizer exists, then the iteration attains a minimizer after finitely many iterations.

\item If $\sum_k t_k=\infty$ and no minimizer exists, then
\[
 f(x^k)\downarrow\inf f,\qquad \norm{x^k}\to\infty,
\]
including when $\inf f=-\infty$.
\end{enumerate}
In particular,
\begin{equation}\label{eq:general-value-convergence}
 \sum\limits_{k=0}^{\infty} t_k=\infty
 \quad\Longrightarrow\quad
 f(x^k)\downarrow\inf f.
\end{equation}
Moreover, for a prescribed positive radius sequence in any fixed finite dimension $d\ge1$, the condition $\sum_k t_k=\infty$ is necessary for $f(x^k)\downarrow\inf f$ to hold for every proper closed convex objective and every starting point in its domain.
\end{theorem}

The proof is given in \Cref{app:general-convergence}. 
The theorem gives a sharp distinction between summable and nonsummable radii. 
Nonsummable radii guarantee convergence of the objective values to $\inf f$. When a minimizer exists, they further force finite termination. 
A summable radius sequence cannot guarantee value convergence uniformly over all convex objectives and starting points, although it may still lead to asymptotic or finite convergence to a minimizer on a particular problem.

\subsection{Self-contraction and finite termination}

We now turn to the finite termination statement in \Cref{thm:general-convergence} and sharpen it quantitatively in fixed dimension.
For any nonterminal step $j$ and any later iterate $x^m$, $m>j$, the monotonicity of the objective implies that $x^m$ belongs to the sublevel set used in the projection formulation of the $j$-th step. 
Applying the corresponding projection inequality gives \eqref{eq:future-fejer}, and hence the trajectory is self-contracted. 
In finite dimension, bounded self-contracted sequences have finite total length. 
Since every nonterminal \BPM{} step satisfies $\norm{x^{k+1}-x^k}=t_k$, this length bound directly controls the total radius spent before termination.

Self-contraction and finite length are classical for proximal sequences. 
See \citet[Proposition 4.16 and Theorem 4.17]{DaniilidisEtAl2015} and \citet[Corollary 3.8 and Section 3.4]{BohmDaniilidis2020}. 
Each nonterminal brox step is also a proximal step with a positive parameter, so those results give an alternative route to the same conclusion. 
We give a direct argument in the radius notation used here.

\begin{definition}[Self-contracted sequence]\label{def:self-contracted}
A finite or infinite sequence $(z^k)$ is self-contracted if $\norm{z^j-z^m}\le\norm{z^i-z^m}$ whenever $i\le j\le m$.
\end{definition}

\begin{lemma}[Distance decrease toward later iterates]\label{lem:self-contracted}
Under \assref{ass:convex} and \assref{ass:iteration}, every finite prefix consisting of nonterminal outputs is self-contracted. More precisely, for $j<m$ in that prefix,
\begin{equation}\label{eq:future-fejer}
 \norm{x^{j+1}-x^m}^2\le\norm{x^j-x^m}^2-t_j^2.
\end{equation}
\end{lemma}

The proof is given in \Cref{proof:termination:1}.

\begin{lemma}[Finite length of bounded self-contracted sequences]\label{lem:length}
There is a finite constant $C_d$, depending only on dimension, such that every bounded self-contracted sequence in $\R^d$ has total length at most $C_d$ times its diameter. One may take
\[
 C_1=2,\qquad C_d=4\sqrt d(1+4\sqrt d)^d\quad(d\ge2).
\]
\end{lemma}

This is an instance of the Euclidean self-contraction theory of \citet{DaniilidisEtAl2015}. 
For completeness, \Cref{app:length} gives a discrete proof using a finite spherical net. 
Related finite-net arguments appear in \citet[Theorem 4.5 and Remark 4.6, accepted manuscript]{DaniilidisDevilleDurand2018}. 
The explicit constant is sufficient here and is not claimed to be sharp.

Together, these results bound the total radius that can be spent along any sequence of nonterminal steps.

\begin{theorem}[Bound on the sum of radii before termination]\label{thm:dimension}
Under \assref{ass:convex}, \assref{ass:attained}, and \assref{ass:iteration}, every prefix of $K$ nonterminal iterations satisfies
\begin{equation}\label{eq:length-budget}
 \sum\limits_{k=0}^{K-1}t_k\le2C_dD_0.
\end{equation}
For a constant radius and $D_0>0$,
\begin{equation}\label{eq:dimension-count}
 \left\lceil\frac{D_0}{t}\right\rceil
 \le N_t\le\min\left\{
 \left\lceil\frac{D_0^2}{t^2}\right\rceil,
 1+\left\lfloor\frac{2C_dD_0}{t}\right\rfloor
 \right\}.
\end{equation}
Consequently, for fixed dimension, the worst-case complexity is $\Theta_d(D_0/t)$ as $D_0/t\to\infty$.
\end{theorem}

The proof is left to \Cref{proof:termination:2}. 
Equation \eqref{eq:length-budget} makes the finite-termination statement in \Cref{thm:general-convergence}(ii) quantitative: if the run remained nonterminal while $\sum_k t_k=\infty$, its partial radius sums would eventually exceed $2C_dD_0$. For a constant radius, the same bound gives a linear dependence on $D_0/t$ in fixed dimension, improving the dimension-free quadratic upper bound.

The next result adds the converse and characterizes exactly which prescribed radius sequences force finite termination on every problem with a minimizer.

\begin{theorem}[Universal criterion for finite termination]\label{thm:criterion}
Fix a positive radius sequence $(t_k)$ and a finite dimension $d\ge1$. The following are equivalent:
\begin{enumerate}[label=(\roman*)]
\item $\sum_{k=0}^\infty t_k=\infty$;
\item for every $f$ satisfying \assref{ass:convex} and \assref{ass:attained} and every $x^0\in\dom f$, the stopped exact iteration attains a minimizer after finitely many iterations.
\end{enumerate}
\end{theorem}

The proof is given in \Cref{proof:termination:3}. 
Sufficiency follows from \eqref{eq:length-budget}: on any particular problem, termination must occur before the partial radius sum exceeds $2C_dD_0$.
The converse shows that no summable positive schedule has this guarantee uniformly over all convex objectives and starting points. 
For example, the harmonic schedule $t_k=1/(k+1)$ forces finite termination even though $\sum_k t_k^2<\infty$.

The bound \eqref{eq:length-budget} controls only the nonterminal part of the trajectory. 
Under a constant radius and nearest-minimizer selection at termination, the terminal displacement can also be bounded, yielding a bound on the full trajectory length.

\begin{proposition}[Trajectory length including the terminal step]\label{prop:general-radius-cap}
Under \assref{ass:convex}, \assref{ass:attained}, and \assref{ass:iteration}, suppose that the nearest-minimizer selection at termination of \Cref{def:projected} is used, the radius is constant, $t_k\equiv t>0$, and $D_0>0$.
Let
\[
 J=\left\lceil\frac{D_0^2}{t^2}\right\rceil-1,
 \qquad \Psi_t(D_0)=Jt+\sqrt{D_0^2-Jt^2}.
\]
Then the full trajectory length $S_t$, including the terminal displacement, satisfies
\[
 S_t
 \le
 \Psi_t(D_0)
 \le
 \frac{D_0^2}{t}+\frac{t}{4}.
\]
If $D_0^2/t^2$ is an integer, then $Jt+\sqrt{D_0^2-Jt^2}=D_0^2/t$. If $t\ge D_0$, the first step is terminal and nearest-minimizer selection at termination gives $S_t=D_0$.
\end{proposition}

The proof is given in \Cref{app:general-radius-cap}.  The nearest-minimizer selection at termination is needed to control the final displacement.

\subsection{Summable radii and trajectory convergence}

\Cref{thm:criterion} identifies nonsummability as the exact condition for finite termination uniformly over all problems with a minimizer. 
When the radii are summable, finite termination is no longer guaranteed. Nevertheless, \Cref{thm:general-convergence} shows that the trajectory still has finite length and converges. 
We now quantify this convergence by relating the remaining radius sum to the distance and objective error relative to the limiting point.
The proof of the following theorem is given in \Cref{app:arbitrary-radii}. 

\begin{theorem}[Finite length and tail bounds]\label{thm:arbitrary-radii}
Under \assref{ass:convex}, \assref{ass:attained}, and \assref{ass:iteration}, every stopped exact \BPM{} trajectory satisfies
\begin{equation}\label{eq:arbitrary-radius-convergence}
 \sum\limits_{k=0}^\infty\norm{x^{k+1}-x^k}<\infty,\qquad
 x^k\longrightarrow\bar x\in\dom f,\qquad
 f(x^k)\downarrow f(\bar x).
\end{equation}
If no step is terminal, then necessarily $\sum_{k=0}^\infty t_k<\infty$. Moreover, for any radial subgradient $g^0\in\partial f(x^1)$ from the first step,
\begin{equation}\label{eq:radius-tail-error}
 \norm{x^k-\bar x}
 \le
 \sum\limits_{j=k}^\infty t_j
 \quad(k\ge0),
 \qquad
 0\le f(x^k)-f(\bar x)
 \le
 \norm{g^0}\sum\limits_{j=k}^\infty t_j
 \quad(k\ge1).
\end{equation}
If $\sum_{k=0}^\infty t_k<\infty$ is assumed, then \assref{ass:attained} is unnecessary: \eqref{eq:arbitrary-radius-convergence} still holds, and on an infinite nonterminal run so does \eqref{eq:radius-tail-error}.
\end{theorem}

The preceding bounds use the remaining radius sum. 
The next corollary instead controls the actual length of the remaining trajectory in terms of the current distance to the solution set.

\begin{corollary}[Remaining trajectory length]\label{cor:tail-length}
Under \assref{ass:convex}, \assref{ass:attained}, and \assref{ass:iteration}, use the nearest-minimizer selection at termination of \Cref{def:projected}. Let $\bar x=\lim_k x^k$ and $L_k=\sum_{i=k}^\infty\norm{x^{i+1}-x^i}$. Then, for every $k\ge0$,
\begin{equation}\label{eq:tail-length-distance}
 \norm{x^k-\bar x}\le L_k\le2C_dD_k.
\end{equation}
In particular, if $D_k\le A\rho^k$ for some $A>0$ and $0<\rho<1$, then $\bar x\in\X$, and both $L_k$ and $\norm{x^k-\bar x}$ decay at least geometrically, with upper bound $2C_dA\rho^k$.
\end{corollary}

The proof is given in \Cref{app:tail-length}. 
Although summable radii give point and objective value convergence, the limit still depends on the radius schedule.
Even for the same objective, the trajectory may terminate at a minimizer, converge to one asymptotically, or converge to a nonoptimal point.

\begin{example}[Three behaviors under summable radii]\label{ex:radius-outcomes}
Let $f(x)=|x|$ and $x^0=D>0$. Consider the following three summable radius schedules.

\begin{enumerate}[label=(\roman*)]
\item If $t_k=D\,2^{-k}$, then the first radius equals $D$, so the first step reaches the minimizer $0$.

\item If $t_k=D\,2^{-k-1}$, then $x^k=D\,2^{-k}$ for every $k\ge0$. Hence no step is terminal, but $x^k\to0$.

\item If $t_k=D\,2^{-k-2}$, then
$x^k=D/2+D\,2^{-k-1}$ for every $k\ge0$. Hence the trajectory is again infinite, but now $x^k\to D/2$, which is not optimal.
\end{enumerate}

Thus, even for the same objective and starting point, summable radii can lead to finite termination, asymptotic convergence to a minimizer, or convergence to a nonoptimal point. 
These formulas follow from $x^{k+1}=\max\{x^k-t_k,0\}$.  
In the two infinite trajectories, both bounds in \eqref{eq:radius-tail-error} hold with equality, with $\norm{g^0}=1$.
\end{example}

\subsection{Dimension dependence of the linear bound}

The linear step bound in \eqref{eq:dimension-count} is a fixed dimensional result because its constant depends on $d$. 
This dependence is unavoidable. 
The linear bound cannot hold with a constant independent of dimension. 
Indeed, the construction below shows that, as the dimension increases, the number of iterations can scale like $D_0^2/t^2$ rather than $D_0/t$.

\begin{theorem}[Polyhedral lower bound for the step count]\label{thm:hard}
For every integer $n\ge2$, every $t>0$, and every $\epsilon\in(0,1/2)$, there is a finite convex piecewise-affine function on $\R^n$ with a unique minimizer such that \BPM{} performs exactly $n$ radius-$t$ iterations from $x^0=0$ and
\begin{equation}\label{eq:hard-distance}
 D_0^2=t^2[n+2(n-1)\epsilon].
\end{equation}
If $\epsilon=o(1/n)$, then $D_0^2/t^2=n+o(1)$ and
\[
 \frac{N_t}{\lceil D_0^2/t^2\rceil}\longrightarrow1
 \qquad\text{as }n\to\infty.
\]
Consequently, no constant independent of dimension can bound $N_t$ by $CD_0/t$ for all convex objectives.
\end{theorem}

The proof is given in \Cref{proof:termination:4}. 
The lower bound is asymptotic over increasing dimensions, while for each fixed dimension the linear bound in \Cref{thm:dimension} still applies. 
In this family, the preterminal trajectory has length $(n-1)t$. 
Consequently, any constant $L_d$ satisfying $\sum_{k<N_t-1}t\le L_dD_0$ uniformly over all \BPM{} trajectories in dimension $d$ must satisfy $L_d=\Omega(\sqrt d)$.

\section{Reformulations and preservation of BPM geometry}
\label{sec:reformulations}

Equivalent formulations of the same optimization problem need not induce the same \BPM{} trajectory, because \BPM{} depends on the geometry of Euclidean balls. 
We first identify transformations that preserve exact ball calls. 
We then study the epigraph reformulation, which preserves the underlying optimization problem but changes distances, and hence the radius geometry. 
The epigraph reformulation itself is classical \citep[Section 4.1.3]{BoydVandenberghe2004}.

\begin{proposition}[Transformations preserving exact ball calls]
\label{prop:oracle-transformations}
Under \assref{ass:convex}, the following transformations preserve exact \BPM{} iterations.

\begin{enumerate}[label=(\roman*)]
\item Let $x=a+sQy$, where $s>0$ and $Q$ is orthogonal. Exact \BPM{} iterations in $y$-space with radius $t$ correspond to exact \BPM{} iterations in $x$-space with radius $st$. In particular, translations and orthogonal changes of coordinates preserve the radii.

\item Let $\widetilde f=\chi\circ f$ on $\dom f$, with $\widetilde f=+\infty$ elsewhere, where $\chi$ is strictly increasing on the finite range of $f$ and $\widetilde f$ is proper, closed and convex. Then $f$ and $\widetilde f$ have the same exact ball-minimizer sets at every center and radius, and the same solution set. Hence, under the same selections, they generate the same fixed-radius trajectories and termination counts.
\end{enumerate}
\end{proposition}
The proof is given in \Cref{app:reformulations}. Positive affine changes of the objective are a special case of part~(ii). Although the exact ball calls are preserved, function gaps, subgradient norms, and hence radius rules based on them need not be. 
A general nonsingular linear change of variables falls outside this invariance: it maps Euclidean balls to ellipsoids and therefore does not preserve the same \BPM{} radius geometry.

\subsection{Epigraphical lifting and induced radii}
For $f$ satisfying \assref{ass:convex}, define on $\R^{d+1}$
\begin{equation}\label{eq:epigraph-objective}
 F(x,s)\coloneq s+\delta_{\operatorname{epi}f}(x,s),\qquad
 \operatorname{epi}f=\{(x,s):f(x)\le s\},
\end{equation}
where $\delta_C$ denotes the indicator function of a set $C$.
Then $F$ is proper, closed, and convex, with $\inf F=\inf f$ and, under attainment, $\argmin F=\X\times\{\fs\}$.
We apply \BPM{} to $F$ using Euclidean balls in $\R^{d+1}$, initialized at the graph point $z^0=(x^0,f(x^0))$ with $x^0\in\dom f$. 
This initialization therefore requires the value $f(x^0)$.

\begin{theorem}[Epigraphical step correspondence]
\label{thm:epigraph-steps}
Every \BPM{} iterate on $F$ from the above initialization remains on the graph: let $z^k=(x^k,f(x^k))$.
At a nonterminal radius-$t_k$ step, define
\[
 \delta_k=f(x^k)-f(x^{k+1}),\qquad
 r_k=\norm{x^{k+1}-x^k}.
\]
Then $\delta_k>0$, $r_k>0$, and
\begin{equation}\label{eq:epigraph-radius}
 t_k^2=r_k^2+\delta_k^2,\qquad
 x^{k+1}\in\brox_f^{r_k}(x^k).
\end{equation}
Conversely, any nonterminal radius $r$ \BPM{} step $x\mapsto u$ for $f$ lifts to an exact step for $F$ from $(x,f(x))$ with radius $\sqrt{r^2+(f(x)-f(u))^2}$.
Moreover, if
$G^k=c_k(z^k-z^{k+1})\in\partial F(z^{k+1})$,
$c_k>0$, is a radial subgradient for the lifted step, then
\begin{equation}\label{eq:epigraph-certificate}
 \lambda_k=1-c_k\delta_k\in(0,1),\qquad
 g^k=\frac{c_k}{\lambda_k}(x^k-x^{k+1})\in\partial f(x^{k+1}),
 \qquad
 G^k=(\lambda_kg^k,1-\lambda_k).
\end{equation}
At a terminal lifted step, $x^{k+1}\in\X$, and the projected output is an exact \BPM{} output for $f$ with radius
$\sqrt{t_k^2-\Delta_k^2}\ge D_k$.
Conversely, a terminal radius $r$ step for $f$ lifts to a terminal step for $F$ with radius
$\sqrt{r^2+\Delta_k^2}$.
These terminal correspondences hold for any optimal output selected by the terminal step.
\end{theorem}

The proof is given in \Cref{app:epigraph-steps} and allows extended-valued $f$. 
The theorem shows that epigraphical lifting preserves projected \BPM{} trajectories after an induced change of radii, but does not preserve a prescribed radius schedule.
This differs from the cylindrical epigraph formulation in \eqref{eq:view-epigraph-cylinder}. 
There, only the horizontal displacement is constrained, $\norm{v-x}\le t$, whereas a Euclidean ball in the lifted space imposes $\norm{v-x}^2+(s-f(x))^2\le t^2$.
Thus part of the lifted radius is spent in the vertical direction, and the projected nonterminal radius is generally strictly smaller than $t$.
Initialization on the graph is essential. 
For example, if $f\equiv0$ and $z^0=(0,2)$, the lifted method may take purely vertical nonterminal steps even though the projected point is already optimal.

For the same center and radius $t$, the projected lifted output is feasible for the \BPM{} subproblem on $f$. 
Hence \BPM{} attains an objective value no larger than the projected lifted step. 
This comparison applies only to the current step. The two methods generally produce different next centers, so it does not imply an ordering of their subsequent trajectories.

\subsection{Convergence guarantees under epigraphical lifting}

Under attainment, define the lifted distance to the solution set
$R_k \coloneq \dist(z^k,\argmin F)$. 
Since $z^k=(x^k,f(x^k))$ and $\argmin F=\X\times\{\fs\}$,

\begin{equation}
\label{eq:epigraph-distance}
F(z^k)-\inf F\overset{\eqref{eq:epigraph-objective}}{=}\Delta_k,\qquad
R_k^2=D_k^2+\Delta_k^2.
\end{equation}

Hence the general \BPM{} bounds apply to $F$ with dimension $d+1$, initial distance $R_0$, and the radii used in the lifted space. 
In particular, for a constant lifted radius $t$, every sequence of nonterminal steps satisfies

\begin{equation}
\label{eq:epigraph-value}
\Delta_K\overset{\eqref{eq:jensen}}{\le}\Delta_0\exp\!\left(-\frac{2K^2t^2}{R_0^2}\right),\qquad
N_t^{\rm epi}\overset{\eqref{eq:squared-count}}{\le}\left\lceil \frac{R_0^2}{t^2}\right\rceil.
\end{equation}

Here $N_t^{\rm epi}$ denotes the termination count for lifted \BPM{}. 
In particular,
$R_0=\sqrt{D_0^2+\Delta_0^2}$, so these bounds depend on the lifted distance rather than on $D_0$ alone.

\begin{proposition}[Stationarity and Bregman-distance relations]
\label{prop:epigraph-measures}
For every $x\in\dom f$,
\begin{equation}\label{eq:epigraph-stationarity}
 \dist(0,\partial F(x,f(x))) = \frac{m(x)}{\sqrt{1+m(x)^2}},
\end{equation}
where the right-hand side is interpreted as one if $\partial f(x)=\varnothing$.
Consequently, a lifted stationarity bound gives
\[
 \dist(0,\partial F(x,f(x)))\le\varepsilon<1
 \quad\Longrightarrow\quad
 m(x)\le\frac{\varepsilon}{\sqrt{1-\varepsilon^2}}.
\]
By contrast, the bound
$\dist(0,\partial F(x,f(x)))\le1$
alone gives no finite upper bound on $m(x)$.
Under attainment, fix $x^\star\in\X$. 
At a nonterminal paired output $(u,f(u))$, let $G=(\lambda g,1-\lambda)$ be the paired subgradient from \eqref{eq:epigraph-certificate}. 
Then
\begin{equation}\label{eq:epigraph-dual-gap}
 \ip G{(u,f(u))-(x^\star,\fs)}
 =
 \lambda\mathcal G_{x^\star}(u,g)
 +(1-\lambda)(f(u)-\fs).
\end{equation}
The primal gap is unchanged,
\[
 F(u,f(u))-\inf F=f(u)-\fs,
\]
and the complementary Bregman distance satisfies
\[
 D_F^G((x^\star,\fs),(u,f(u)))
 =
 \lambda D_f^g(x^\star,u).
\]
\end{proposition}

The proof is given in \Cref{app:epigraph-measures}. Applying the earlier \BPM{} bounds to $F$ gives bounds in terms of the lifted distance $R_k$, related to the original quantities by \eqref{eq:epigraph-distance}. Stationarity and primal--dual quantities are instead related by \eqref{eq:epigraph-stationarity} and \eqref{eq:epigraph-dual-gap}.
The epigraphical lift also does not preserve additional regularity, $F$ is extended-valued and affine along feasible vertical rays, so smoothness or strong convexity of $f$ does not carry over to $F$.

\begin{theorem}[Summability and length bounds]
\label{thm:epigraph-budgets}
For an infinite nonterminal lifted iteration initialized at $z^0=(x^0,f(x^0))$, the lifted radii $t_k$ and the corresponding projected radii $r_k$ from \eqref{eq:epigraph-radius} satisfy
\[
 \sum\limits_{k=0}^{\infty} t_k<\infty
 \quad\Longleftrightarrow\quad
 \sum\limits_{k=0}^{\infty} r_k<\infty.
\]
Consequently, for any prescribed positive lifted radius sequence $(t_k)$, the convergence alternatives of \Cref{thm:general-convergence} also hold for the projected iterates: if $\sum_k t_k=\infty$, finite termination occurs when $f$ attains its minimum, while otherwise $f(x^k)\downarrow\inf f$ and $\norm{x^k}\to\infty$.
If $\sum_k t_k<\infty$, then $x^k\to\bar x\in\dom f$, and $f(x^k)\downarrow f(\bar x)$, where $\bar x$ need not be optimal.
Under attainment and a nonoptimal start, every sequence of nonterminal steps satisfies
\begin{equation}\label{eq:epigraph-length-budget}
 \sum\limits_{k=0}^{K-1}t_k
 \le
 2C_dD_0+\Delta_0-\Delta_K
 <
 2C_dD_0+\Delta_0.
\end{equation}
Hence, for a constant lifted radius $t$,
\begin{equation}\label{eq:epigraph-count}
 N_t^{\rm epi}\le\min\left\{
 \left\lceil\frac{D_0^2+\Delta_0^2}{t^2}\right\rceil,
 \left\lceil\frac{2C_dD_0+\Delta_0}{t}\right\rceil
 \right\}.
\end{equation}
If the nearest-minimizer selection at termination of \Cref{def:projected} is used, the total lifted trajectory length is at most
$2C_dD_0+\Delta_0$.
\end{theorem}

The proof is given in \Cref{app:epigraph-budgets}. 
The bound \eqref{eq:epigraph-length-budget} combines the projected trajectory-length bound \eqref{eq:length-budget} with the total decrease in objective value, which is why the term $2C_dD_0$ depends on the original dimension $d$ rather than the lifted dimension $d+1$.

\begin{example}[Fixed-radius counts under objective scaling]
\label{ex:epigraph-scaling}
Let $f(x)=a|x|$, $a>0$, and $x^0=D_0>0$. 
Ordinary \BPM{} has $N_t=\lceil D_0/t\rceil$. 
The lifted trajectory initialized at $(D_0,f(D_0))$ follows the line segment $s=ax$ from $(D_0,aD_0)$ to $(0,0)$, so
\[
 N_t^{\rm epi}
 =\left\lceil\frac{\sqrt{1+a^2}\,D_0}{t}\right\rceil.
\]
Hence $N_t^{\rm epi}/N_t$ can grow arbitrarily large as $a\to\infty$. 
\end{example}
Positive objective scaling leaves fixed-radius \BPM{} unchanged by \Cref{prop:oracle-transformations}, but changes the geometry of the epigraph lift. 
Thus it can change the lifted fixed-radius termination count by an arbitrarily large factor.
The verification is given in \Cref{app:epigraph-scaling}.

\section{Adaptive radius selection}
\label{sec:calibration}
The choice of radius can substantially affect the behavior of \BPM{}. 
A larger radius can produce a better first objective value yet require more iterations to reach a minimizer: \Cref{prop:radius-nonmonotone} gives $0<s<t$ with $N_s=2$ but $N_t=3$.
For geometric schedules $t_k=t\rho^k$ with $0<t<D_0$, \Cref{thm:geometric-factor} identifies the smallest successful factor $\rho_\star$.
At this factor the trajectory converges to a minimizer without reaching one in finitely many iterations, whereas any smaller factor leaves a nonoptimal limit.
These results show why radius selection must account for the trajectory that the radii generate.

We therefore study adaptive rules that choose the radius using information available at the current iterate, specifically subgradient or objective gap information. The guarantees below require no smoothness, strong convexity, or growth condition. 
Throughout, \assref{ass:convex}, \assref{ass:attained}, and \assref{ass:iteration} hold.

\subsection{Subgradient-based radius selection}

If $t_k\ge\theta D_k$ for some fixed $0<\theta<1$, \Cref{cor:squared,thm:segment} give

\[
D_{k+1}\overset{\eqref{eq:lower-fejer}}{\le}\sqrt{1-\theta^2}D_k,\qquad
\Delta_{k+1}\overset{\eqref{eq:segment}}{\le}(1-\theta)\Delta_k.
\]

Thus a radius comparable to the current distance to the solution set gives geometric contraction in both distance and objective gap. 
Since $D_k$ is generally unavailable, we next consider an adaptive rule based on an available subgradient. 
Its convergence rate follows by relating the resulting \BPM{} steps to the proximal-point method \citep{Rockafellar1976,TaylorHendrickxGlineur2017}.

\begin{theorem}[Subgradient-based radius selection]
\label{thm:certificate-calibration}
Suppose $x^0\in\dom\partial f$. At iteration $k$, choose an available
$h^k\in\partial f(x^k)$. If $h^k=0$, stop; otherwise set
$t_k=\tau\norm{h^k}$ for some $\tau>0$.
For the first $K\ge1$ iterations, all assumed nonterminal, let $g^k$ be a radial
subgradient and let $\lambda_{k+1}=t_k/\norm{g^k}$ be the corresponding
proximal parameter from \eqref{eq:prox-equivalence}. Then, for $k<K$,
\[
 \lambda_{k+1}\ge\tau,\qquad
 \Delta_K
 \le\frac{D_0^2}{4\sum\limits_{k=0}^{K-1}\lambda_{k+1}}
 \le\frac{D_0^2}{4\tau K}.
\]
Under the stopped convention of \assref{ass:iteration}, the final bound
remains valid after termination.
\end{theorem}

The proof is given in \Cref{app:general-calibration}. 
The rule is implementable whenever a subgradient is available at the current iterate. 
In particular, the radial subgradient from \Cref{lem:certificate} can be reused at the next iteration if the oracle returns it. 
A zero subgradient provides a stopping certificate, whereas a nonzero subgradient may still be returned at a nonsmooth minimizer. 
The constant $1/4$ is inherited from the proximal point estimate of \citet[Theorem~4.1]{TaylorHendrickxGlineur2017} and is also sharp for this radius rule: the one-dimensional absolute-value example in \Cref{app:general-calibration} attains equality in both bounds of \Cref{thm:certificate-calibration}.

\subsection{Objective gap based radius selection}
If $\fs$ is known, the current objective gap can be used directly to choose the radius. For fixed $\tau>0$ and $\alpha>0$, set $t_k=\tau\Delta_k^\alpha$ whenever $\Delta_k>0$.

\begin{theorem}[Objective gap based radius selection]
\label{thm:gap-power}
Suppose $\Delta_0>0$ and, at each nonoptimal iterate, choose $t_k=\tau\Delta_k^\alpha$. 
Then the resulting \BPM{} iterates satisfy
\begin{equation}\label{eq:gap-power-rate}
 \Delta_K\le\left(\Delta_0^{-\alpha}+\frac{\alpha\tau K}{D_0}\right)^{-1/\alpha}.
\end{equation}
Consequently, for $0<\varepsilon<\Delta_0$, it suffices to take
\[
 K\ge\left\lceil\frac{D_0}{\alpha\tau}
 (\varepsilon^{-\alpha}-\Delta_0^{-\alpha})\right\rceil
\]
to obtain $\Delta_K\le\varepsilon$.
\end{theorem}

The proof is given in \Cref{proof:calibration:1}. In particular, $\alpha=1$ gives an $O(K^{-1})$ rate, while $\alpha=1/2$ gives $O(K^{-2})$. 
These rates arise from different radius schedules, rather than from accelerating a fixed ball-oracle iteration.
For a fixed initial radius $t_0$, the rule can be written as $t_k=t_0(\Delta_k/\Delta_0)^\alpha$, so a smaller $\alpha$ decreases the radius more slowly as the objective gap closes.

\begin{proposition}[Radius selection with a lower bound]
\label{prop:gap-extensions}
Suppose the assumptions of \Cref{thm:gap-power} hold, but only a finite lower bound
$\ell\le\fs$ is known. If the radius is chosen as
$t_k=\tau(f(x^k)-\ell)^\alpha$, then the objective-gap bound
\eqref{eq:gap-power-rate} remains valid. Moreover, if $\ell<\fs$ and
$D_0>0$, a terminal step occurs within
\[
 \left\lceil
 \frac{D_0^2}{\tau^2(\fs-\ell)^{2\alpha}}
 \right\rceil
\]
iterations.
\end{proposition}

The proof is given in \Cref{proof:calibration:2}.  
A valid lower bound keeps the radius positive at every nonoptimal point, an upper estimate need not. 
For $f(x)=x^2/2$ and $u>0$, the rule $t=\tau\max\{f(x)-u,0\}^\alpha$ gives $t=0$ whenever $f(x)\le u$, even if $x$ is not optimal.

\section{Relaxed \BPM{} iterations}
\label{sec:relaxation}

Besides choosing the radius, one can also modify how far the next center moves toward the projected \BPM{} output. 
Let $T_t$ denote the projected selection from \Cref{def:projected}. 
We consider the relaxed update
$R_{t,\lambda}(x) \coloneq x+\lambda(T_t(x)-x)$,
where $t>0$ and $0<\lambda<2$.
The convergence analysis relies on the cutter property of $T_t$ \citep{CegielskiCensor2012}. 
To continue the iteration, the relaxed center must remain in $\dom f$. This is automatic for $0<\lambda\le1$, and for all $0<\lambda<2$ when $f$ is finite-valued. 
For extended-valued objectives and $\lambda>1$, this must be checked separately.

\begin{theorem}[Convergence of relaxed \BPM{}]
\label{thm:general-relaxed}
Assume \assref{ass:convex} and \assref{ass:attained}. Fix $t>0$ and $0<\lambda<2$, and let
$x^{k+1}=R_{t,\lambda}(x^k)$ from some $x^0\in\dom f$, assuming that every iterate remains in $\dom f$. Then $x^k$ converges to a minimizer. Moreover, for every $x\in\dom f$ and $z\in\X$,
\begin{equation}\label{eq:general-relaxed-fejer}
 \norm{R_{t,\lambda}(x)-z}^2
 \le
 \norm{x-z}^2
 -\lambda(2-\lambda)\min\{t,D(x)\}^2.
\end{equation}
Once $D(x^k)\le t$, the distance to the solution set contracts as
$D(x^{k+1})\le |1-\lambda|D(x^k)$.
\end{theorem}

The proof is given in \Cref{app:general-relaxation}.
When $\lambda\ne1$, an optimal ball output does not necessarily make the next relaxed iterate optimal, so the relaxed sequence may continue after the first such output. 
For $\lambda>1$, the objective values need not be monotone.

\section{Discussion}
\label{sec:discussion}

Taken together, the results in this paper reveal two distinct roles of the radius sequence. 
In any fixed finite dimension, nonsummable radii force finite termination whenever a minimizer exists, whereas summable radii may lead to finite termination, asymptotic convergence to a minimizer, or convergence to a nonoptimal point. 
For a constant radius, the worst-case step count is linear in $D_0/t$ when the dimension is fixed, but this dimension dependence cannot be removed: as the dimension grows, the quadratic dependence on $D_0/t$ can be asymptotically attained.

Two sharpness questions remain open. 
The first concerns the dependence on dimension in the trajectory-length bound. 
\Cref{thm:hard} shows that the corresponding constant must grow at least as $\sqrt d$, while the upper bound in \Cref{lem:length} leaves a gap in its dependence on $d$. 
The second concerns the optimal worst-case objective-gap bound for \BPM{}.
The Jensen estimate \eqref{eq:jensen} is sharp for its scalar relaxation, but \Cref{rem:jensen-final-distance} shows that retaining the gap--distance relation gives a strictly smaller factor for fixed $S$ and $K$.
Determining the optimal bound for convex \BPM{} trajectories, and the asymptotic sharpness of its dependence on $S$ and $K$, remains open.

The complexity bounds in this paper are stated in terms of exact ball calls. 
Converting them into computational work would additionally require a model for solving the ball subproblems and obtaining the required subgradients.

\section*{Acknowledgments and declarations}
\paragraph{Funding.}
This work was supported by funding from King Abdullah University of Science and Technology (KAUST): i) KAUST Baseline Research Scheme, ii) Center of Excellence for Generative AI (award no. 5940).

\paragraph{Competing interests.}
The authors have no competing interests to declare that are relevant to the content of this article.

\paragraph{Use of AI tools.}
The authors used OpenAI's ChatGPT and Codex to assist with manuscript drafting and revision, mathematical exploration, and checks of proofs, calculations, and references. The authors reviewed and verified the resulting material and take full responsibility for the final manuscript.

\bibliographystyle{plainnat}
\bibliography{references}

@inproceedings{CarmonEtAlBallOracle,
  author    = {Yair Carmon and Arun Jambulapati and Qijia Jiang and Yujia Jin and Yin Tat Lee and Aaron Sidford and Kevin Tian},
  title     = {Acceleration with a Ball Optimization Oracle},
  booktitle = {Advances in Neural Information Processing Systems},
  volume    = {33},
  year      = {2020},
  url       = {https://arxiv.org/abs/2003.08078}
}

@article{CegielskiCensor2012,
  author  = {Andrzej Cegielski and Yair Censor},
  title   = {Extrapolation and Local Acceleration of an Iterative Process for Common Fixed Point Problems},
  journal = {Journal of Mathematical Analysis and Applications},
  volume  = {394},
  number  = {2},
  pages   = {809--818},
  year    = {2012},
  doi     = {10.1016/j.jmaa.2012.04.072},
  url     = {https://doi.org/10.1016/j.jmaa.2012.04.072}
}

@article{DaniilidisEtAl2015,
  author  = {Aris Daniilidis and Guy David and Estibalitz Durand-Cartagena and Antoine Lemenant},
  title   = {Rectifiability of Self-Contracted Curves in the {Euclidean} Space and Applications},
  journal = {Journal of Geometric Analysis},
  volume  = {25},
  number  = {2},
  pages   = {1211--1239},
  year    = {2015},
  url     = {https://arxiv.org/abs/1211.3227}
}

@article{Rockafellar1976,
  author  = {R. Tyrrell Rockafellar},
  title   = {Monotone Operators and the Proximal Point Algorithm},
  journal = {SIAM Journal on Control and Optimization},
  volume  = {14},
  number  = {5},
  pages   = {877--898},
  year    = {1976},
  url     = {https://doi.org/10.1137/0314056}
}

@misc{BPM,
  author = {Kaja Gruntkowska and Hanmin Li and Aadi Rane and Peter Richt{\'a}rik},
  title  = {The Ball-Proximal (= Broximal) Point Method: a New Algorithm, Convergence Theory, and Applications},
  note   = {Preprint, arXiv:2502.02002},
  year   = {2025},
  url    = {https://arxiv.org/abs/2502.02002}
}

@misc{NonEuclideanBPM,
  author = {Kaja Gruntkowska and Peter Richt{\'a}rik},
  title  = {Non-{Euclidean} Broximal Point Method: A Blueprint for Geometry-Aware Optimization},
  note   = {Preprint, arXiv:2510.00823},
  year   = {2025},
  url    = {https://arxiv.org/abs/2510.00823}
}

@misc{StabilizedPPM,
  author = {Hanmin Li and Kaja Gruntkowska and Peter Richt{\'a}rik},
  title  = {Stabilized Proximal Point Method via Trust Region Control},
  note   = {Preprint, arXiv:2604.02943},
  year   = {2026},
  url    = {https://arxiv.org/abs/2604.02943}
}

@misc{BroximalAlignment,
  author = {Kaja Gruntkowska and Hanmin Li and Xun Qian and Peter Richt{\'a}rik},
  title  = {Broximal Alignment for Global Non-Convex Optimization},
  note   = {Preprint, arXiv:2604.13483},
  year   = {2026},
  url    = {https://arxiv.org/abs/2604.13483}
}

@misc{LocalLMO,
  author = {Peter Richt{\'a}rik and Kaja Gruntkowska and Hanmin Li},
  title  = {Local {LMO}: Constrained Gradient Optimization via a Local Linear Minimization Oracle},
  note   = {Preprint, arXiv:2605.08850},
  year   = {2026},
  url    = {https://arxiv.org/abs/2605.08850}
}

@article{TaylorHendrickxGlineur2017,
  author  = {Adrien B. Taylor and Julien M. Hendrickx and Fran{\c{c}}ois Glineur},
  title   = {Exact Worst-Case Performance of First-Order Methods for Composite Convex Optimization},
  journal = {SIAM Journal on Optimization},
  volume  = {27},
  number  = {3},
  pages   = {1283--1313},
  year    = {2017},
  url     = {https://doi.org/10.1137/16M108104X},
  note    = {Also available at \url{https://arxiv.org/abs/1512.07516}}
}

@article{BohmDaniilidis2020,
  author  = {Axel B{\"o}hm and Aris Daniilidis},
  title   = {Ubiquitous Algorithms in Convex Optimization Generate Self-Contracted Sequences},
  journal = {Journal of Convex Analysis},
  volume  = {29},
  number  = {1},
  pages   = {119--128},
  year    = {2022},
  url     = {https://journalofconvexanalysis.com/articles/jca29006/}
}

@article{DaniilidisDevilleDurand2018,
  author  = {Aris Daniilidis and Robert Deville and Estibalitz Durand-Cartagena},
  title   = {Metric and Geometric Relaxations of Self-Contracted Curves},
  journal = {Journal of Optimization Theory and Applications},
  volume  = {182},
  number  = {1},
  pages   = {81--109},
  year    = {2019},
  doi     = {10.1007/s10957-018-1408-0},
  url     = {https://doi.org/10.1007/s10957-018-1408-0},
  note    = {Theorem numbering follows the \href{https://oai.e-spacio.uned.es/server/api/core/bitstreams/f5a61760-3508-45ac-8f01-ff9dfdf9506e/content}{accepted manuscript}}
}

@article{Bregman1967,
  author  = {L. M. Bregman},
  title   = {The Relaxation Method of Finding the Common Point of Convex Sets and Its Application to the Solution of Problems in Convex Programming},
  journal = {USSR Computational Mathematics and Mathematical Physics},
  volume  = {7},
  number  = {3},
  pages   = {200--217},
  year    = {1967},
  url     = {https://doi.org/10.1016/0041-5553(67)90040-7}
}

@book{BoydVandenberghe2004,
  author    = {Stephen Boyd and Lieven Vandenberghe},
  title     = {Convex Optimization},
  publisher = {Cambridge University Press},
  year      = {2004},
  url       = {https://www.stanford.edu/~boyd/cvxbook/bv_cvxbook.pdf}
}

@inproceedings{CarmonEtAl2021InsideBall,
  author    = {Y. Carmon and A. Jambulapati and Y. Jin and A. Sidford},
  title     = {Thinking Inside the Ball: Near-Optimal Minimization of the Maximal Loss},
  booktitle = {Proceedings of COLT},
  series    = {PMLR},
  volume    = {134},
  pages     = {866--882},
  year      = {2021},
  url       = {https://proceedings.mlr.press/v134/carmon21a.html}
}

@article{Burger2015Bregman,
  author = {Martin Burger},
  title = {{Bregman} Distances in Inverse Problems and Partial Differential Equations},
  journal = {arXiv preprint arXiv:1505.05191},
  year = {2015},
  url = {https://arxiv.org/abs/1505.05191}
}
\clearpage
\appendix
\part*{Appendix}
\tableofcontents
\clearpage

\section{Notation}
\label{app:notation}
The symbols below are grouped by their role in the analysis. Quantities involving the solution set or the optimal value assume attainment unless stated otherwise. The last column points to the defining result or discussion.

\begingroup
\small\setlength{\tabcolsep}{5pt}\renewcommand{\arraystretch}{1.22}
\begin{longtable}{@{}>{\raggedright\arraybackslash}p{.18\textwidth}>{\raggedright\arraybackslash}p{.63\textwidth}>{\raggedright\arraybackslash}p{.14\textwidth}@{}}
\caption{Notation for exact Euclidean BPM.}\label{tab:notation}\\
\toprule
\textbf{Symbol} & \textbf{Meaning or definition} & \textbf{Location}\\
\midrule
\endfirsthead
\multicolumn{3}{@{}l}{\textbf{Table \thetable. Notation for exact Euclidean BPM (continued)}}\\[5pt]
\toprule
\textbf{Symbol} & \textbf{Meaning or definition} & \textbf{Location}\\
\midrule
\endhead
\midrule\multicolumn{3}{r@{}}{\footnotesize Continued on the next page}\\
\endfoot
\bottomrule
\endlastfoot
\multicolumn{3}{@{}l}{\color{TableInk}\bfseries 1. Problem and Euclidean geometry}\\*[3pt]
$f$, $d$ & Proper, closed, convex objective on $\R^d$; $d$ is finite. & Assump.~\ref{ass:convex}\\
$\dom f$, $\partial f(x)$ & Effective domain $\{x:f(x)<+\infty\}$ and convex subdifferential at $x$. & \S\ref{sec:intro}\\
$\X$, $\fs$ & Solution set $\argmin f$ and attained minimum $\min f$. & Assump.~\ref{ass:attained}\\
$\ip uv$, $\norm u$ & Euclidean inner product and its induced norm. & \S\ref{sec:intro}\\
$\B(x,t)$ & Closed ball $\{u:\norm{u-x}\le t\}$, with $t>0$. & \S\ref{sec:intro}\\
$\mathbb{S}^{d-1}$ & Unit sphere $\{u\in\R^d:\norm u=1\}$. & Lem.~\ref{lem:acute-directions}\\
$\dist(x,C)$, $\Proj_C(x)$ & Distance to a set and Euclidean projection onto a nonempty closed convex set. & \S\ref{sec:foundation}\\
\midrule
\multicolumn{3}{@{}l}{\color{TableInk}\bfseries 2. Iterates, BPM steps, and subgradients}\\*[3pt]
$x^k$, $t_k$, $t$ & Iterate, radius of the step from $x^k$, and a constant radius; $K$ denotes a number of iterations. & Assump.~\ref{ass:iteration}\\
$\brox_f^t(x)$ & Set of exact minimizers of $f$ on $\B(x,t)$. A nonterminal step has a unique output. & \textup{(\ref*{eq:bpm-intro})}\\
$u$, $g^k$, $c_k$ & Returned point $u=x^{k+1}$ and radial subgradient $g^k=c_k(x^k-u)\in\partial f(u)$, with $c_k>0$ on a nonterminal step. & \eqref{eq:radial}\\
$\prox_{\lambda f}(x)$ & Minimizer of $f(u)+\norm{u-x}^2/(2\lambda)$; the corresponding proximal parameter is $\lambda=1/c=t/\norm g$. & \eqref{eq:prox-equivalence}\\
$T_t(x)$ & Unique nonterminal output, or the nearest minimizer $\Proj_{\X}(x)$ when the step is terminal. & Def.~\ref{def:projected}\\
$h^k$, $\tau$, $\alpha$ & Available subgradient at $x^k$; positive radius-rule parameters in $t_k=\tau\norm{h^k}$ or $t_k=\tau\Delta_k^\alpha$. & \S\ref{sec:calibration}\\
$R_{t,\lambda}(x)$ & Relaxed update $x+\lambda(T_t(x)-x)$, with relaxation factor $0<\lambda<2$. & \S\ref{sec:relaxation}\\
\midrule
\multicolumn{3}{@{}l}{\color{TableInk}\bfseries 3. Error measures and termination counts}\\*[3pt]
$D_k$, $D(x)$ & Distances to the solution set $\dist(x^k,\X)$ and $\dist(x,\X)$; $p^k=\Proj_{\X}(x^k)$. & \S\ref{sec:foundation}\\
$\Delta_k$ & Objective gap $f(x^k)-\fs$. & \S\ref{sec:values}\\
$m(x)$, $m_K$ & Minimum subgradient norm $\dist(0,\partial f(x))$ and $m(x^K)$; the distance to the empty set is $+\infty$. & Cor.~\ref{prop:minimum-subgradient}\\
$D_f^g(z,u)$ & Generalized Bregman distance $f(z)-f(u)-\ip g{z-u}$, with $g\in\partial f(u)$. & \eqref{eq:bregman-definition}\\
$\mathcal G_{x^\star}(u,g)$ & Symmetric Bregman distance to a minimizer $\ip g{u-x^\star}$, for $x^\star\in\X$ and $g\in\partial f(u)$. & \eqref{eq:pd-gap}\\
$\mathcal G_K$ & $\mathcal G_{x^\star}(x^K,g^{K-1})$, with $x^\star=\Proj_{\X}(x^0)$; set to zero at and after termination. & \S\ref{sec:primal-dual}\\
$N_t$, $N_\varepsilon$ & Iterations to termination at constant radius $t$, including the terminal step; first index with $\Delta_K\le\varepsilon$, respectively. & \eqref{eq:squared-count}, \eqref{eq:value-complexity}\\
\newpage
\multicolumn{3}{@{}l}{\color{TableInk}\bfseries 4. Squared-distance decrease and trajectory length}\\*[3pt]
$S$, $q_k$, $\beta_k$ & $S=D_0^2/t^2$, $q_k=\Delta_{k+1}/\Delta_k$, and $\beta_k=(D_k^2-D_{k+1}^2)/t_k^2$, on nonterminal iterations. & \S\ref{sec:values}\\
$r$, $r_\star$ & In the final-distance refinement, $r=\Delta_K/\Delta_0$ and $r_\star$ is the unique root defining its implicit upper bound. & Rem.~\ref{rem:jensen-final-distance}\\
$E_j$ & Minimum of the geometric and Jensen bounds on the relative objective gap at integer horizon $j$. & \eqref{eq:stationarity-value-envelope}\\
$C_d$ & Dimension-dependent constant bounding the length of a bounded self-contracted sequence by its diameter. & Lem.~\ref{lem:length}\\
$S_t$, $\Psi_t(D_0)$ & Full fixed-radius trajectory length, including the projected terminal step, and its explicit upper bound. & Prop.~\ref{prop:general-radius-cap}\\
$\bar x$, $L_k$ & Limit of the stopped trajectory and remaining length $\sum_{i=k}^\infty\norm{x^{i+1}-x^i}$. & Cor.~\ref{cor:tail-length}\\
$\rho$, $\rho_\star$ & Decay factor in $t_k=t\rho^k$ and the minimum successful factor when $0<t<D_0$. & Thm.~\ref{thm:geometric-factor}\\
$\mathcal T_k(\rho)$ & Remaining geometric radius sum $t\rho^k/(1-\rho)$. & App.~\ref{app:geometric-factor}\\
\midrule
\multicolumn{3}{@{}l}{\color{TableInk}\bfseries 5. Envelope, duality, and epigraph lifting}\\*[3pt]
$F_t$, $f^*$ & Ball envelope $F_t(x)=\inf_{\norm{u-x}\le t}f(u)$ and convex conjugate $f^*(g)=\sup_x\{\ip gx-f(x)\}$. & \eqref{eq:envelope-definition}\\
$q_t(\gamma)$, $u_\gamma$ & Scalar dual value and proximal point $\prox_{f/\gamma}(x)$, for $\gamma>0$. & \eqref{eq:view-scalar-dual}--\eqref{eq:view-scalar-path}\\
$\delta_C$, $\operatorname{epi}f$ & Indicator of a set ($0$ on $C$, $+\infty$ outside) and epigraph $\{(x,s):f(x)\le s\}$. & \eqref{eq:epigraph-objective}\\
$F$, $z^k$, $R_k$ & Lifted objective $F(x,s)=s+\delta_{\operatorname{epi}f}(x,s)$, lifted iterate $z^k=(x^k,f(x^k))$, and lifted distance to the solution set $R_k$. & \S\ref{sec:reformulations}\\
$r_k$, $\delta_k$ & Projected displacement $\norm{x^{k+1}-x^k}$ and value decrease $f(x^k)-f(x^{k+1})$; $t_k^2=r_k^2+\delta_k^2$. & \eqref{eq:epigraph-radius}\\
$G^k$, $\lambda_k$ & Radial subgradient of the lifted objective $(\lambda_kg^k,1-\lambda_k)$ and its coefficient $\lambda_k=1-c_k\delta_k\in(0,1)$. & \eqref{eq:epigraph-certificate}\\
$N_t^{\rm epi}$ & Fixed-radius termination count for lifted BPM initialized on the graph. & \eqref{eq:epigraph-count}\\
\end{longtable}
\endgroup
Context distinguishes three uses of $\lambda$: a positive proximal parameter, a relaxation factor in $(0,2)$, and a coefficient in the epigraph subgradient representation in $(0,1)$. Likewise, $F_t$ is the ball envelope, whereas $F$ is the lifted objective; $S$ is a squared-distance budget, whereas $S_t$ is a trajectory length.
\clearpage

\section{Auxiliary results}
\label{app:auxiliary-results}

\subsection{Auxiliary geometry for the Euclidean length bound}

\begin{lemma}[Directions with pairwise nonnegative inner products]\label{lem:acute-directions}
Let $\mathbb{S}^{d-1}=\{u\in\R^d:\norm u=1\}$. 
If a nonempty set $U\subset \mathbb{S}^{d-1}$ satisfies $\ip{u}{v}\ge0$ for every $u,v\in U$, then there exists $w\in \mathbb{S}^{d-1}$ such that
$\ip{w}{u}\ge1/\sqrt d$ for every $u\in U$.
\end{lemma}

\begin{proof}
Replace $U$ by its closure and let $p$ be the point of minimum norm in its convex hull. 
By Carath\'eodory's theorem, $p=\sum_{i=1}^m a_i u_i$ with $m\le d+1$, $a_i\ge0$, and $\sum_i a_i=1$. 
Pairwise nonnegativity gives $\norm p^2\ge\sum_i a_i^2\ge1/(d+1)$, so $p\ne0$. 
Since $p$ is the projection of the origin onto $\operatorname{conv} U$, $\ip{p}{u}\ge\norm p^2$ for every $u\in U$.
It remains to sharpen the representation to $m\le d$. 
If $\operatorname{conv} U$ has affine dimension at most $d-1$, this follows directly from Carath\'eodory's theorem. 
Otherwise, the supporting hyperplane $\{u:\ip{p}{u}=\norm p^2\}$ contains $p$. 
In any convex representation of $p$, every point with positive weight must lie on this hyperplane, so Carath\'eodory's theorem there again gives a representation with at most $d$ points. 
Hence $\norm p^2\ge1/d$. 
Setting $w=p/\norm p$ gives $\ip{w}{u}\ge1/\sqrt d$ for every $u\in U$.
\end{proof}

For $\delta\in(0,1]$, the unit sphere $\mathbb{S}^{d-1}$ has a finite $\delta$-net $V\subset \mathbb{S}^{d-1}$ with $|V|\le(1+2/\delta)^d$.
Indeed, take a maximal $\delta$-separated subset of $\mathbb{S}^{d-1}$.
The disjoint open balls of radius $\delta/2$ centered at its points lie inside the ball of radius $1+\delta/2$.
Volume comparison gives the cardinality bound, while maximality gives the net property.

\section{Proofs}
\label{app:main-proofs}

\subsection{Proof of \texorpdfstring{\Cref{prop:four-views}}{Proposition~\ref*{prop:four-views}}}
\label{app:four-views}

\begin{proof}
\emph{Optimality system.}
By \Cref{lem:certificate}, every step from $x\in\dom f$ attains a finite value.
If the output $u$ is nonterminal, it is unique, $\norm{x-u}=t$, and there exists $g=c(x-u)\in\partial f(u)$ with $c>0$. 
Hence

\[
g\in\partial f(u), \qquad x-u\overset{\eqref{eq:radial}}{\in} t\partial\norm{\cdot}(g).
\]

If $u$ is terminal, then $0\in\partial f(u)$ and $\norm{x-u}\le t$, so the same system holds with $g=0$.
Conversely, suppose
\[
g\in\partial f(u), \qquad x-u\in t\partial\norm{\cdot}(g).
\]
If $g=0$, then $u$ is a feasible global minimizer. If $g\ne0$, then $x-u=tg/\norm g$, and for every $z\in\B(x,t)$,

\[
f(z)\ge f(u)+\ip g{z-u} \overset{\eqref{eq:radial}}{=}f(u)+\ip g{z-x}+t\norm g \ge f(u).
\]

Thus $u$ is an exact ball minimizer. This proves \eqref{eq:view-system}.

\emph{(i): Exact ball minimization.}
This is the defining relation $u\in\brox_f^t(x)$ in \eqref{eq:view-ball}.

\emph{(ii): Normalized implicit subgradient formulation.}
On a nonterminal step, the optimality system gives $g=c(x-u)\ne0$ and $\norm{x-u}=t$, hence

\[
x-u\overset{\eqref{eq:radial}}{=}t\frac{g}{\norm g}.
\]

Moreover,
$0\in\partial f(u)+c(u-x)$, so

\[
u\overset{\eqref{eq:prox-equivalence}}{=}\prox_{\lambda f}(x), \qquad \lambda=\frac1c=\frac{t}{\norm g}.
\]

This proves \eqref{eq:view-implicit}. Conversely, any $g\in\partial f(u)\setminus\{0\}$ satisfying $x-u=tg/\norm g$ satisfies the optimality system and therefore gives an exact ball minimizer.

\emph{(iii): Normalized explicit envelope formulation.}
Write $F_t=f\mathbin\square\delta_{\B(0,t)}$.
Convexity follows directly from convexity of $f$ and the ball, and $F_t$ is proper because $F_t(y)\le f(y)<\infty$ for every $y\in\dom f$.
To see that $F_t$ is closed, let $y_n\to y$ with $\liminf_n F_t(y_n)<\infty$. Passing to a subsequence, assume the limit exists and choose minimizers $v_n$ with $\norm{v_n-y_n}\le t$. Then $(v_n)$ is bounded, so along a further subsequence $v_n\to v$ with $\norm{v-y}\le t$. Closedness of $f$ gives

\[
F_t(y)\overset{\eqref{eq:envelope-definition}}{\le} f(v)\le\liminf_n f(v_n) =\liminf_n F_t(y_n).
\]

The same argument rules out a limit of $-\infty$ by properness of $f$. Thus $F_t$ is proper, closed, and convex.
Writing $y=v+w$ with $\norm w\le t$ gives

\begin{equation}
\label{eq:envelope-conjugate-proof}
F_t^*(g)=f^*(g)+t\norm g.
\end{equation}

Fix an exact output $u$ at $x$, so $F_t(x)=f(u)$. For any $g$,

\begin{equation}
\label{eq:fenchel-slacks}
f(u)+f^*(g)-\ip gu\ge0, \qquad t\norm g-\ip g{x-u}\ge0.
\end{equation}

Their sum is $F_t(x)+F_t^*(g)-\ip gx$. 
Hence $g\in\partial F_t(x)$ exactly when both terms vanish, namely when

\[
g\in\partial f(u), \qquad \ip g{x-u}\overset{\eqref{eq:fenchel-slacks}}{=}t\norm g.
\]

This proves \eqref{eq:view-common-subgrad}. On a nonterminal step, $0\notin\partial f(u)$, and equality in Cauchy--Schwarz gives

\[
x-u\overset{\eqref{eq:view-common-subgrad}}{=}t\frac{g}{\norm g}
\]

for every $g\in\partial F_t(x)$. This proves \eqref{eq:view-explicit}.

\emph{(iv): Projection formulation.}
Let $u$ be nonterminal and $C_u=\{z:f(z)\le f(u)\}$. For every $z\in C_u$,
\[
\ip g{z-u}\le f(z)-f(u)\le0.
\]
Since $g=c(x-u)$ with $c>0$, this is the projection optimality condition for $u=\Proj_{C_u}(x)$. Together with $\norm{x-u}=t$, this proves \eqref{eq:view-sublevel}.
Conversely, suppose $u=\Proj_{C_u}(x)$ and $\norm{x-u}=t$. If some $z\in\B(x,t)$ satisfied $f(z)<f(u)$, then $v=(z+u)/2$ would satisfy $f(v)<f(u)$ and $\norm{v-x}<t$, contradicting $\dist(x,C_u)=t$. Hence $u$ is an exact ball minimizer.
Taking $\alpha=f(u)$ gives \eqref{eq:view-exchange}.

\emph{(v): Vector dual formulation.}
The identity $F_t^*=f^*+t\norm{\cdot}$ gives

\[
F_t(x)\overset{\eqref{eq:envelope-conjugate-proof}}{=}\max_g\{\ip xg-f^*(g)-t\norm g\},
\]

and the maximizing set is exactly $\partial F_t(x)$. 
Hence the maximum is attained.
For any dual optimizer $g$ and exact primal output $u$, the two nonnegative terms above have zero sum and therefore vanish separately. Thus
\[
u\in\partial f^*(g), \qquad x-u\in t\partial\norm{\cdot}(g),
\]
which proves \eqref{eq:view-dual}. 
Conversely, these inclusions give the optimality system and therefore primal optimality and equality of primal and dual values.
On a nonterminal step, $g\ne0$, so

\[
u\overset{\eqref{eq:view-system}}{=}x-t\frac{g}{\norm g}.
\]

If $p\in\partial f(x)$ is available, conjugacy gives $x\in\partial f^*(p)$ and $f^*(p)<\infty$. 
Moreover,

\[
t\norm h+D_{f^*}^{x}(h,p) \overset{\eqref{eq:view-dual-bregman}}{=} f^*(h)+t\norm h-\ip xh +\ip xp-f^*(p),
\]

so \eqref{eq:view-dual-bregman} differs from the vector dual objective only by a constant. 
Thus the two minimization problems are equivalent.

\emph{(vi): Scalar multiplier formulation.}
For $\gamma\ge0$ and every feasible $z$,
\[
f(z)+\frac{\gamma}{2} \bigl(\norm{z-x}^2-t^2\bigr) \le f(z),
\]
so $q_t(\gamma)\le F_t(x)$. 
For a nonterminal output, let $g=c(x-u)$ be a radial subgradient. Since
$0\in\partial f(u)+c(u-x)$, the point $u$ minimizes $f(z)+c\norm{z-x}^2/2$. 
Together with $\norm{u-x}=t$, this gives $q_t(c)=f(u)=F_t(x)$. 
Hence the scalar maximum is attained at a positive multiplier.
For $\gamma>0$, let
\[
u_\gamma=\prox_{f/\gamma}(x).
\]
The minimizer is unique and depends continuously on $\gamma>0$.
Indeed, on a compact positive parameter interval, comparison with $z=x$ and an affine lower bound on $f$ keep the minimizers bounded. Every cluster point at a limiting parameter solves the corresponding strongly convex problem and is therefore its unique minimizer.
Set
\[
a(\gamma)=\frac{\norm{u_\gamma-x}^2-t^2}{2}.
\]
For $\eta>\gamma>0$, evaluating each parameter's objective at the other parameter's minimizer gives

\begin{equation}
\label{eq:dual-secants}
a(\eta) \overset{\eqref{eq:view-scalar-dual}}{\le} \frac{q_t(\eta)-q_t(\gamma)}{\eta-\gamma} \overset{\eqref{eq:view-scalar-dual}}{\le} a(\gamma).
\end{equation}

Continuity therefore yields

\[
q_t'(\gamma) \overset{\eqref{eq:dual-secants}}{=} \frac12\bigl(\norm{u_\gamma-x}^2-t^2\bigr),
\]

which proves \eqref{eq:view-scalar-derivative}. 
Since $q_t$ is an infimum of affine functions of $\gamma$, it is concave, and the same inequalities show that $\norm{u_\gamma-x}$ is nonincreasing.
Every positive maximizing multiplier satisfies $\norm{u_\gamma-x}=t$. Its proximal optimality condition gives $g=\gamma(x-u_\gamma)$, so $u_\gamma$ is the exact ball output.
Conversely, any positive $\gamma$ satisfying $\norm{u_\gamma-x}=t$ gives the same radial subgradient and attains the scalar maximum. 
The maximizing multiplier need not be unique.

\emph{(vii): Epigraph formulation.}
If $u$ minimizes $f$ over $\B(x,t)$, then $(u,f(u))$ is feasible for \eqref{eq:view-epigraph-cylinder}, and every feasible $(v,s)$ satisfies $s\ge f(v)\ge f(u)$. 
Hence $(u,f(u))$ is optimal.
Conversely, if $(u,s)$ is optimal, then $s=f(u)$; otherwise lowering $s$ to $f(u)$ preserves feasibility and decreases the objective. 
For every $v\in\B(x,t)\cap\dom f$, the feasible pair $(v,f(v))$ then gives $f(u)\le f(v)$. 
Thus $u$ is an exact ball minimizer, and

\[
\argmin_{\substack{f(v)\le s\\ \norm{v-x}\le t}}s\overset{\eqref{eq:view-epigraph-cylinder}}{=}\{(v,f(v)):v\in\brox_f^t(x)\}.
\]

This proves (vii).
\end{proof}

\subsection{Proof of \texorpdfstring{\Cref{lem:certificate}}{Lemma~\ref*{lem:certificate}}}
\label{proof:foundation:1}

\begin{proof} 
Since $x\in\dom f$, the compact ball $\B(x,t)$ contains a point where $f$ is finite.
Since $f$ is proper and closed, it attains a finite minimum on $\B(x,t)$. 
Hence $\brox_f^t(x)$ is nonempty and compact.
Suppose the step is nonterminal and let $u\in\brox_f^t(x)$. 
Then $u$ must lie on the boundary of the ball: otherwise $u$ would be a local minimizer of the convex function $f$, hence a global minimizer. 
The output is unique as well, since two distinct minimizers would have an interior midpoint of no larger value, which would again be a global minimizer. 
Thus $\norm{x-u}=t$. 
Under \assref{ass:attained}, if $D(x)\le t$, the ball contains a global minimizer, and hence $\brox_f^t(x)=\X\cap\B(x,t)$. 
If $D(x)=t$, every point in this intersection is a projection of $x$ onto the closed convex set $\X$, so uniqueness of the projection gives $\brox_f^t(x)=\{\Proj_{\X}(x)\}$. 
It remains to prove the subgradient for a nonterminal step. 
Choose $z\in\ri(\dom f)$. 
Points on the segment from $x$ toward $z$, sufficiently close to $x$, lie in $\ri(\dom f)\cap\B^\circ(x,t)$, so the convex subdifferential sum rule applies. 
Fermat's rule gives $0\in\partial f(u)+N_{\B(x,t)}(u)$. Since $\norm{x-u}=t$, we have $N_{\B(x,t)}(u)=\{c(u-x):c\ge0\}$, and therefore, for some $g\in\partial f(u)$ and $c\ge0$, 
\begin{equation*}
  \norm{x-u}=t,\qquad g=c(x-u)\in\partial f(u). 
\end{equation*} 
If $c=0$, then $0\in\partial f(u)$ and $u$ is globally optimal, contradicting nonterminality. 
Hence $c>0$. 
Finally, setting $\lambda=1/c$ gives $0\in\partial f(u)+(u-x)/\lambda$, which is the optimality condition for the proximal subproblem. 
Therefore 
\begin{equation*}
  u=\prox_{\lambda f}(x),\qquad \lambda=\frac{1}{c}=\frac{t}{\norm g}.
\end{equation*} 
\end{proof}

\subsection{Proof of \texorpdfstring{\Cref{prop:certificate-monotone}}{Proposition~\ref*{prop:certificate-monotone}}}
\label{proof:foundation:4}

\begin{proof}
For two consecutive nonterminal iterations, monotonicity of $\partial f$ at $x^{k+1}$ and $x^{k+2}$ gives
\[
\ip{g^k-g^{k+1}}{x^{k+1}-x^{k+2}}\ge0.
\]
Using $x^{k+1}-x^{k+2}=t_{k+1}g^{k+1}/\norm{g^{k+1}}$
and dividing by $t_{k+1}>0$ yields \eqref{eq:acute}. 
Cauchy--Schwarz then gives $\norm{g^{k+1}}\le\norm{g^k}$, and \eqref{eq:acute} also shows that the two radial subgradient directions have positive inner product.
Now suppose the first $K$ iterations are nonterminal. Summing \eqref{eq:drop} and using the monotonicity just proved gives

\[
f(x^0)-f(x^K)
\overset{\eqref{eq:drop}}{\ge} \sum\limits_{k=0}^{K-1} t_k\norm{g^k}
\overset{\eqref{eq:acute}}{\ge} \left(\sum\limits_{k=0}^{K-1}t_k\right)\norm{g^{K-1}}.
\]

Since $g^{K-1}\in\partial f(x^K)$,
$\dist(0,\partial f(x^K))\le\norm{g^{K-1}}$, and under \assref{ass:attained},
$f(x^0)-f(x^K)\le\Delta_0$. This proves \eqref{eq:stationarity}. If a terminal output has already occurred, then by convention $g^{K-1}=0$ and $0\in\partial f(x^K)$, so the same inequalities hold immediately.
\end{proof}

\subsection{Proof of \texorpdfstring{\Cref{thm:segment}}{Theorem~\ref*{thm:segment}}}
\label{proof:value_rates:1}

\begin{proof}
Suppose first that $D_k>0$, and let $p^k=\Proj_{\X}(x^k)$ and $a_k=\min\{t_k/D_k,1\}$. 
The point $v=(1-a_k)x^k+a_kp^k$ satisfies
$\norm{v-x^k}=a_kD_k\le t_k$ and is therefore feasible for the $k$th step. 
Exact minimization and convexity give

\[
f(x^{k+1})\overset{\textup{(\ref*{eq:bpm-intro})}}{\le} f(v)
\le(1-a_k)f(x^k)+a_k\fs.
\]

Subtracting $\fs$ yields \eqref{eq:segment}. If $D_k=0$, then $x^k\in\X$, and both gaps are zero by the stopping convention. 
Now let $t_k=t$ with $0<t<D_0$. Before termination, $D_k\le D_0$, so $1-t/D_k\le1-t/D_0$ on every nonterminal step. 
Iterating \eqref{eq:segment} therefore gives \eqref{eq:frozen}. 
At and after a terminal step the objective gap is zero, so the same bound holds for the stopped sequence.
\end{proof}

\subsection{Proof of \texorpdfstring{\Cref{prop:gap-distance}}{Corollary~\ref*{prop:gap-distance}}}
\label{proof:value_rates:4}

\begin{proof}
Consider a nonterminal step and let $p^{k+1}=\Proj_{\X}(x^{k+1})$. 
Since $g^k\in\partial f(x^{k+1})$, the subgradient inequality gives $\Delta_{k+1}\le\norm{g^k}D_{k+1}$, while \eqref{eq:drop} gives $\Delta_k-\Delta_{k+1}\ge t_k\norm{g^k}$. 
Hence

\[
\frac{\Delta_{k+1}}{\Delta_k}
\overset{\eqref{eq:drop}}{\le}\frac{D_{k+1}}{D_{k+1}+t_k}
\le\frac{D_{k+1}}{D_k},
\]

where the second inequality uses $D_k\le D_{k+1}+\norm{x^{k+1}-x^k}=D_{k+1}+t_k$.
Multiplying these inequalities over a sequence of nonterminal steps gives $\Delta_K/\Delta_0\le D_K/D_0$, proving \eqref{eq:gap-distance}.
If a terminal step occurs earlier, then the stopped sequence has $\Delta_K=D_K=0$, so the same bound holds.
\end{proof}

\subsection{Proof of \texorpdfstring{\Cref{thm:refined}}{Theorem~\ref*{thm:refined}}}
\label{proof:value_rates:2}

\begin{proof}
Let $p^k=\Proj_{\X}(x^k)$. 
Applying \eqref{eq:identity} with $z=p^k$ and discarding the nonnegative Bregman term gives

\[
D_k^2-\norm{x^{k+1}-p^k}^2
\overset{\eqref{eq:identity}}{\ge} t_k^2+\frac{2\Delta_{k+1}}{c_k}.
\]

Since $D_{k+1}\le\norm{x^{k+1}-p^k}$, this yields \eqref{eq:refined-c}.
By \eqref{eq:drop},
$\Delta_k-\Delta_{k+1}\ge t_k\norm{g^k}=c_kt_k^2>0$.
Hence

\[
\frac{2\Delta_{k+1}}{c_k}
\overset{\eqref{eq:drop},\eqref{eq:radial}}{\ge}
\frac{2t_k^2\Delta_{k+1}}
{\Delta_k-\Delta_{k+1}}.
\]

Substituting this bound into \eqref{eq:refined-c} and combining the two $t_k^2$ terms gives \eqref{eq:refined-ratio}. 
Since the step is nonterminal, $\Delta_{k+1}>0$, so the multiplier in \eqref{eq:refined-ratio} is strictly larger than one.
\end{proof}

\subsection{Proof of \texorpdfstring{\Cref{thm:jensen}}{Theorem~\ref*{thm:jensen}}}
\label{proof:value_rates:3}

\begin{proof}
For the first $K$ nonterminal iterations, set $q_k=\Delta_{k+1}/\Delta_k\in(0,1)$ and $\beta_k=(D_k^2-D_{k+1}^2)/t^2>1$. 
By \Cref{thm:refined},

\begin{equation}
\label{eq:jensen-input}
q_k\overset{\eqref{eq:refined-ratio}}{\le}\frac{\beta_k-1}{\beta_k+1}, \qquad \sum\limits_{k=0}^{K-1}\beta_k =\frac{D_0^2-D_K^2}{t^2}<S,
\end{equation}

where the strict inequality follows from $D_K>0$. Since every $\beta_k>1$, this also implies $K<S$.
Let $h(b)=\log((b-1)/(b+1))$. 
On $(1,\infty)$, $h'(b)=2/(b^2-1)>0$ and $h''(b)=-4b/(b^2-1)^2<0$, so $h$ is increasing and strictly concave. Jensen's inequality therefore gives

\[
\begin{aligned}\log\frac{\Delta_K}{\Delta_0}&\overset{\eqref{eq:ratio-definitions}}{=}\sum\limits_{k=0}^{K-1}\log q_k\\&\overset{\eqref{eq:jensen-input}}{\le}K h\left(\frac1K\sum\limits_{k=0}^{K-1}\beta_k\right)\\&\overset{\eqref{eq:jensen-input}}{\le}K\log\frac{S-K}{S+K}.\end{aligned}
\]

Exponentiating proves the first inequality in \eqref{eq:jensen}. With $v=K/S\in(0,1)$, the inequality $\log((1-v)/(1+v))\le-2v$ yields $\Delta_K\le\Delta_0\exp(-2K^2/S)$.
For \eqref{eq:value-complexity}, \Cref{cor:squared} gives
$N_\varepsilon\le\lceil S\rceil$. 
Now set
\[
\widehat K = \left\lceil \frac{D_0}{t} \sqrt{\frac12\log\frac{\Delta_0}{\varepsilon}} \right\rceil.
\]
If $\widehat K\ge\lceil S\rceil$, the first bound already gives $N_\varepsilon\le\widehat K$. Otherwise $\widehat K<S$. 
If a terminal step occurs before $\widehat K$, then $\Delta_{\widehat K}=0\le\varepsilon$; if the first $\widehat K$ iterations are nonterminal, the exponential bound gives $\Delta_{\widehat K}\le\varepsilon$. 
Thus $N_\varepsilon\le\widehat K$, proving \eqref{eq:value-complexity}.
\end{proof}

\subsection{Proof of \texorpdfstring{\Cref{thm:refined-stationarity}}{Theorem~\ref*{thm:refined-stationarity}}}
\label{app:refined-stationarity}

\begin{proof}
Suppose first that the first $K$ iterations are nonterminal, and choose radial subgradients $g^i\in\partial f(x^{i+1})$ for $0\le i<K$. 
For any $0\le j<K$, summing \eqref{eq:drop} from $i=j$ to $K-1$ and discarding the nonnegative Bregman distances gives

\begin{equation}
\label{eq:suffix-drop-proof}
\Delta_j-\Delta_K \overset{\eqref{eq:drop}}{\ge} \sum\limits_{i=j}^{K-1}t_i\norm{g^i}.
\end{equation}

By \Cref{prop:certificate-monotone}, $\norm{g^i}\ge\norm{g^{K-1}}\ge m(x^K)$ for $j\le i<K$. 
Hence

\[
\Delta_j-\Delta_K \overset{\eqref{eq:suffix-drop-proof},\eqref{eq:acute}}{\ge} \left(\sum\limits_{i=j}^{K-1}t_i\right)m(x^K).
\]

Dividing by the positive radius sum and minimizing over $j$ gives the first inequality in \eqref{eq:suffix-stationarity}. 
The second follows from $\Delta_K\ge0$.
If a terminal output occurs before or at $K$, then $m(x^K)=0$, while
$\Delta_j\ge\Delta_K$ for every $j<K$, so \eqref{eq:suffix-stationarity} remains valid.
Now let $t_k=t$ with $0<t<D_0$, and suppose the first $K$ iterations are nonterminal.
By \Cref{thm:jensen}, $K<S$. Taking $j=K-1$ in \eqref{eq:suffix-stationarity} gives $m(x^K)\le{\Delta_{K-1}}/{t}$.
For $K\ge2$, applying \Cref{thm:jensen} at horizon $K-1$ yields

\[
\Delta_{K-1} \overset{\eqref{eq:jensen}}{\le} \Delta_0 \left(\frac{S-K+1}{S+K-1}\right)^{K-1} \overset{\eqref{eq:jensen}}{\le} \Delta_0 \exp\left(-\frac{2(K-1)^2}{S}\right),
\]

which proves \eqref{eq:jensen-stationarity}. For $K=1$, the same bound reduces to $m(x^1)\le\Delta_0/t$.
Finally, for every $0\le j<K$, \eqref{eq:frozen} and \eqref{eq:jensen} give $\Delta_j\le\Delta_0E_j$.
Since
$\sum_{i=j}^{K-1}t_i=(K-j)t$, the second inequality in
\eqref{eq:suffix-stationarity} gives

\[
m(x^K) \overset{\eqref{eq:suffix-stationarity},\eqref{eq:stationarity-value-envelope}}{\le} \frac{\Delta_0}{t}\frac{E_j}{K-j}.
\]

Minimizing over $0\le j<K$ proves \eqref{eq:optimized-stationarity}.
If $f$ is differentiable, then $m(x^K)=\norm{\nabla f(x^K)}$, giving the stated specialization.
\end{proof}

\subsection{Proof of \texorpdfstring{\Cref{prop:primal-dual-gap}}{Proposition~\ref*{prop:primal-dual-gap}}}
\label{app:primal-dual-gap}

\begin{proof}
Since $x^\star\in\X$, we have $0\in\partial f(x^\star)$, hence
$x^\star\in\partial f^*(0)$ and $f^*(0)=-\fs$. 
Similarly, $g\in\partial f(u)$ implies $u\in\partial f^*(g)$ and
$f^*(g)=\ip ug-f(u)$. 
Therefore

\begin{equation}
\label{eq:reverse-gap-proof}
D_{f^*}^{u}(0,g) \overset{\eqref{eq:view-dual-bregman}}{=}f^*(0)-f^*(g)+\ip ug =f(u)-\fs,
\end{equation}

which proves \eqref{eq:pd-reverse}. 
Moreover,

\begin{equation}
\label{eq:forward-gap-proof}
D_{f^*}^{x^\star}(g,0) \overset{\eqref{eq:view-dual-bregman}}{=}f^*(g)-f^*(0)-\ip{x^\star}g =\fs-f(u)-\ip g{x^\star-u} \overset{\eqref{eq:bregman-definition}}{=}D_f^g(x^\star,u),
\end{equation}

proving \eqref{eq:pd-forward}. 
Both quantities are nonnegative by convexity. 
Adding them gives

\[
\mathcal G_{x^\star}(u,g) \overset{\eqref{eq:reverse-gap-proof},\eqref{eq:forward-gap-proof}}{=}\ip g{u-x^\star},
\]

which is \eqref{eq:pd-gap}. Hence

\[
0\le f(u)-\fs \overset{\eqref{eq:reverse-gap-proof},\eqref{eq:forward-gap-proof}}{\le}\mathcal G_{x^\star}(u,g) \overset{\eqref{eq:pd-gap}}{\le}\norm g\,\norm{u-x^\star},
\]

proving \eqref{eq:pd-sandwich}.
Now consider the stopped iteration and set $x^\star=\Proj_{\X}(x^0)$.
Before termination, \eqref{eq:lower-fejer} gives
$\norm{x^K-x^\star}\le D_0$, and therefore $\mathcal G_K\le D_0\norm{g^{K-1}}.$
For any $0\le j<K$, summing \eqref{eq:drop} gives

\begin{equation}
\label{eq:pd-drop-proof}
\Delta_j-\Delta_K \overset{\eqref{eq:drop}}{\ge} \sum\limits_{i=j}^{K-1}t_i\norm{g^i}.
\end{equation}

By \Cref{prop:certificate-monotone}, $\norm{g^i}\ge\norm{g^{K-1}}$ for $j\le i<K$, so

\[
\left(\sum\limits_{i=j}^{K-1}t_i\right)\norm{g^{K-1}} \overset{\eqref{eq:pd-drop-proof},\eqref{eq:acute}}{\le} \Delta_j-\Delta_K.
\]

Combining the last two inequalities and minimizing over $j$ proves \eqref{eq:pd-suffix}.
At and after termination, the convention $g^{K-1}=0$ gives $\mathcal G_K=0$, so the same bound holds.
Finally, for a constant-radius sequence of nonterminal steps, $\Delta_j\le\Delta_0E_j$ and $\sum_{i=j}^{K-1}t_i=(K-j)t$. Substituting these into \eqref{eq:pd-suffix} gives

\[
\mathcal G_K \overset{\eqref{eq:pd-suffix},\eqref{eq:stationarity-value-envelope}}{\le} \frac{D_0\Delta_0}{t} \frac{E_j}{K-j}.
\]

Minimizing over $0\le j<K$ proves \eqref{eq:pd-optimized}. 
Since the two Bregman distances are nonnegative and sum to $\mathcal G_K$, each satisfies the same upper bounds.
\end{proof}

\subsection{Proof of \texorpdfstring{\Cref{thm:general-convergence}}{Theorem~\ref*{thm:general-convergence}}}
\label{app:general-convergence}

\begin{proof}
If $\sum_k t_k<\infty$, the final assertion of \Cref{thm:arbitrary-radii} applies without assuming attainment and gives finite trajectory length together with

\[
x^k\overset{\eqref{eq:arbitrary-radius-convergence}}{\to}\bar x\in\dom f,\qquad f(x^k)\overset{\eqref{eq:arbitrary-radius-convergence}}{\downarrow} f(\bar x).
\]

This proves (i).
Suppose now that $\sum_k t_k=\infty$.
If a minimizer exists, \Cref{thm:criterion} gives finite termination, proving (ii).
It remains to consider $\argmin f=\varnothing$.
Every step is then nonterminal, so $\norm{x^{k+1}-x^k}=t_k$ for every $k$.
Since $x^k$ is feasible for the $k$th ball problem, the values $f(x^k)$ are nonincreasing.
Let $\ell=\lim_k f(x^k)$.
Suppose $\ell>\inf f$.
Choose $z\in\dom f$ with $f(z)<\ell$.
Then $f(z)<f(x^{k+1})$ for every $k$, so \eqref{eq:lower-fejer} gives

\[
\norm{x^{k+1}-z}^2\overset{\eqref{eq:lower-fejer}}{\le}\norm{x^k-z}^2-t_k^2.
\]

Hence the trajectory is bounded.
By \Cref{lem:self-contracted}, it is self-contracted, and \Cref{lem:length} therefore gives finite total length.
But every step has length $t_k$, so

\[
\sum\limits_{k=0}^{\infty}\norm{x^{k+1}-x^k}\overset{\eqref{eq:radial}}{=}\sum\limits_{k=0}^{\infty} t_k=\infty,
\]

contradicting the finite-length conclusion.
Therefore $f(x^k)\downarrow\inf f$.
Suppose next that $\norm{x^k}$ does not tend to infinity.
Then some bounded subsequence has a cluster point $\bar x$.
By closedness of $f$,
\[
f(\bar x)\le\liminf_j f(x^{k_j})=\inf f.
\]
If $\inf f$ is finite, then $\bar x$ attains the infimum, contradicting $\argmin f=\varnothing$.
If $\inf f=-\infty$, then $f(\bar x)=-\infty$, contradicting properness.
Thus $\norm{x^k}\to\infty$, proving (iii).
The two nonsummable cases also give \eqref{eq:general-value-convergence}.
For universal necessity, use the construction in the proof of \Cref{thm:criterion}.
For any summable schedule with $T=\sum_k t_k$, the objective $f(x)=\norm x$ and starting point $x^0=(T+1)e_1$ produce $x^k\to e_1$ and $f(x^k)\to1>0=\inf f$.
Thus no summable positive schedule guarantees value convergence to the infimum on every problem.
\end{proof}

\subsection{Proof of \texorpdfstring{\Cref{lem:self-contracted}}{Lemma~\ref*{lem:self-contracted}}}
\label{proof:termination:1}

\begin{proof}
Since $x^k$ is feasible for the $k$-th ball problem, the objective values are nonincreasing.
Thus, for $j<m$ in a sequence of nonterminal steps, $f(x^m)\le f(x^{j+1})$.
Applying \eqref{eq:lower-fejer} to the $j$th step with $z=x^m$ gives $\norm{x^{j+1}-x^m}^2\le\norm{x^j-x^m}^2-t_j^2$, which is \eqref{eq:future-fejer}.
In particular, $\norm{x^{j+1}-x^m}\le\norm{x^j-x^m}$.
For $i\le j<m$, applying this comparison successively from $i$ to $j-1$ gives $\norm{x^j-x^m}\le\norm{x^i-x^m}$.
The case $j=m$ is immediate, so the prefix is self-contracted.
\end{proof}

\subsection{Proof of \texorpdfstring{\Cref{lem:length}}{Lemma~\ref*{lem:length}}}
\label{app:length}
It suffices to consider a finite sequence $z^0,\ldots,z^m$.
Let $D$ be its diameter and $\ell_i=\norm{z^{i+1}-z^i}$.
Discard zero-length steps and, for each $i<m$, define
\[
U_i \coloneq \left\{\frac{z^j-z^i}{\norm{z^j-z^i}}:j>i\right\}.
\]
No future point equals $z^i$.
If $i<j<k$, set $a=z^j-z^i$ and $b=z^k-z^i$.
Self-contraction gives $\norm{b-a}\le\norm b$, hence $\ip{a}{b}\ge\norm a^2/2\ge0$.
Thus the directions in $U_i$ have pairwise nonnegative inner products, and \Cref{lem:acute-directions} gives $v_i\in \mathbb{S}^{d-1}$ with $\ip{v_i}{u}\ge1/\sqrt d$ for every $u\in U_i$.
Suppose first that $d\ge2$.
Take a $\delta=1/(2\sqrt d)$ net $V$ and choose $w_i\in V$ with $\norm{w_i-v_i}\le\delta$.
Then $\ip{w_i}{u}\ge1/(2\sqrt d)$ for every $u\in U_i$.
For $j>i$, self-contraction gives $\norm{z^{i+1}-z^j}\le\norm{z^i-z^j}$, and therefore $\ell_i\le2\norm{z^i-z^j}$.
It follows that $\ip{w_i}{z^j-z^i}\ge\ell_i/(4\sqrt d)$.
For $w\in V$, define $b_i(w)=\min_{j\ge i}\ip{w}{z^j}$.
These quantities are nondecreasing in $i$.
If $w_i=w$, then $\ip{w}{z^j}\ge\ip{w}{z^i}+\ell_i/(4\sqrt d)$ for every $j>i$, so $b_{i+1}(w)-b_i(w)\ge\ell_i/(4\sqrt d)$.
Summing over the indices with $w_i=w$ gives
\[
\sum\limits_{\substack{i=0\\w_i=w}}^{m-1}\ell_i\le4\sqrt d\bigl(b_m(w)-b_0(w)\bigr)\le4\sqrt d\,D.
\]
Summing over $w\in V$ and using $|V|\le(1+4\sqrt d)^d$ gives the stated constant.
For $d=1$, all future points lie on the same side of the current point.
Let $I_i=[\min_{j\ge i}z^j,\max_{j\ge i}z^j]$.
Then $z^i$ is an endpoint of $I_i$.
The estimate $\ell_i\le2\norm{z^i-z^j}$ for every $j>i$ shows that removing $z^i$ moves this endpoint inward by at least $\ell_i/2$.
Hence
\[
\ell_i\le2\bigl(|I_i|-|I_{i+1}|\bigr),
\]
and telescoping gives total length at most $2D$.
Applying the finite bound to every prefix proves the result for infinite bounded self-contracted sequences.

\subsection{Proof of \texorpdfstring{\Cref{thm:dimension}}{Theorem~\ref*{thm:dimension}}}
\label{proof:termination:2}

\begin{proof}
Let $p^0=\Proj_{\X}(x^0)$. For every iterate before terminality, \eqref{eq:lower-fejer} with $z=p^0$ gives
$\norm{x^k-p^0}\le D_0$. 
Hence every sequence of nonterminal steps has diameter at most $2D_0$. 
By \Cref{lem:self-contracted,lem:length}, its total length is at most $2C_dD_0$. 
Since every nonterminal displacement has length $t_k$, this proves $\sum_{k<K}t_k\le2C_dD_0$, which is \eqref{eq:length-budget}.
Now suppose the radius is constant. 
The first $N_t-1$ iterations are nonterminal, so $(N_t-1)t\le2C_dD_0$, and therefore $N_t\le1+\lfloor2C_dD_0/t\rfloor$. The other upper bound in \eqref{eq:dimension-count} follows from \Cref{cor:squared}.
For the lower bound, let $x^{N_t}\in\X$ be the terminal output. The total trajectory length satisfies

\[
D_0\le\norm{x^{N_t}-x^0}
\le\sum\limits_{k=0}^{N_t-1}\norm{x^{k+1}-x^k}
\overset{\textup{(\ref*{eq:bpm-intro})}}{\le} N_tt.
\]

Thus $N_t\ge\lceil D_0/t\rceil$. 
This lower bound is attained in one dimension by $f(x)=|x|$, and together with the fixed-dimensional upper bound yields the stated $\Theta_d(D_0/t)$ complexity.
\end{proof}

\subsection{Proof of \texorpdfstring{\Cref{thm:criterion}}{Theorem~\ref*{thm:criterion}}}
\label{proof:termination:3}
\begin{proof}
Assume first that $\sum_k t_k=\infty$.
If the iteration remained nonterminal indefinitely, then \eqref{eq:length-budget} would give $\sum_{k<K}t_k\le2C_dD_0$ for every $K$, contradicting $\sum_k t_k=\infty$.
Hence a terminal step must occur after finitely many steps, and its output is a minimizer.
Conversely, suppose $T=\sum_k t_k<\infty$.
Take $f(x)=\norm{x}$ and $x^0=(T+1)e_1$.
We claim that
\[
x^k=\left(T+1-\sum\limits_{j=0}^{k-1}t_j\right)e_1.
\]
Indeed, $\norm{x^k}=1+\sum_{j\ge k}t_j>t_k$, so $\B(x^k,t_k)$ does not contain the minimizer $0$.
Since $x^k$ lies on the positive $e_1$-axis, the point in this ball nearest to the origin is $x^k-t_ke_1$.
Thus $x^{k+1}=x^k-t_ke_1$, which gives the claimed formula at index $k+1$.
Every step is therefore nonterminal, so the iteration never attains the minimizer and (ii) fails.
\end{proof}

\subsection{Proof of \texorpdfstring{\Cref{prop:general-radius-cap}}{Proposition~\ref*{prop:general-radius-cap}}}
\label{app:general-radius-cap}
\begin{proof}
There are $n=N_t-1$ nonterminal displacements of length $t$ and a projected terminal displacement $r=D_n\le t$.
Let $p=\Proj_{\X}(x^0)$.
Applying \eqref{eq:lower-fejer} with $z=p$ to the first $n$ iterations gives

\begin{equation}
\label{eq:prefix-fixed-projection}
\norm{x^n-p}^2\overset{\eqref{eq:lower-fejer}}{\le} D_0^2-nt^2.
\end{equation}

Since the terminal step is $x^{n+1}=\Proj_{\X}(x^n)$, the projection inequality gives $\norm{x^{n+1}-p}^2\le\norm{x^n-p}^2-r^2$.
Hence $nt^2+r^2\le D_0^2$, so

\[
S_t=nt+r\overset{\eqref{eq:prefix-fixed-projection}}{\le} nt+\sqrt{D_0^2-nt^2}.
\]

By \Cref{cor:squared}, $N_t\le\lceil D_0^2/t^2\rceil$, and therefore $n\le J$.
The function $n\mapsto n+\sqrt{D_0^2/t^2-n}$ is nondecreasing on the integers $0,\ldots,J$: for $n<J$, its increment is $1+\sqrt{a-1}-\sqrt a\ge0$, where $a=D_0^2/t^2-n>1$.
Thus $S_t\le Jt+\sqrt{D_0^2-Jt^2}$.
Set $q=D_0^2/t^2-J\in(0,1]$.
Then $J=D_0^2/t^2-q$, and hence
\[
Jt+\sqrt{D_0^2-Jt^2}=\frac{D_0^2}{t}+t(\sqrt q-q)\le\frac{D_0^2}{t}+\frac{t}{4},
\]
since $0\le\sqrt q-q\le1/4$ for $q\in[0,1]$.
If $D_0^2/t^2$ is an integer, then $q=1$ and the first bound equals $D_0^2/t$.
If $t\ge D_0$, the initial ball contains a minimizer, so the first step is terminal and nearest-minimizer selection at termination gives $S_t=D_0$.
\end{proof}

\subsection{Proof of \texorpdfstring{\Cref{thm:arbitrary-radii}}{Theorem~\ref*{thm:arbitrary-radii}}}
\label{app:arbitrary-radii}
\begin{proof}
If the initial point is optimal or a terminal step occurs, the stopped sequence is eventually a fixed minimizer. It therefore has finite total length and satisfies \eqref{eq:arbitrary-radius-convergence}. The terminal displacement is finite, although no radius-independent bound on it is asserted.
Suppose henceforth that every step is nonterminal. Under attainment, \Cref{thm:dimension} bounds every partial sum of radii by $2C_dD_0$, so $\sum_k t_k<\infty$. If summability is assumed instead, this conclusion requires no attainment assumption. In either case, full-radius motion gives $\sum_k\norm{x^{k+1}-x^k}=\sum_k t_k<\infty$.
Hence $(x^k)$ is Cauchy in $\R^d$ and converges to some $\bar x$. For $m>k$, the triangle inequality gives $\norm{x^m-x^k}\le\sum_{j=k}^{m-1}t_j$. Letting $m\to\infty$ proves the first estimate in \eqref{eq:radius-tail-error}.
Choose a radial subgradient $g^k\in\partial f(x^{k+1})$ at each step. By \Cref{prop:certificate-monotone}, $\norm{g^k}\le M:=\norm{g^0}$ for all $k\ge0$. The objective values are nonincreasing and also bounded below. Indeed, the subgradient inequality at $x^1$ gives
\[
f(x^k)\ge f(x^1)+\ip{g^0}{x^k-x^1}\qquad(k\ge1),
\]
and the convergent sequence $(x^k)$ is bounded. Hence $f(x^k)\downarrow\ell$ for some $\ell\in\R$. Lower semicontinuity gives $f(\bar x)\le\ell<\infty$, so $\bar x\in\dom f$.
For every $k\ge1$, applying the subgradient inequality at $x^k$ with subgradient $g^{k-1}$ to $\bar x$ gives

\begin{equation}
\label{eq:limit-value-proof}
f(\bar x)\ge f(x^k)+\ip{g^{k-1}}{\bar x-x^k}\overset{\eqref{eq:acute}}{\ge} f(x^k)-M\norm{\bar x-x^k}.
\end{equation}

Passing to the limit yields $f(\bar x)\ge\ell$. Thus $f(\bar x)=\ell$, proving value convergence. The same inequality, together with $f(x^k)\ge\ell$ and the distance tail estimate, gives

\[
0\le f(x^k)-f(\bar x)\overset{\eqref{eq:limit-value-proof}}{\le} M\sum\limits_{j=k}^\infty t_j,
\]

which is the second estimate in \eqref{eq:radius-tail-error}. No assumption on $\partial f(x^0)$ is needed, since the first nonterminal output supplies the initial subgradient.
\end{proof}

\subsection{Proof of \texorpdfstring{\Cref{cor:tail-length}}{Corollary~\ref*{cor:tail-length}}}
\label{app:tail-length}
\begin{proof}
The limit exists by \Cref{thm:arbitrary-radii}.
Fix $k$ and let $p=\Proj_{\X}(x^k)$.
For every nonterminal step after $k$, \eqref{eq:lower-fejer} with $z=p$ shows that the distance to $p$ cannot increase.
If a terminal step occurs, nearest-minimizer selection at termination and the projection inequality give the same conclusion.
Hence $\norm{x^i-p}\le\norm{x^k-p}=D_k$ for every $i\ge k$, so the tail has diameter at most $2D_k$.
The tail is also self-contracted.
Indeed, if the update from $x^i$ is nonterminal and $j>i$, monotonicity gives $f(x^j)\le f(x^{i+1})$, so \eqref{eq:lower-fejer} yields $\norm{x^{i+1}-x^j}\le\norm{x^i-x^j}$.
If the update from $x^i$ is terminal, every future point equals $x^{i+1}$ by the stopping convention.
Iterating these comparisons proves self-contraction.
Applying \Cref{lem:length} to the tail gives $L_k\le2C_dD_k$, while the triangle inequality on finite tails and passage to the limit give $\norm{x^k-\bar x}\le L_k$.
If $D_k\le A\rho^k$, substitution gives the stated geometric bounds.
Moreover, $D_k\to0$, and since $\X$ is closed, $x^k\to\bar x$ implies $\bar x\in\X$.
Finally, nearest-minimizer selection at termination is essential for a radius-independent bound on the terminal displacement.
Take $f(x,y)=x^2/2$ and $x^0=(1,0)$.
For any $t>1$, an exact ball oracle may return the minimizer $(0,\sqrt{t^2-1})$, whose displacement from $x^0$ is $t$, although $D_0=1$.
Nearest-minimizer selection at termination instead returns $(0,0)$.
On an infinite nonterminal run, BPM with arbitrary terminal selection and BPM with nearest-minimizer selection coincide, so no additional terminal-selection assumption is needed there.
\end{proof}

\subsection{Proof of \texorpdfstring{\Cref{thm:hard}}{Theorem~\ref*{thm:hard}}}
\label{proof:termination:4}
\begin{proof}
Let $G$ be the $n\times n$ tridiagonal matrix with diagonal entries one and adjacent entries $\epsilon\in(0,1/2)$. For every $v\in\R^n$,
\[
v^\top Gv=\sum\limits_{j=0}^{n-1}v_j^2+2\epsilon\sum\limits_{j=0}^{n-2}v_jv_{j+1}\ge(1-2\epsilon)\norm v^2,
\]
so $G$ is positive definite. Choose unit vectors $n_0,\ldots,n_{n-1}$ with Gram matrix $G$, and define
\[
x^0=0,\qquad x^{j+1}=x^j-tn_j,\qquad x^\star=x^n,\qquad \rho=\frac{\epsilon}{1+\epsilon}.
\]
For $j=0,\ldots,n-2$, set $u_j=x^{j+1}$, $a_j=\rho^j$, $g_j=a_jn_j$, and $F_j=a_jt\epsilon$. Let $\eta=a_{n-2}\epsilon/[2(1+\epsilon)]$ and define
\begin{equation}\label{eq:hard-function}
f(z)=\max\left\{0,\ \max_{0\le j\le n-2}[F_j+\ip{g_j}{z-u_j}],\ \max_{0\le j<n}\eta\left|\ip{n_j}{z-x^\star}\right|\right\}.
\end{equation}
After expanding the absolute values, $f$ is the maximum of finitely many affine functions and is therefore finite, closed, convex, and piecewise affine.
We first verify that the affine piece with gradient $g_i$ is active at $u_i$. For $i>j$, $\ip{g_j}{u_i-u_j}=-a_jt\epsilon$, so $F_j+\ip{g_j}{u_i-u_j}=0\le F_i$. If $i=j-1$, then $F_j+\ip{g_j}{u_i-u_j}=F_j+a_jt=F_{j-1}$. If $i\le j-2$, then
\[
F_j+\ip{g_j}{u_i-u_j}=F_j+a_jt(1+\epsilon)\le F_i,
\]
because $F_i-F_j\ge a_jt((1+\epsilon)^2/\epsilon-\epsilon)\ge a_jt(1+\epsilon)$. Finally, when $i=j$ the value is exactly $F_i$. Thus none of the other affine pieces exceeds $F_i$ at $u_i$.
At $x^\star$, each of these affine pieces vanishes. Indeed, $x^\star-u_j=-t\sum_{r=j+1}^{n-1}n_r$, and only the $r=j+1$ term has nonzero inner product with $n_j$, giving $F_j+\ip{g_j}{x^\star-u_j}=F_j-a_jt\epsilon=0$.
It remains to control the symmetric closing pieces. Since $u_i-x^\star=t\sum_{r=i+1}^{n-1}n_r$, for all relevant $i$ and $j$,

\begin{equation}
\label{eq:closing-bound}
0\le\ip{n_j}{u_i-x^\star}\le t(1+2\epsilon).
\end{equation}

Hence

\[
\eta\left|\ip{n_j}{u_i-x^\star}\right|\overset{\eqref{eq:closing-bound}}{\le}\frac{a_{n-2}t\epsilon(1+2\epsilon)}{2(1+\epsilon)}\le a_it\epsilon=F_i.
\]

Therefore $f(u_i)=F_i$, and the active affine piece gives $g_i\in\partial f(u_i)$.
We now verify that $u_i$ is the exact radius-$t$ response at $x^i$. Since $g_i=(a_i/t)(x^i-u_i)$ and $\norm{x^i-u_i}=t$, every $z\in\B(x^i,t)$ satisfies
\[
\ip{g_i}{z-u_i}=\frac{a_i}{t}\bigl(\ip{x^i-u_i}{z-x^i}+t^2\bigr)\ge0.
\]
The subgradient inequality then gives $f(z)\ge f(u_i)$, so $u_i$ is an exact ball minimizer.
Moreover, $f(x^\star)=0$. If $f(z)=0$, the symmetric pieces in \eqref{eq:hard-function} force $\ip{n_j}{z-x^\star}=0$ for every $j$. Since $G$ is positive definite, the vectors $n_0,\ldots,n_{n-1}$ form a basis of $\R^n$, and hence $z=x^\star$. Thus $x^\star$ is the unique minimizer.
For $j\le n-2$,

\begin{equation}
\label{eq:hard-tail-distance}
\norm{x^j-x^\star}^2=t^2\bigl[(n-j)+2(n-j-1)\epsilon\bigr]>t^2.
\end{equation}

Thus the first $n-1$ iterations are nonterminal, and their exact outputs are uniquely the prescribed points $x^{j+1}$. At $x^{n-1}$, the distance to $x^\star$ is exactly $t$, so the $n$th step is terminal. Therefore $N_t=n$.
Taking $j=0$ in the same distance formula gives

\[
D_0^2\overset{\eqref{eq:hard-tail-distance}}{=}t^2[n+2(n-1)\epsilon],
\]

which is \eqref{eq:hard-distance}. If $\epsilon=o(1/n)$, then $D_0^2/t^2=n+o(1)$, and since the excess over $n$ is positive and tends to zero, $\lceil D_0^2/t^2\rceil=n+1$ for all sufficiently large $n$. Hence $N_t/\lceil D_0^2/t^2\rceil=n/(n+1)\to1$. Finally, a dimension-independent bound $N_t\le CD_0/t$ would imply $n\le C\sqrt{n+o(1)}$, which is impossible as $n\to\infty$.
\end{proof}

\subsection{Proof of \texorpdfstring{\Cref{prop:oracle-transformations}}{Proposition~\ref*{prop:oracle-transformations}}}
\label{app:reformulations}
\begin{proof}
For part~(i), under the change of variables $x=a+sQy$, orthogonality gives $\norm{x-x_0}=s\norm{y-y_0}$. Thus the radius-$t$ ball centered at $y_0$ corresponds exactly to the radius-$st$ ball centered at $x_0=a+sQy_0$. Substitution preserves objective values and therefore identifies the corresponding exact ball-minimizer sets.
For part~(ii), strict monotonicity of $\chi$ gives $f(x)\le f(y)$ if and only if $\widetilde f(x)\le\widetilde f(y)$ for all $x,y\in\dom f$. Since the center of every ball call lies in the domain and has finite objective value, points outside the domain, where both objectives equal $+\infty$, cannot be minimizers. Hence $f$ and $\widetilde f$ have the same exact ball-minimizer sets at every center and radius. The same order equivalence shows that their global minimizer sets coincide. Therefore, under the same selections, induction gives identical fixed-radius trajectories and hence identical termination counts.
\end{proof}

\subsection{Proof of \texorpdfstring{\Cref{thm:epigraph-steps}}{Theorem~\ref*{thm:epigraph-steps}}}
\label{app:epigraph-steps}
We first record the subgradients of $F$ at a graph point. If $(v,\beta)\in\partial F(u,f(u))$, testing points vertically above $(u,f(u))$ gives $\lambda:=1-\beta\ge0$. Testing $(y,f(y))$ for $y\in\dom f$ then gives
\begin{equation}\label{eq:epigraph-subgradient-proof}
\ip v{y-u}\le\lambda(f(y)-f(u)).
\end{equation}
Conversely, these conditions imply the subgradient inequality at every point of the epigraph. If $\lambda>0$, \eqref{eq:epigraph-subgradient-proof} is equivalent to $v/\lambda\in\partial f(u)$. If $\lambda=0$, it means $v\in N_{\dom f}(u)$, where $\ip v{y-u}\le0$ for every $y\in\dom f$. No closedness of $\dom f$ is needed. At a point $(u,s)$ strictly above the graph, testing both $(u,f(u))$ and points vertically above $(u,s)$ forces $\beta=1$, and then $v\in N_{\dom f}(u)$.
\begin{proof}
The epigraph is nonempty, closed, and convex, so $F$ is proper, closed, and convex. Its infimum and minimizers follow by minimizing first over $s\ge f(x)$.
Consider a nonterminal output $(u,s)$ from $z=(x,f(x))$. By \Cref{lem:certificate} applied to $F$, there is $c>0$ such that

\[
(v,\beta)\overset{\eqref{eq:radial}}{=}c(x-u,f(x)-s)\in\partial F(u,s),\qquad \norm{z-(u,s)}\overset{\eqref{eq:radial}}{=}t.
\]

Since the step is nonterminal, its objective value is strictly smaller than that of the center, so $s<f(x)$. Suppose $s>f(u)$. The preceding subgradient characterization gives $v\in N_{\dom f}(u)$. Testing $x\in\dom f$ then gives $c\norm{x-u}^2\le0$, so $u=x$, contradicting $f(u)\le s<f(x)$. Hence every nonterminal output lies on the graph. Terminal outputs also lie on the graph because $\argmin F=\X\times\{\fs\}$. Induction therefore shows that every iterate remains on the graph.
Set $\delta=f(x)-f(u)>0$. At the graph output, the subgradient characterization gives $\lambda:=1-c\delta\ge0$. If $\lambda=0$, then $c(x-u)\in N_{\dom f}(u)$, and testing $x$ again forces $u=x$, contradicting $\delta>0$. Thus $0<\lambda<1$, and

\[
g=\frac{c}{\lambda}(x-u)\overset{\eqref{eq:epigraph-subgradient-proof}}{\in}\partial f(u),\qquad (v,\beta)=(\lambda g,1-\lambda).
\]

In particular, $r=\norm{x-u}>0$, while the full-radius identity gives $t^2=r^2+\delta^2$. If $\norm{y-x}\le r$, then $\ip{x-u}{y-u}\ge r^2-r\norm{y-x}\ge0$, and the subgradient inequality yields $f(y)\ge f(u)$. Hence $u\in\brox_f^r(x)$, proving \eqref{eq:epigraph-radius} and \eqref{eq:epigraph-certificate}.
Conversely, let $u$ be a nonterminal radius-$r$ \BPM{} output for $f$ from $x$. By \Cref{lem:certificate}, there is $a>0$ such that $g=a(x-u)\in\partial f(u)$, with $\norm{x-u}=r$ and $\delta=f(x)-f(u)>0$. Set $\lambda=(1+a\delta)^{-1}$ and $c=a(1+a\delta)^{-1}$. Then $(\lambda g,1-\lambda)=c(x-u,\delta)$ belongs to $\partial F(u,f(u))$ by \eqref{eq:epigraph-subgradient-proof}. Since this subgradient is radial from $(x,f(x))$ to $(u,f(u))$, it certifies an exact lifted step of radius $\sqrt{r^2+\delta^2}$.
Finally, consider a terminal lifted step of radius $t$. Its output has the form $(u,\fs)$ with $u\in\X$, and feasibility gives $\norm{x-u}^2+\Delta^2\le t^2$. Thus $u$ is an exact terminal output for $f$ at radius $\sqrt{t^2-\Delta^2}\ge D$. Conversely, any terminal radius-$r$ output $u\in\X$ for $f$ satisfies $\norm{x-u}\le r$, so $(u,\fs)$ is feasible for the lifted ball of radius $\sqrt{r^2+\Delta^2}$ and is therefore a terminal lifted output. This proves the terminal correspondences and completes the proof.
\end{proof}

\subsection{Proof of \texorpdfstring{\Cref{prop:epigraph-measures}}{Proposition~\ref*{prop:epigraph-measures}}}
\label{app:epigraph-measures}
\begin{proof}
Since $\argmin F=\X\times\{\fs\}$ and $F(x,f(x))=f(x)$, the gap and distance identities in \eqref{eq:epigraph-distance} follow immediately. Applying \Cref{thm:jensen,cor:squared} to $F$ with initial distance $R_0$ gives \eqref{eq:epigraph-value}.
Suppose first that $\partial f(x)\ne\varnothing$. By \eqref{eq:epigraph-subgradient-proof}, for fixed $\lambda>0$ the minimum squared norm of a lifted subgradient $(\lambda g,1-\lambda)$ is $\lambda^2m(x)^2+(1-\lambda)^2$. Minimizing over $\lambda\ge0$ gives
\[
\lambda=\frac{1}{1+m(x)^2},\qquad \min_{\lambda\ge0}\bigl[\lambda^2m(x)^2+(1-\lambda)^2\bigr]=\frac{m(x)^2}{1+m(x)^2}.
\]
The branch $\lambda=0$ has squared norm $\norm v^2+1\ge1$, while $(0,1)\in\partial F(x,f(x))$. If $\partial f(x)=\varnothing$, no branch with $\lambda>0$ is available, so the minimum subgradient norm of $F$ is one. This proves \eqref{eq:epigraph-stationarity}. Solving the resulting inequality for $m(x)$ gives the stated inverse estimate.
For the paired subgradients, $G=(\lambda g,1-\lambda)$ and $\mathcal G_{x^\star}(u,g)=\ip g{u-x^\star}$, so

\begin{equation}
\label{eq:lifted-pairing-proof}
\ip G{(u,f(u))-(x^\star,\fs)}\overset{\eqref{eq:epigraph-certificate},\eqref{eq:pd-gap}}{=}\lambda\mathcal G_{x^\star}(u,g)+(1-\lambda)(f(u)-\fs),
\end{equation}

which is \eqref{eq:epigraph-dual-gap}. Since the primal gap is $f(u)-\fs$, subtracting it from the lifted symmetric Bregman distance gives

\[
\begin{aligned}D_F^G((x^\star,\fs),(u,f(u)))&\overset{\eqref{eq:lifted-pairing-proof},\eqref{eq:epigraph-distance}}{=}\lambda\bigl(\mathcal G_{x^\star}(u,g)-(f(u)-\fs)\bigr)\\&\overset{\eqref{eq:pd-forward},\eqref{eq:pd-reverse}}{=}\lambda D_f^g(x^\star,u).\end{aligned}
\]

\end{proof}

\subsection{Proof of \texorpdfstring{\Cref{thm:epigraph-budgets}}{Theorem~\ref*{thm:epigraph-budgets}}}
\label{app:epigraph-budgets}
\begin{proof}
On an infinite nonterminal corresponding run, $r_k\le t_k=\sqrt{r_k^2+\delta_k^2}\le r_k+\delta_k$.
Thus $\sum_k t_k<\infty$ immediately implies $\sum_k r_k<\infty$.
Conversely, suppose $\sum_k r_k<\infty$.
Then $(x^k)$ converges and is contained in a compact set.
Since $f$ is proper and closed, it is bounded below on this set.
As the values are nonincreasing, they converge to a finite limit, and hence $\sum_k\delta_k=f(x^0)-\lim_k f(x^k)<\infty$.
The preceding inequality then gives $\sum_k t_k<\infty$.
Apply \Cref{thm:general-convergence} to the projected \BPM{} run with radii $r_k$.
If $\sum_k t_k=\infty$, an infinite nonterminal run has $\sum_k r_k=\infty$.
Under attainment this is impossible, so finite termination occurs; without attainment, $f(x^k)\downarrow\inf f$ and $\norm{x^k}\to\infty$.
If $\sum_k t_k<\infty$, then $\sum_k r_k<\infty$, and an infinite projected run satisfies $x^k\to\bar x\in\dom f$ and $f(x^k)\downarrow f(\bar x)$.
Finite termination gives the corresponding conclusions immediately.
Under attainment, \Cref{thm:dimension} gives $\sum_{k<K}r_k\le2C_dD_0$ for every sequence of nonterminal steps.
Summing $t_k\le r_k+\delta_k$ and using $\sum_{k<K}\delta_k=\Delta_0-\Delta_K$ gives \eqref{eq:epigraph-length-budget}; the strict inequality follows from $\Delta_K>0$.
Let $B=2C_dD_0+\Delta_0$.
For a constant lifted radius $t$, if the first $M=\lceil B/t\rceil$ iterations were all nonterminal, then \eqref{eq:epigraph-length-budget} would give $Mt<B$, contradicting $Mt\ge B$.
This gives the second bound in \eqref{eq:epigraph-count}, while the first follows from \eqref{eq:epigraph-value}.
Finally, under nearest-minimizer selection at termination for $F$, projection onto $\argmin F=\X\times\{\fs\}$ selects the nearest original minimizer in the horizontal coordinate.
By \Cref{cor:tail-length}, the total projected length is at most $2C_dD_0$.
The total vertical variation is at most $\Delta_0$, so the triangle inequality bounds the total lifted trajectory length by $2C_dD_0+\Delta_0$.
\end{proof}

\subsection{Verification of \texorpdfstring{\Cref{ex:epigraph-scaling}}{Example~\ref*{ex:epigraph-scaling}}}
\label{app:epigraph-scaling}
\begin{proof}
For $x>0$, a radius-$t$ BPM step gives $(x-t)_+$, so $N_t=\lceil D_0/t\rceil$.
For the lifted method, \Cref{thm:epigraph-steps} shows that every nonterminal output projects to $u=x-r>0$ for some $r>0$.
Since $f(x)-f(u)=ar$, \eqref{eq:epigraph-radius} gives $t^2=(1+a^2)r^2$, hence $r=t/\sqrt{1+a^2}$.
A lifted step is terminal exactly when its ball contains $(0,0)$, that is, when $t\ge\sqrt{1+a^2}\,x$.
Thus each nonterminal step advances a lifted distance $t$ along the segment from $(D_0,aD_0)$ to $(0,0)$, whose length is $\sqrt{1+a^2}\,D_0$, and therefore
\[
N_t^{\rm epi}=\left\lceil\frac{\sqrt{1+a^2}\,D_0}{t}\right\rceil.
\]
Since $N_t$ is independent of $a$, letting $a\to\infty$ gives the asserted unbounded separation.
\end{proof}

\subsection{Proof of \texorpdfstring{\Cref{thm:certificate-calibration}}{Theorem~\ref*{thm:certificate-calibration}}}
\label{app:general-calibration}
\begin{proof}
By monotonicity of $\partial f$, $\ip{h^k-g^k}{x^k-x^{k+1}}\ge0$.
Since $x^k-x^{k+1}=\lambda_{k+1}g^k$ with $\lambda_{k+1}>0$, this gives $\ip{h^k-g^k}{g^k}\ge0$, and hence $\norm{g^k}\le\norm{h^k}$.
Therefore

\begin{equation}
\label{eq:calibration-parameter-proof}
\lambda_{k+1}\overset{\eqref{eq:prox-equivalence}}{=}\frac{t_k}{\norm{g^k}}=\tau\frac{\norm{h^k}}{\norm{g^k}}\ge\tau.
\end{equation}

By \eqref{eq:prox-equivalence}, every nonterminal update satisfies $x^{k+1}=\prox_{\lambda_{k+1}f}(x^k)$.
The tight proximal-point estimate of \citet[Theorem~4.1]{TaylorHendrickxGlineur2017} therefore gives

\[
\Delta_K\le\frac{D_0^2}{4\sum\limits_{k=0}^{K-1}\lambda_{k+1}}\overset{\eqref{eq:calibration-parameter-proof}}{\le}\frac{D_0^2}{4\tau K}.
\]

After termination, $\Delta_K=0$ by \assref{ass:iteration}, so the final bound remains valid for the stopped sequence.

\emph{Sharpness.}
Fix $D>0$, $\tau>0$, and an integer $K\ge1$, and take $f(x)=a|x|$ on $\R$, with $x^0=D$ and $a=D/(2\tau K)$.
The minimizer is $0$, so $D_0=D$.
For $0\le k<K$, the unique subgradient at the positive center is $h^k=a$, and the prescribed radius is $t_k=\tau a=D/(2K)$.
Induction gives $x^k=D-kD/(2K)$ for $0\le k\le K$, so all $K$ iterations are nonterminal and $x^K=D/2$.
Their radial subgradients are $g^k=a$, hence
\[
\lambda_{k+1}\overset{\eqref{eq:prox-equivalence}}{=}\frac{t_k}{|g^k|}=\tau,
\qquad
\Delta_K=\frac{aD}{2}=\frac{D_0^2}{4\tau K}
=\frac{D_0^2}{4\sum\limits_{k=0}^{K-1}\lambda_{k+1}}.
\]
Thus both upper bounds are attained for every prescribed horizon $K$.
\end{proof}

\subsection{Proof of \texorpdfstring{\Cref{thm:gap-power}}{Theorem~\ref*{thm:gap-power}}}
\label{proof:calibration:1}
\begin{proof}
If termination occurs, the claim is immediate. Before termination, $t_k<D_k\le D_0$, so \Cref{thm:segment} gives

\begin{equation}
\label{eq:power-recurrence-proof}
\Delta_{k+1}\overset{\eqref{eq:segment}}{\le}\Delta_k\left(1-\frac{t_k}{D_k}\right)\le\Delta_k(1-a\Delta_k^\alpha),\qquad a=\tau/D_0,
\end{equation}

with $0<a\Delta_k^\alpha<1$. Since $(1-z)^{-\alpha}\ge1+\alpha z$ for $0\le z<1$, we obtain

\[
\Delta_{k+1}^{-\alpha}\overset{\eqref{eq:power-recurrence-proof}}{\ge}\Delta_k^{-\alpha}(1-a\Delta_k^\alpha)^{-\alpha}\ge\Delta_k^{-\alpha}+\alpha a.
\]

Summing over $k=0,\ldots,K-1$ gives $\Delta_K^{-\alpha}\ge\Delta_0^{-\alpha}+\alpha\tau K/D_0$, and inversion yields \eqref{eq:gap-power-rate}. Solving this bound for $\Delta_K\le\varepsilon$ gives the stated iteration count.
\end{proof}

\subsection{Proof of \texorpdfstring{\Cref{prop:gap-extensions}}{Proposition~\ref*{prop:gap-extensions}}}
\label{proof:calibration:2}
\begin{proof}
Since $\ell\le\fs$, we have $f(x^k)-\ell\ge\Delta_k$, and hence $t_k\ge\tau\Delta_k^\alpha$. 
If a step is terminal, the objective-gap bound is immediate. Otherwise, \Cref{thm:segment} gives the same upper recurrence as in the proof of \Cref{thm:gap-power}, so \eqref{eq:gap-power-rate} follows without comparing the two trajectories.
If $\ell<\fs$, then every nonoptimal iterate satisfies $t_k\ge\tau(\fs-\ell)^\alpha>0$. 
Thus, if the first $K$ iterations were all nonterminal, \Cref{cor:squared} would give $K\tau^2(\fs-\ell)^{2\alpha}<D_0^2$. 
Therefore, among the first
\[
\left\lceil\frac{D_0^2}{\tau^2(\fs-\ell)^{2\alpha}}\right\rceil
\]
iterations, at least one must be terminal.
\end{proof}

\subsection{Proof of \texorpdfstring{\Cref{thm:general-relaxed}}{Theorem~\ref*{thm:general-relaxed}}}
\label{app:general-relaxation}
\begin{proof}
Let $p=T_t(x)$ and $d=x-p$.
For every $z\in\X$, the radial subgradient in the nonterminal case and the projection property in the terminal case give
\begin{equation}\label{eq:relax-cutter-proof}
\ip d{p-z}\ge0.
\end{equation}
Since $R_{t,\lambda}(x)=x-\lambda d$ and $x-z=d+(p-z)$,

\[
\norm{R_{t,\lambda}(x)-z}^2\overset{\eqref{eq:relax-cutter-proof}}{\le}\norm{x-z}^2-\lambda(2-\lambda)\norm d^2.
\]

Because $\norm d=\min\{t,D(x)\}$, this is \eqref{eq:general-relaxed-fejer}.
Summing \eqref{eq:general-relaxed-fejer} along the iterates gives $\min\{t,D(x^k)\}\to0$, and hence $D(x^k)\to0$.
The sequence is Fej\'er monotone with respect to $\X$ and therefore bounded.
Any cluster point belongs to $\X$ because $D(x^k)\to0$ and $\X$ is closed.
Let $\bar x\in\X$ be such a cluster point.
Then $\norm{x^k-\bar x}$ is nonincreasing by \eqref{eq:general-relaxed-fejer}, while a subsequence converges to $\bar x$.
Thus $\norm{x^k-\bar x}\to0$, so the whole sequence converges to $\bar x$.
Finally, if $D(x)\le t$, then $p=\Proj_{\X}(x)$ and
\[
D(R_{t,\lambda}(x))\le\norm{R_{t,\lambda}(x)-p}=|1-\lambda|D(x),
\]
which proves the final estimate.
\end{proof}

\subsection{Proof of \texorpdfstring{\Cref{thm:identity}}{Result~\ref*{thm:identity}}}
\label{proof:foundation:2}
\begin{proof}
Expanding $x^k-z=(x^k-u)+(u-z)$ gives the cross term

\[
\begin{aligned}\ip{x^k-u}{u-z}&\overset{\eqref{eq:radial}}{=}c_k^{-1}\ip{g^k}{u-z}\\&\overset{\eqref{eq:bregman-definition}}{=}c_k^{-1}\bigl[f(u)-f(z)+D_f^{g^k}(z,u)\bigr].\end{aligned}
\]

which yields \eqref{eq:identity}. 
Evaluating the definition of the Bregman distance at $z=x^k$ and using
$\ip{g^k}{x^k-u}=c_kt_k^2=t_k\norm{g^k}$
gives \eqref{eq:drop}. 
Finally, if $f(z)\le f(u)$, then $f(u)-f(z)\ge0$ and $D_f^{g^k}(z,u)\ge0$, so \eqref{eq:lower-fejer} follows directly from \eqref{eq:identity}.
\end{proof}

\subsection{Proof of \texorpdfstring{\Cref{cor:squared}}{Result~\ref*{cor:squared}}}
\label{proof:foundation:3}
\begin{proof}
For a nonterminal step, take $z=\Proj_{\X}(x^k)$ in \eqref{eq:lower-fejer}. Since $z\in\X$, minimizing the resulting distance over $\X$ gives $D_{k+1}^2\le D_k^2-t_k^2$.
If the first $K$ iterations were all nonterminal, telescoping would yield
$0<D_K^2\le D_0^2-\sum_{k<K}t_k^2$. 
Hence $\sum_{k<K}t_k^2\ge D_0^2$ forces a terminal step among the first $K$ iterations. 
For $t_k=t$, taking $K=\lceil D_0^2/t^2\rceil$ gives \eqref{eq:squared-count}. 
Finally, before termination the same telescoping argument applies, while at and after termination $D_K=0$, which gives the stated bound for the stopped sequence.
\end{proof}

\subsection{Proof of \texorpdfstring{\Cref{prop:minimum-subgradient}}{Result~\ref*{prop:minimum-subgradient}}}
\label{proof:foundation:5}
\begin{proof}
Suppose first that $u$ is a nonterminal output, and let
$g=c(x-u)\in\partial f(u)$ be its radial subgradient. Since $\norm{x-u}=t$, we have $(x-u)/t=g/\norm g$. For any
$p\in\partial f(x)$, monotonicity of $\partial f$ gives $\ip{p-g}{x-u}\ge0$, and hence $\ip{p}{(x-u)/t}\ge\norm g\ge m(u)$.
By Cauchy--Schwarz, $\norm p\ge m(u)$. Taking the infimum over
$p\in\partial f(x)$ yields $m(x)\ge m(u)$.
If $u$ is terminal, then $u$ is a global minimizer, so
$0\in\partial f(u)$ and therefore $m(u)=0$. Thus $m(u)\le m(x)$
in all cases.
\end{proof}

\subsection{Proof of \texorpdfstring{\Cref{cor:geometric-stationarity}}{Result~\ref*{cor:geometric-stationarity}}}
\label{app:stationarity-transfer}
\begin{proof}
Suppose first that the $K$-th output is nonterminal, and let
$g^{K-1}\in\partial f(x^K)$ be its radial subgradient. 
By \eqref{eq:drop},

\begin{equation}
\label{eq:last-drop}
t_{K-1}\norm{g^{K-1}} \overset{\eqref{eq:drop}}{\le} \Delta_{K-1}-\Delta_K.
\end{equation}

Since $m(x^K)\le\norm{g^{K-1}}$ and $\Delta_K\ge0$, division by
$t_{K-1}>0$ gives

\[
m(x^K) \overset{\eqref{eq:last-drop}}{\le} \frac{\Delta_{K-1}-\Delta_K}{t_{K-1}} \le \frac{\Delta_{K-1}}{t_{K-1}},
\]

which is \eqref{eq:local-stationarity}.
If a terminal output has already occurred, then $0\in\partial f(x^K)$ and hence $m(x^K)=0$, so the same inequalities hold.
Now let $t_k=t$ with $0<t<D_0$. By \eqref{eq:frozen},

\[
\Delta_{K-1} \overset{\eqref{eq:frozen}}{\le} \left(1-\frac{t}{D_0}\right)^{K-1}\Delta_0.
\]

Substituting this into \eqref{eq:local-stationarity} gives
\eqref{eq:geometric-stationarity}. 
Finally, if $f$ is differentiable, $\partial f(x^K)=\{\nabla f(x^K)\}$, so $m(x^K)=\norm{\nabla f(x^K)}$.
\end{proof}

\section{Geometric radius schedules and the minimum successful factor}
\label{app:geometric-factor}
\label{sec:geometric-factor}
\Cref{thm:general-convergence,thm:criterion} leave the behavior of summable radius schedules problem dependent. 
As \Cref{ex:radius-outcomes} shows, even for the same objective and starting point, one summable schedule may attain a minimizer, another may converge to one only asymptotically, and another may converge to a nonoptimal point. 
We now examine how these possibilities arise within the one-parameter family
\[
 t_k=t\rho^k,\qquad 0<\rho<1.
\]
For fixed $f$, $x^0$, and $t$, only the decay factor $\rho$ varies. The question is whether there is a smallest value of $\rho$ for which the trajectory still converges to a minimizer.
Write
\[
 \mathcal T_k(\rho)=\sum\limits_{j=k}^\infty t\rho^j=\frac{t\rho^k}{1-\rho}
\]
for the total radius remaining after iteration $k$. 
Call $\rho\in(0,1)$ successful if the corresponding trajectory converges to a point in $\X$.
The theorem below shows that the successful factors have a minimum $\rho_\star$.
At this minimum successful factor, the trajectory converges to a minimizer without finite termination, and $\rho_\star$ can be characterized through finite prefixes.

\begin{theorem}[Minimum successful geometric factor]\label{thm:geometric-factor}
Under \assref{ass:convex}, \assref{ass:attained}, and \assref{ass:iteration}, fix $x^0\in\dom f$ and $t>0$, and let $D_k(\rho)=\dist(x^k(\rho),\X)$.
If $t\ge D_0$, every $\rho\in(0,1)$ attains a minimizer in at most one step, so there is no smallest positive successful factor.
Suppose $0<t<D_0$.
The first output $x^1$ is independent of $\rho$.
Set $d_1=\dist(x^1,\X)>0$ and
\[
a=\frac{d_1}{t+d_1},\qquad b=\frac{d_1}{\sqrt{t^2+d_1^2}}.
\]
Then:
\begin{enumerate}[label=(\roman*)]
\item The successful factors have a minimum $\rho_\star$, with
\begin{equation}\label{eq:geometric-bracket}
0<a\le\rho_\star<b<1.
\end{equation}
Every $\rho\in[b,1)$ attains a minimizer finitely, whereas every $\rho<\rho_\star$ converges to a nonoptimal limit.
\item The trajectory at $\rho_\star$ never attains a minimizer finitely, but converges to some $\bar x\in\X$. For every $k\ge0$,
\begin{equation}\label{eq:critical-geometric-rate}
\frac{t\rho_\star^k}{2C_d(1-\rho_\star)}
\le D_k(\rho_\star)\le\norm{x^k(\rho_\star)-\bar x}
\le\frac{t\rho_\star^k}{1-\rho_\star}.
\end{equation}
Consequently,
\[
D_k(\rho_\star)^{1/k}\to\rho_\star,\qquad
\norm{x^k(\rho_\star)-\bar x}^{1/k}\to\rho_\star.
\]
If $M=\norm{g^0}$ for any radial subgradient from the first step, then $\Delta_k\le Mt\rho_\star^k/(1-\rho_\star)$ for every $k\ge1$.
\item A factor $\rho$ is successful if and only if
\begin{equation}\label{eq:geometric-prefix-test}
D_k(\rho)\le\mathcal T_k(\rho)\qquad\text{for every }k\ge1.
\end{equation}
For $K\ge1$, define
\[
A_K=\{\rho\in[a,b]:D_j(\rho)\le\mathcal T_j(\rho),\ 1\le j\le K\}.
\]
Each $A_K$ is nonempty and compact, and
\begin{equation}\label{eq:geometric-prefix-minimum}
\rho_\star=\lim_{K\to\infty}\min A_K,\qquad
\min A_K\ \text{is nondecreasing in }K.
\end{equation}
\end{enumerate}
\end{theorem}

The proof is given in \Cref{app:geometric-factor-proofs}. 
The finite-prefix characterization is exact but not intended as an efficient computational procedure, since evaluating $A_K$ requires distances to $\X$ and a global scalar search over $\rho$. 
In particular, the theorem does not assert that the successful factors form an interval. 
If $\rho$ fails, then some finite prefix satisfies $D_k(\rho)>\mathcal T_k(\rho)$, certifying that the remaining radius budget is insufficient to reach any minimizer. 
At $\rho_\star$, every step is nonterminal, so the remaining trajectory length equals $\mathcal T_k(\rho_\star)$ and \eqref{eq:critical-geometric-rate} gives the asymptotic convergence rate.
For $f(x)=|x|$ with $x^0=D_0>t$, direct summation gives $\rho_\star=1-t/D_0=a$. 
In this example, any decrease below $\rho_\star$ destroys convergence to the minimizer.

The dependence on trajectory geometry can be seen by comparing two objectives with the same $D_0/t$.
For the quadratic example in \Cref{sec:four-views}, take $f(z)=\tfrac12(z_1^2+4z_2^2)$, $x^0=(2,5/4)$, and $t=\sqrt2$.
The first output is $x^1=(1,1/4)$, so $D_0=\sqrt{89}/4$ and $d_1=\sqrt{17}/4>D_0-t$.
Hence \eqref{eq:geometric-bracket} gives

\[
\rho_\star\overset{\eqref{eq:geometric-bracket}}{\ge}\frac{d_1}{t+d_1}>1-\frac{t}{D_0}.
\]

For the objective $f(z)=\norm z$ at the same starting point and radius, the trajectory is radial and $\rho_\star=1-t/D_0$.
Thus the two problems have the same initial distance and radius but different minimum successful factors.

\subsection{Proof of \texorpdfstring{\Cref{thm:geometric-factor}}{Theorem~\ref*{thm:geometric-factor}}}
\label{app:geometric-factor-proofs}
\begin{proof}
If $t\ge D_0$, the initial ball meets $\X$, so the first step is terminal for every $\rho\in(0,1)$. Hence assume $0<t<D_0$.
We may use the nearest-minimizer selection at termination, since it leaves all nonterminal iterates and the first terminal step unchanged. Write $x^k(\rho)$ for the resulting stopped trajectory and $z(\rho)=\lim_k x^k(\rho)$, which exists by \Cref{thm:arbitrary-radii}. Since each displacement is at most its radius,
\begin{equation}\label{eq:geometric-uniform-tail}
\norm{z(\rho)-x^k(\rho)}\le\mathcal T_k(\rho).
\end{equation}

\emph{Continuity in $\rho$.}
We first show continuity of a single selected step. Let $y_n\to y$, $r_n\to r>0$, and $f(y_n)\le f(x^0)$. Suppose first that $D(y)\ge r$. The limiting ball problem has a unique selected output $u$: the step is nonterminal if $D(y)>r$, while at $D(y)=r$ uniqueness follows from \Cref{lem:certificate}. Let $u_n$ denote the selected outputs for $(y_n,r_n)$. Since $\norm{u_n-y_n}\le r_n$, the sequence $(u_n)$ is bounded. For fixed $\epsilon\in(0,1)$, set $v_n=(1-\epsilon)u+\epsilon y_n$. Since $\norm{u-y}=r$, we have $\norm{v_n-y_n}\to(1-\epsilon)r<r$, so $v_n$ is feasible for all sufficiently large $n$. Hence
\[
f(u_n)\le f(v_n)\le(1-\epsilon)f(u)+\epsilon f(x^0).
\]
If a subsequence $u_{n_j}$ converges to $\bar u$, then $\bar u$ is feasible for the limiting ball and lower semicontinuity gives $f(\bar u)\le(1-\epsilon)f(u)+\epsilon f(x^0)$. Letting $\epsilon\downarrow0$ yields $f(\bar u)\le f(u)$. Since $u$ minimizes the limiting ball problem, $\bar u$ is also a minimizer, and uniqueness gives $\bar u=u$. Thus $u_n\to u$.
If $D(y)<r$, then $D(y_n)<r_n$ for all sufficiently large $n$, so the nearby steps are terminal and their selected outputs are $\Proj_{\X}(y_n)\to\Proj_{\X}(y)$. Thus the selected step is continuous in its center and radius. Induction now shows that $x^k(\rho)$ is continuous in $\rho$ for every fixed $k$.
On any compact interval $[c,d]\subset(0,1)$, \eqref{eq:geometric-uniform-tail} gives

\[
\norm{z(\rho)-x^k(\rho)}\overset{\eqref{eq:geometric-uniform-tail}}{\le}\frac{td^k}{1-d},
\]

uniformly in $\rho\in[c,d]$. Hence $z$ is continuous as a locally uniform limit of the continuous maps $x^k$. Since $\X$ is closed, the successful set $\{\rho\in(0,1):z(\rho)\in\X\}$ is relatively closed in $(0,1)$.

\emph{The set of factors yielding finite termination is open.}
Suppose $x^n(\rho)\in\X$. Since $D(x^n(\rho))=0<t\rho^n$, continuity of $x^n(\rho)$, of the distance to $\X$, and of $t\rho^n$ implies that, for every sufficiently nearby $\rho'$, $D(x^n(\rho'))<t(\rho')^n$. The next step is therefore terminal. Hence the set of factors that attain a minimizer finitely is open.

\emph{Bounds on the minimum successful factor.}
If $\rho$ is successful, then \eqref{eq:geometric-uniform-tail} at $k=1$ gives

\[
d_1\overset{\eqref{eq:geometric-uniform-tail}}{\le}\frac{t\rho}{1-\rho},
\]

and therefore $\rho\ge a$.
For $\rho\ge b$,
\[
\sum\limits_{k=1}^\infty t^2\rho^{2k}=\frac{t^2\rho^2}{1-\rho^2}\ge d_1^2.
\]
Suppose, to the contrary, that every step after $x^1$ is nonterminal. By \eqref{eq:refined-c}, the step from $x^1$ satisfies

\begin{equation}
\label{eq:critical-first-decrease}
D_2^2\overset{\eqref{eq:refined-c}}{\le} d_1^2-t^2\rho^2-\frac{2\Delta_2}{c_1},
\qquad \Delta_2>0.
\end{equation}

Each later nonterminal step decreases squared distance by at least its squared radius. Hence for every $N\ge2$,

\[
D_N^2\overset{\eqref{eq:critical-first-decrease},\eqref{eq:lower-fejer}}{\le} d_1^2-\sum\limits_{k=1}^{N-1}t^2\rho^{2k}-\frac{2\Delta_2}{c_1}.
\]

Since the infinite sum is at least $d_1^2$, the right-hand side is negative for all sufficiently large $N$, a contradiction. Thus every $\rho\in[b,1)$ attains a minimizer finitely.
The successful set is therefore nonempty, and its intersection with $[a,b]$ is compact. Let $\rho_\star$ be its minimum. Since the set of factors yielding finite termination is open and the iteration terminates finitely at $b$, $\rho_\star<b$. The same openness shows that $\rho_\star$ itself cannot attain finitely, since otherwise some smaller nearby factor would also be successful. Thus every step at $\rho_\star$ is nonterminal, while $z(\rho_\star)\in\X$. Every $\rho<\rho_\star$ is unsuccessful and therefore has a nonoptimal limit.

\emph{Rate at the minimum successful factor.}
Let $\bar x=z(\rho_\star)$. Since every step at $\rho_\star$ is nonterminal, each displacement has length $t\rho_\star^j$, so the remaining trajectory length from $x^k(\rho_\star)$ is exactly $\mathcal T_k(\rho_\star)$. By \Cref{cor:tail-length},

\[
\mathcal T_k(\rho_\star)\overset{\eqref{eq:tail-length-distance}}{\le}2C_dD_k(\rho_\star).
\]

Since $\bar x\in\X$, \eqref{eq:geometric-uniform-tail} also gives

\[
D_k(\rho_\star)\le\norm{x^k(\rho_\star)-\bar x}\overset{\eqref{eq:geometric-uniform-tail}}{\le}\mathcal T_k(\rho_\star).
\]

Together these inequalities give \eqref{eq:critical-geometric-rate}, and taking $k$th roots yields the two root-rate limits. The objective-gap estimate follows from \eqref{eq:radius-tail-error} with the optimal limit $\bar x$.

\emph{Finite-prefix characterization.}
If $\rho$ is successful, then $z(\rho)\in\X$, so \eqref{eq:geometric-uniform-tail} gives $D_k(\rho)\le\mathcal T_k(\rho)$ for every $k\ge1$. Conversely, if these inequalities hold for every $k$, then $\mathcal T_k(\rho)\to0$ implies $D_k(\rho)\to0$. Since $x^k(\rho)\to z(\rho)$ and $\X$ is closed, $z(\rho)\in\X$, so $\rho$ is successful.
For each $K$, continuity of the finite iterates makes $A_K$ compact, and $b\in A_K$, so $A_K$ is nonempty. Moreover, $A_{K+1}\subseteq A_K$. Hence $m_K:=\min A_K$ is nondecreasing and converges to some $m_\infty\le\rho_\star$. For every fixed $j$, $m_K\in A_j$ whenever $K\ge j$, so closedness of $A_j$ gives $m_\infty\in A_j$. Thus $m_\infty$ belongs to every $A_j$, and the preceding characterization shows that $m_\infty$ is successful. By minimality of $\rho_\star$, $m_\infty\ge\rho_\star$. Therefore $m_\infty=\rho_\star$, proving \eqref{eq:geometric-prefix-minimum}.
\end{proof}

\section{Nonmonotonicity of the fixed-radius termination count}
\label{app:radius-monotonicity}
The upper bounds on $N_t$ improve as the fixed radius $t$ increases, but this does not imply that the actual termination count is nonincreasing in $t$. 
Changing the radius also changes the trajectory. 
The following counterexample shows that a larger radius can produce a lower objective value after the first step while placing the next iterate farther from the solution set, resulting in an additional step.
\begin{proposition}[The termination count need not decrease with the radius]\label{prop:radius-nonmonotone}
There exist a finite, continuous, coercive convex function on $\R^2$ with a unique minimizer, a starting point $x^0$, and radii $0<s<t$ such that
\[
N_s=2,\qquad N_t=3.
\]
Moreover, the first radius-$t$ step attains a lower objective value than the first radius-$s$ step.
\end{proposition}

\begin{proof}
Take $x^0=(1,0)$ and
\begin{equation}\label{eq:radius-counterexample}
f(x,y)=\max\left\{
\frac{|x|}{100},\frac{|y|}{100},
x-\frac25y-\frac{37}{100},
\frac{49}{100}x-\frac6{25}y-\frac{1473}{10000}
\right\}.
\end{equation}
Since $f(x,y)\ge\max\{|x|,|y|\}/100$ and $f(0,0)=0$, the function is finite, continuous, coercive, and convex, with unique minimizer $(0,0)$.
Set
\[
s=\frac{\sqrt{29}}{10},\qquad t=\frac{\sqrt{2977}}{100},
\]
so that $0<s<t$.

For radius $s$, direct calculation gives the first output
$u=(1/2,1/5)$. 
Indeed, the third affine term in \eqref{eq:radius-counterexample} is uniquely active at $u$, with $f(u)=1/20$ and $\nabla f(u)=(1,-2/5)=2(x^0-u)$, while $\norm{x^0-u}=\norm u=s$. 
Thus \Cref{lem:certificate} verifies the first step, and the second radius-$s$ ball contains the minimizer. Hence $N_s=2$.

For radius $t$, the first output is similarly $v=(51/100,6/25)$. 
The fourth affine term is uniquely active at $v$, with $f(v)=9/200$ and $\nabla f(v)=(49/100,-6/25)=x^0-v$, and $\norm{x^0-v}=t$. 
Hence \Cref{lem:certificate} verifies the first radius-$t$ step. However, $\norm v^2=3177/10000>t^2$, so the second step is nonterminal. 
If $w$ is its output, \Cref{cor:squared} gives $\norm w^2\le\norm v^2-t^2=1/50<t^2$. Therefore the third ball contains the minimizer, and $N_t=3$.

Finally, $f(v)=9/200<1/20=f(u)$. 
Thus the larger radius gives a better first objective value but requires an additional step to attain the minimizer.
\end{proof}

For completeness, the second output of the radius-$t$ trajectory shown in \Cref{fig:radius-count} is
$w=\alpha(1,1)$ with $\alpha=(15-\sqrt{209})/40$.
\begin{figure}[tbp]
\centering
\includegraphics[width=\textwidth]{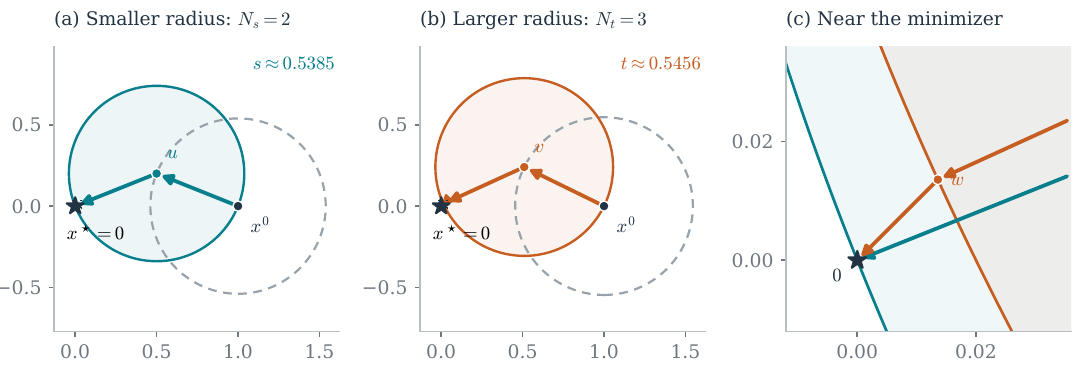}
\caption{Two fixed-radius \BPM{} trajectories from the same starting point. For radius $s$, the path is $x^0\to u\to0$; for radius $t$, it is $x^0\to v\to w\to0$. Dashed circles show the first balls and colored circles the second balls. Panels (a)--(b) use identical scales, while panel (c) magnifies the neighborhood of the minimizer: zero lies on the smaller second ball but outside the larger second ball. The third ball centered at $w$ is omitted.}
\label{fig:radius-count}
\end{figure}

\begin{table}[tbp]
\centering
\caption{Comparison of the first BPM steps in \Cref{fig:radius-count}. The larger radius gives a lower objective value but leaves its first output farther than one radius from the minimizer.}
\label{tab:radius-count-comparison}
\small
\renewcommand{\arraystretch}{1.3}
\begin{tabular}{@{}lll@{}}
\toprule
\textbf{Quantity} & \textbf{Smaller radius $s$} & \textbf{Larger radius $t$}\\
\midrule
Radius & $\sqrt{29}/10\approx0.5385$ & $\sqrt{2977}/100\approx0.5456$\\
First output & $u=(1/2,1/5)$ & $v=(51/100,6/25)$\\
Objective value & $f(u)=1/20=0.050$ & $f(v)=9/200=0.045$\\
Distance to the minimizer & $\norm u=s$ & $\norm v=\sqrt{3177}/100\approx0.5636>t$\\
Second ball contains the minimizer & Yes & No\\
Iterations to termination & $N_s=2$ & $N_t=3$\\
\bottomrule
\end{tabular}
\end{table}

\clearpage
\section{Key concepts}
\label{app:concepts}
Table~\ref{tab:concepts} collects the named concepts used in the paper, grouped by their role. The last column links to the first mention (including the abstract and overview table) and, where different, to the definition or precise formulation. The notation index is in \Cref{app:notation}.

\begingroup
\small
\setlength{\tabcolsep}{5pt}
\renewcommand{\arraystretch}{1.24}
\newcommand{\conceptname}[1]{\textbf{\color{TableInk}#1}}
\newcommand{\conceptfirst}[2]{\hyperref[#1]{#2}, p.~\pageref{#1}}
\newcommand{\conceptabstract}{\conceptfirst{concept:abstract}{Abstract}}
\begin{longtable}{@{}>{\raggedright\arraybackslash}p{.23\textwidth}>{\raggedright\arraybackslash}p{.49\textwidth}>{\raggedright\arraybackslash}p{.23\textwidth}@{}}
\caption{Key concepts: definitions and a guide to their introduction.}\label{tab:concepts}\\
\toprule
\textbf{Concept} & \textbf{Definition or meaning in this paper} & \textbf{First appearance; formulation}\\
\midrule
\endfirsthead
\multicolumn{3}{@{}l}{\textbf{Table \thetable. Key concepts (continued)}}\\[5pt]
\toprule
\textbf{Concept} & \textbf{Definition or meaning in this paper} & \textbf{First appearance; formulation}\\
\midrule
\endhead
\midrule\multicolumn{3}{r@{}}{\footnotesize Continued on the next page}\\
\endfoot
\bottomrule
\endlastfoot
\multicolumn{3}{@{}l}{\color{TableInk}\bfseries 1. Exact steps and equivalent formulations}\\*[5pt]
\conceptname{Ball-proximal point method (BPM)} & Each step minimizes $f$ over the closed Euclidean ball centered at the current iterate: $x^{k+1}\in\brox_f^{t_k}(x^k)$. The brox operator is the set of these minimizers. & \conceptabstract; update \textup{(\ref*{eq:bpm-intro})}.\\[4pt]
\conceptname{Exact ball oracle} & Given $x\in\dom f$ and $t>0$, returns a minimizer of $f$ on $\B(x,t)$. Each BPM iteration uses one oracle call. & \conceptfirst{sec:intro}{\S\ref*{sec:intro}}; definition \hyperref[concept:oracle]{after the update}.\\[4pt]
\conceptname{Terminal and nonterminal steps} & A step is terminal if its output belongs to $\X$, equivalently if its ball meets $\X$; otherwise it is nonterminal. A nonterminal output is unique and lies on the ball boundary. & \conceptfirst{concept:terminal}{\S\ref*{sec:intro}}; Lem.~\ref{lem:certificate}.\\[4pt]
\conceptname{Radial subgradient} & A subgradient $g\in\partial f(u)$ of the form $g=c(x-u)$, $c>0$, for a nonterminal step from $x$ to $u$. Its norm need not be the minimum subgradient norm. & \conceptfirst{prop:four-views}{Prop.~\ref*{prop:four-views}(ii)}; \eqref{eq:radial}.\\[4pt]
\conceptname{Proximal map} & $\prox_{\lambda f}(x)$ is the unique minimizer of $f(v)+\norm{v-x}^2/(2\lambda)$, $\lambda>0$. A nonterminal BPM step is proximal with $\lambda=t/\norm g$. & \conceptfirst{prop:four-views}{Prop.~\ref*{prop:four-views}(ii)}; \eqref{eq:view-implicit}.\\[4pt]
\conceptname{Ball envelope} & The value function $F_t(x)=\inf_{\norm{u-x}\le t}f(u)$. It is convex but need not be differentiable. A BPM step is a normalized explicit subgradient step on this envelope. & \conceptfirst{sec:four-views}{\S\ref*{sec:four-views}}; \eqref{eq:envelope-definition}.\\[4pt]
\conceptname{Sublevel-set projection} & A nonterminal output satisfies $u=\Proj_{\{v:f(v)\le f(u)\}}(x)$ and $\norm{x-u}=t$. The level is determined by the ball problem. & \conceptfirst{sec:four-views}{\S\ref*{sec:four-views}}; Prop.~\ref{prop:four-views}(iv).\\[4pt]
\conceptname{Fenchel-dual formulation} & The ball value equals $\max_g\{\ip xg-f^*(g)-t\norm g\}$, where $f^*(g)=\sup_v\{\ip gv-f(v)\}$ is the convex conjugate. & \conceptfirst{sec:four-views}{\S\ref*{sec:four-views}}; Prop.~\ref{prop:four-views}(v).\\[4pt]
\conceptname{Lagrangian dual formulation} & Maximize $q_t(\gamma)=\inf_v\{f(v)+\frac{\gamma}{2}(\norm{v-x}^2-t^2)\}$ over $\gamma\ge0$. A positive optimal multiplier yields the nonterminal proximal output. & \conceptfirst{sec:four-views}{\S\ref*{sec:four-views}} (scalar dual); Prop.~\ref{prop:four-views}(vi).\\
\newpage
\multicolumn{3}{@{}l}{\color{TableInk}\bfseries 2. Error measures and convergence bounds}\\*[5pt]
\conceptname{Objective gap} & $\Delta_k=f(x^k)-f_\star$ measures function-value error when the minimum is attained. & \conceptabstract; \hyperref[concept:errors]{notation in \S\ref*{sec:intro}}.\\[5pt]
\conceptname{Distance to the solution set} & $D_k=\dist(x^k,\X)$ is the distance to a nearest minimizer, rather than necessarily to a fixed reference minimizer. & \conceptabstract; \hyperref[concept:errors]{notation in \S\ref*{sec:intro}}.\\[5pt]
\conceptname{Minimum subgradient norm; stationarity} & $m(x)=\dist(0,\partial f(x))$, with $m(x)=+\infty$ if $\partial f(x)=\varnothing$. For convex $f$, $m(x)=0$ characterizes a minimizer. & \conceptabstract; Cor.~\ref{prop:minimum-subgradient}.\\[5pt]
\conceptname{Generalized Bregman distance} & $D_f^g(z,u)=f(z)-f(u)-\ip g{z-u}$ for a specified $g\in\partial f(u)$. It is nonnegative, depends on the chosen subgradient, and is not generally a metric. & \conceptfirst{prop:four-views}{Prop.~\ref*{prop:four-views}(v)} (for $f^*$); \eqref{eq:bregman-definition}.\\[5pt]
\conceptname{Symmetric Bregman distance to a minimizer} & $\mathcal G_{x^\star}(u,g)=D_f^g(x^\star,u)+D_f^0(u,x^\star)=\ip g{u-x^\star}$, using $g\in\partial f(u)$ and $0\in\partial f(x^\star)$. & \conceptabstract; Prop.~\ref{prop:primal-dual-gap}.\\[5pt]
\conceptname{Fenchel--Young gap} & $f(u)+f^*(g)-\ip ug$. It is zero whenever $g\in\partial f(u)$, so for radial subgradients it does not measure progress toward a minimizer. & \conceptfirst{concept:fenchel-young}{\S\ref*{sec:primal-dual}} (interpretation).\\[5pt]
\conceptname{Finite convergence; finite termination} & Some finite iterate lies in $\X$. Under the stopped-sequence convention, the method then terminates. This property does not supply an oracle test for detecting optimality. & \conceptabstract; \hyperref[concept:finite]{definition in \S\ref*{sec:intro}}.\\[5pt]
\conceptname{Segment contraction} & Comparing with a feasible point on the segment to a nearest minimizer gives $\Delta_{k+1}\le(1-t_k/D_k)_+\Delta_k$ when $D_k>0$. & \conceptfirst{sec:related-work}{\S\ref*{sec:related-work}}; Thm.~\ref{thm:segment}.\\[5pt]
\conceptname{Jensen objective-gap bound} & For $K$ nonterminal constant-radius steps and $S=D_0^2/t^2$, $\Delta_K/\Delta_0\le((S-K)/(S+K))^K$. Jensen's inequality combines the decreases in squared distance across steps. & \conceptfirst{concept:jensen}{\S3.2} (derivation); Thm.~\ref{thm:jensen}.\\
\newpage
\multicolumn{3}{@{}l}{\color{TableInk}\bfseries 3. Trajectory geometry, radius rules, and reformulations}\\*[5pt]
\conceptname{Self-contracted sequence} & $\norm{z^j-z^m}\le\norm{z^i-z^m}$ for $i\le j\le m$: the distance to any later point cannot increase along earlier iterates. & \conceptabstract; Def.~\ref{def:self-contracted}.\\[4pt]
\conceptname{Finite trajectory length} & The total displacement $\sum_{k=0}^\infty\norm{x^{k+1}-x^k}$ is finite. Bounded self-contracted sequences have this property in finite dimension. & \conceptabstract; Lem.~\ref{lem:length}.\\[4pt]
\conceptname{Summable and nonsummable radii} & Positive radii are summable if $\sum_{k=0}^\infty t_k<\infty$, and nonsummable otherwise. The latter condition guarantees finite termination whenever a minimizer exists. & \conceptabstract; Thm.~\ref{thm:general-convergence}.\\[4pt]
\conceptname{Nearest-minimizer selection at termination} & Choose $\Proj_{\X}(x)$ when the ball meets $\X$; otherwise use the unique nonterminal output. This defines the single-valued selection $T_t$. & \conceptfirst{tab:overview}{Table~\ref*{tab:overview}}; Def.~\ref{def:projected}.\\[4pt]
\conceptname{Epigraph reformulation} & Minimize $F(x,s)=s+\delta_{\operatorname{epi}f}(x,s)$, where $\operatorname{epi}f=\{(x,s):f(x)\le s\}$. Balls in the lifted space induce different radii in the original space. & \conceptabstract; \eqref{eq:epigraph-objective}.\\[4pt]
\conceptname{Subgradient-based radius rule} & Set $t_k=\tau\norm{h^k}$ with $h^k\in\partial f(x^k)$ and $\tau>0$; stop if $h^k=0$. An available radial subgradient from the preceding step may be reused. & \conceptabstract; Thm.~\ref{thm:certificate-calibration}.\\[4pt]
\conceptname{Objective-gap-based radius rule} & Set $t_k=\tau\Delta_k^\alpha$ with $\tau,\alpha>0$ and known $f_\star$. The paper also analyzes replacement of $f_\star$ by a valid lower bound. & \conceptabstract; Thm.~\ref{thm:gap-power}.\\[4pt]
\conceptname{Minimum successful geometric factor} & For $t_k=t\rho^k$, a factor is successful if the trajectory converges to a minimizer. When $0<t<D_0$, the successful factors have a minimum; they are not asserted to form an interval. & \conceptabstract; Thm.~\ref{thm:geometric-factor}.\\[4pt]
\conceptname{Relaxed BPM; cutter property} & The update is $R_{t,\lambda}(x)=x+\lambda(T_t(x)-x)$, $0<\lambda<2$. The cutter property is $\ip{x-T_t(x)}{z-T_t(x)}\le0$ for $z\in\X$. & Relaxation: \conceptabstract. Cutter: \S\ref{sec:related-work}; \S\ref{sec:relaxation}.\\[4pt]
\conceptname{Fej\'er monotonicity} & $\norm{x^{k+1}-z}\le\norm{x^k-z}$ for every $z\in\X$. Unlike self-contraction, the reference points belong to a fixed set. & \conceptfirst{sec:related-work}{\S\ref*{sec:related-work}} (Fej\'er-type decrease); \eqref{eq:general-relaxed-fejer}.\\
\end{longtable}
\endgroup

\end{document}